\documentclass[11pt]{article}

\usepackage[a4paper,margin=1in]{geometry}
\usepackage[numbers,sort&compress]{natbib}
\usepackage{microtype}
\usepackage{amsmath,amssymb,amsfonts,amsthm,mathrsfs,mathtools}
\usepackage{algorithm}
\usepackage{algpseudocode}
\usepackage{booktabs}
\usepackage{multirow}
\usepackage{array}
\usepackage{tabularx}
\usepackage{longtable}
\usepackage{graphicx}
\usepackage{subcaption}
\usepackage{xcolor}
\definecolor{DarkGreen}{RGB}{0,100,0}
\definecolor{DarkRed}{RGB}{139,0,0}
\definecolor{LightBlue}{RGB}{90,150,205}
\usepackage{enumitem}
\usepackage[section]{placeins}
\usepackage[hypertexnames=false]{hyperref}
\hypersetup{
  colorlinks=true,
  citecolor=DarkGreen,
  linkcolor=DarkRed,
  urlcolor=LightBlue,
  pdftitle={Composite and Stochastic Smoothed Decentralized Bregman Proximal-Gradient Methods},
  pdfauthor={Shengze Xu, Chenhan Jin, Qingsong Wang, Tieyong Zeng},
  pdfsubject={Decentralized nonconvex optimization with composite regularization, stochastic gradients, and Bregman geometry}
}

\graphicspath{{figures/}}
\allowdisplaybreaks[4]
\numberwithin{equation}{section}
\renewcommand{\arraystretch}{1.06}

\newtheorem{theorem}{Theorem}[section]
\newtheorem{assumption}[theorem]{Assumption}
\newtheorem{proposition}[theorem]{Proposition}
\newtheorem{lemma}[theorem]{Lemma}
\newtheorem{corollary}[theorem]{Corollary}
\theoremstyle{definition}
\newtheorem{definition}[theorem]{Definition}
\newtheorem{example}[theorem]{Example}
\theoremstyle{remark}
\newtheorem{remark}[theorem]{Remark}

\newcommand{\R}{{\mathbb R}}

\newcommand{\E}{{\mathbb E}}

\newcommand{\best}[1]{\textbf{#1}}

\renewcommand{\geq}{\geqslant}
\def\epsilon{\varepsilon}
\def\phi{\varphi}

\title{Composite and Stochastic Smoothed Decentralized Bregman Proximal-Gradient Methods}
\author{
Shengze Xu$^{1}$\thanks{These authors contributed equally to this work.} \quad
Chenhan Jin$^{1}$\footnotemark[1] \quad Qingsong Wang$^{2}$ \quad
Tieyong Zeng$^{3,4}$\thanks{Corresponding author: \texttt{tieyongzeng@bnbu.edu.cn}}\\[0.5em]
$^{1}$The Chinese University of Hong Kong \quad
$^{2}$Xiangtan University\\
$^{3}$Beijing Normal-Hong Kong Baptist University \quad
$^{4}$Guangzhou Nanfang College
}
\date{}

\begin{document}
\maketitle

\begin{abstract}
Smoothed decentralized Bregman proximal-gradient (SDBPG) methods combine relative-smoothness geometry with edge-dual consensus. We extend SDBPG to node-specific convex nonsmooth regularizers and unbiased stochastic local gradients while retaining its Bregman primal--dual smoothing mechanism. The analysis isolates temporal subgradient motion and aggregation mismatch as two defect sequences. For trajectories in a forward-invariant compact subset of the kernel domain and admissible parameters, finite accumulated defects yield an $O(T^{-1})$ bound on the squared best-iterate proximal residual over $T$ iterations. The stochastic bound separates accumulated oracle variance from expected defects and attains the same order when the variance bounds and expected defects are summable. Common active-structure identification, together with local subgradient conditions, gives sufficient finite-defect conditions for $\ell_1$ and group $\ell_{2,1}$ regularization. Experiments with common regularizers and heterogeneous sensing operators relate defect decay to observed support stabilization and compare reconstruction and communication with decentralized proximal and gradient-tracking methods. A separate phase-retrieval study evaluates the effect of a quartic Bregman kernel. Code and experiment scripts are available at \url{https://github.com/shengze-xu/composite-stochastic-sdbpg}.
\end{abstract}


\section{Introduction}
Decentralized optimization enables networked signal-processing, sensing, and learning systems to operate through neighbor communication without a fusion center \citep{tsitsiklis1986distributed,nedic2009distributed}. Avoiding centralized data aggregation improves scalability and data locality, but makes optimization accuracy depend on communication, graph connectivity, and heterogeneous local information.

Signal-processing models commonly exhibit several of these difficulties at once. Phase retrieval and low-rank recovery have nonconvex data-fidelity terms \citep{candes2015phase,tu2016low,chi2019nonconvex}; sparse and shared-support estimation use nonsmooth $\ell_1$ or group penalties \citep{tibshirani1996regression,donoho2006compressed,yuan2006model}; and partial-Fourier imaging combines structured complex-valued sensing with compressible representations. Local gradients may also be sampled or noisy. Euclidean proximal and gradient-tracking methods address some of these features, but do not preserve the problem-adapted Bregman geometry and edge-dual smoothing considered here.

Relative smoothness represents non-Euclidean curvature through a Bregman divergence instead of a global Euclidean Lipschitz-gradient bound \citep{bregman1967relaxation,beck2003mirror,bauschke2017descent,lu2018relatively}. The smooth deterministic SDBPG method of \citet{pu2024smoothed} couples this geometry to edge-based dual consensus through a moving smoothing reference. Its formulation contains no nonsmooth regularizer and uses exact local gradients, so it does not cover structural penalties or stochastic local information.

Kernel choice links the local optimization model to the curvature of the data-fidelity term. For example, quadratic measurements in phase retrieval produce a loss with quartic growth and unbounded Euclidean curvature. A kernel with comparable growth represents this geometry through the Bregman proximal step. Structural penalties then act within that step to promote sparse or shared-support solutions, while the edge-dual and smoothing updates coordinate the local variables through neighbor communication. This combination motivates a decentralized method that accommodates both problem-adapted geometry and nonsmooth structure.

The analysis of these extensions must account for how nonsmooth structure interacts with consensus. Temporal changes in the subgradients selected by the proximal updates affect the control of dual motion. Aggregating node subgradients at the network average introduces a separate compatibility issue: nearby node variables can lie on different active supports or faces. These effects couple to the motion of both primal and smoothing references, and stochastic gradients add oracle error to the same system.

We separate temporal subgradient motion and aggregation mismatch into two defect sequences and keep oracle variance explicit in the residual bounds. Common active-structure identification, together with local subgradient conditions, provides sufficient control of accumulated defects. The resulting analysis relates convergence to structural quantities that can be examined in sparse and group-sparse recovery. The main contributions are:
\begin{enumerate}[leftmargin=2em]
\item \emph{Composite and stochastic SDBPG algorithms.} The composite method places node-specific structural penalties inside the Bregman proximal step, and the stochastic method replaces exact local gradients by unbiased estimates. Both retain the edge-dual and smoothing updates of SDBPG.
\item \emph{Defect-based convergence analysis.} Temporal subgradient motion and aggregation mismatch are represented by separate defects. A Lyapunov analysis on a forward-invariant compact set gives deterministic and stochastic residual bounds. For trajectories with common active-structure identification, local subgradient conditions provide defect control for $\ell_1$, group $\ell_{2,1}$, polyhedral, and partly smooth regularizers. Under the corresponding admissible parameter conditions, summable defects give an $O(T^{-1})$ squared best-iterate proximal residual; summable expected defects and oracle variance give the same order in expectation for the stochastic method.
\item \emph{Signal-processing experiments with communication accounting.} With common coordinate- and group-sparse regularizers across nodes, the experiments relate observed active-structure stabilization to defect decay and evaluate recovery under heterogeneous local sensing operators. A separate study compares quadratic and matched non-Euclidean kernels. Partial-Fourier imaging and stochastic tests also compare reconstruction, consensus, and communication with decentralized proximal and gradient-tracking methods.
\end{enumerate}

\section{Related Work}
The closest lines of work concern decentralized composite optimization, stochastic gradient tracking, Bregman geometry, and active identification.

\paragraph{Decentralized deterministic optimization.}
Distributed optimization grew from consensus, asynchronous computation, and distributed subgradient methods \citep{tsitsiklis1986distributed,nedic2009distributed}. Convex developments include stochastic projection, dual averaging, time-varying digraph methods, and distributed alternating direction method of multipliers (ADMM) \citep{sundhar2010distributed,duchi2011dual,nedic2014distributed,neal2011distributed}; decentralized gradient descent relates objective smoothness to graph mixing \citep{yuan2016convergence}. EXTRA and gradient tracking remove its steady-state bias in smooth convex problems \citep{shi2015extra,nedic2017achieving,qu2017harnessing}, while proximal-gradient EXTRA (PG-EXTRA) admits convex proximal terms \citep{shi2015proximal}.
Unified distributed composite frameworks and inexact proximal-gradient schemes give topology-explicit rates, communication--computation tradeoffs, and network-independent stepsizes in convex settings \citep{xu2021distributed,guo2023decentralized}. Related developments include probabilistic multi-gossip communication, local adapt-then-combine updates, optimal strongly convex complexity, and adaptive three-operator splitting \citep{guo2025proximal,guo2026local,ye2025optimal,chen2026adaptive}. These analyses do not cover the nonconvex relatively smooth SDBPG model considered here.

\paragraph{Decentralized nonconvex and primal-dual methods.}
Nonconvex decentralized methods combine communication with local surrogates, time-varying networks, or Frank--Wolfe updates \citep{di2016next,scutari2018distributed,wai2017decentralized}. Primal--dual analyses characterize consensus error, dual motion, and graph spectra \citep{hong2017prox,sun2019distributed}, and stochastic variants also track sampling error \citep{yang2016parallel,lian2017can,koloskova2020unified}. Recent proximal gradient-tracking methods include ProxGT, which supports exact, mini-batch, and variance-reduced estimators with multiple consensus rounds \citep{xin2021stochastic}, and proximal decentralized averaged stochastic approximation with gradient tracking (Prox-DASA-GT), which uses constant batch sizes \citep{xiao2023one}. Compressed stochastic proximal-gradient methods address heterogeneous data and communication reduction, normal-map schemes provide another route to distributed composite optimization, and proximal gradient tracking has also been developed for composite objectives on Riemannian manifolds \citep{yan2023compressed,huang2024distributed,wang2025decentralized}. In the Bregman setting, \citet{pu2024smoothed} reported an $O(\epsilon^{-2})$ proximal-residual complexity for smooth deterministic SDBPG. The ICASSP paper does not include the complete convergence proof. The present work retains that geometry, treats node-specific regularizers and stochastic local gradients, and provides complete proofs of its smooth, composite, and stochastic convergence results in Appendices~\ref{app:deterministic-proof}--\ref{app:stochastic-proof}.

\paragraph{Bregman geometry.}
Bregman divergences underlie row-action, proximal-point, and mirror-descent methods \citep{bregman1967relaxation,chen1993convergence,eckstein1993nonlinear,kiwiel1997proximal,beck2003mirror}. Relative smoothness replaces a global Euclidean Lipschitz-gradient condition with a Bregman model inequality \citep{bauschke2017descent,lu2018relatively}; distributed relatively smooth and delay-tolerant Bregman proximal methods extend this geometry to networked or asynchronous computation \citep{jegnell2022distributed,chraibi2025delay}. Interior-gradient and proximal methods address constrained domains \citep{auslender2006interior}, while structured nonconvex analyses combine descent with limiting subdifferentials \citep{attouch2011convergence,bolte2014proximal,bolte2018first}. Approximate Bregman proximal-gradient methods address nonconvex relatively smooth subproblems \citep{takahashi2025approximate}, and stochastic Bregman analyses cover mini-batch, momentum, and variance-reduced estimators \citep{xiao2021unified,ding2025nonconvex,liao2025momentum}. Kernel-conditioned mappings also connect Bregman residuals to subdifferential stationarity \citep{zhang2025stochastic}. The present analysis uses local curvature bounds on a forward-invariant compact subset of the kernel domain.

\paragraph{Composite regularization and active structure.}
Composite methods keep nonsmooth convex terms inside the proximal step \citep{rockafellar1998variational,parikh2014proximal,combettes2011proximal}. In structured recovery, $\ell_1$ and group penalties encode sparsity and shared support \citep{tibshirani1996regression,donoho2006compressed,yuan2006model}, with proximal-gradient and separable approximations providing standard templates \citep{beck2009fast,wright2009sparse}. Active-set and partial-smoothness theory characterizes eventual motion on an identified face or manifold and the resulting local convergence behavior \citep{lewis2002active,liang2017activity}. Stochastic identification has been established for variance-reduced proximal methods, stabilized support-identification schemes, and normal-map proximal stochastic gradients \citep{poon2018local,dai2023variance,qiu2026normal}. In our analysis, identification is used to prove finite defect accumulation, avoiding a global Lipschitz assumption on subgradient variation.

\paragraph{Stochastic approximation and recovery models.}
Stochastic approximation incorporates sampling noise into convergence guarantees \citep{robbins1951stochastic,nemirovski2009robust}. Centralized nonconvex bounds track variance, weak convexity, or composite structure \citep{ghadimi2013stochastic,lan2012optimal,davis2019stochastic,bottou2018optimization}, whereas multi-agent analyses separate sampling error from descent and consensus \citep{yang2016parallel,lian2017can,koloskova2020unified}. In our bounds, oracle variance and nonsmooth structural changes appear as separate terms. The experiments use sparse recovery, phase retrieval, low-rank recovery, and structured Fourier imaging \citep{candes2015phase,eldar2014phase,candes2012exact,recht2010guaranteed,tu2016low,chi2019nonconvex} to evaluate reconstruction quality and decentralized cost.

Table~\ref{tab:method-scope-comparison} summarizes the closest algorithmic reference points.
\begin{table}[!htbp]
\centering
\caption{Comparison with closely related decentralized proximal methods.}
\label{tab:method-scope-comparison}

\small
\setlength{\tabcolsep}{3.5pt}
\renewcommand{\arraystretch}{1.12}

\begin{tabularx}{\textwidth}{@{}>{\raggedright\arraybackslash}p{0.15\textwidth}>{\raggedright\arraybackslash}X>{\raggedright\arraybackslash}p{0.16\textwidth}>{\raggedright\arraybackslash}p{0.18\textwidth}>{\raggedright\arraybackslash}p{0.175\textwidth}@{}}
\toprule
\textbf{Method} & \textbf{Objective and regularizer} & \textbf{Geometry and updates} & \textbf{Gradient oracle} & \textbf{Communication}\newline \textbf{per iteration}\\
\midrule
\textbf{Composite and stochastic SDBPG}\newline (this work)
& Possibly nonconvex; two-sided relative smoothness; compact localization.\newline Node-specific convex nonsmooth terms.
& Bregman; edge-dual smoothing
& Exact or unbiased gradients; explicit variance term
& One neighbor round; two states\\
\midrule
SDBPG\newline \citep{pu2024smoothed}
& Possibly nonconvex; relative smoothness.\newline No regularizer.
& Bregman; edge-dual smoothing
& Exact gradients
& One neighbor round; two states\\
\addlinespace[5pt]
ProxGT\newline \citep{xin2021stochastic}
& Possibly nonconvex; globally Lipschitz gradients.\newline Common convex nonsmooth term.
& Euclidean; gradient tracking
& Exact, mini-batch, variance-reduced
& Two consensus blocks\\
\addlinespace[5pt]
Prox-DASA-GT\newline \citep{xiao2023one}
& Stochastic nonconvex; Lipschitz gradients.\newline Common convex nonsmooth term.
& Euclidean; proximal averaging and tracking
& Unbiased bounded-variance gradients
& Multiple mixing rounds; three states\\
\addlinespace[5pt]
MG-Skip\newline \citep{guo2025proximal}
& Smooth strongly convex local losses.\newline Common convex nonsmooth term.
& Euclidean primal--dual splitting
& Exact gradients
& Randomized communication; probabilistic multi-gossip\\
\bottomrule
\end{tabularx}
\end{table}

\section{Problem Formulation and Composite/Stochastic SDBPG}
We formulate the decentralized composite problem and state the smooth, composite, and stochastic SDBPG updates. Each application is specified by local fidelity functions $f_i$, structural terms $r_i$, and a convex kernel $\phi$.

\subsection{Consensus Model and Bregman Geometry}
Consider $N$ agents communicating over a fixed, connected, undirected graph $G=(\mathcal V,\mathcal E)$, where $|\mathcal V|=N$. Agent $i$ holds a differentiable, possibly nonconvex function $f_i:\R^d\to\R$ and communicates with its neighbor set $\mathcal N_i$. The agents share a closed convex parameter domain $\mathcal X\subseteq\R^d$ with nonempty interior $\mathcal C:=\operatorname{int}\mathcal X$. Unless stated otherwise, vector and product-space norms are Euclidean, and matrix norms are Frobenius norms.

The decentralized problem couples the per-node objectives through exact consensus,
\begin{equation}\label{eq:problem-main}
\min_{x_1,\ldots,x_N}\ \frac1N\sum_{i=1}^N f_i(x_i)
\quad\text{s.t.}\quad x_i=x_j,\ (i,j)\in\mathcal E,
\qquad x_i\in\mathcal X,\ i\in\mathcal V.
\end{equation}
The edge constraints impose a common primal variable. The choice $\mathcal X=\R^d$ gives the unconstrained model; for a proper domain, the analysis remains within a compact subset of $\mathcal C$.

Structural priors are added through a composite term $r_i$ at each node:
\begin{equation}\label{eq:problem-composite}
\min_{x_1,\ldots,x_N}\ \frac1N\sum_{i=1}^N \big(f_i(x_i)+r_i(x_i)\big)
\quad\text{s.t.}\quad x_i=x_j,\ (i,j)\in\mathcal E,\quad x_i\in\mathcal X,
\end{equation}
where each $r_i:\R^d\to(-\infty,+\infty]$ is proper, closed, and convex; $r_i\equiv0$ recovers the smooth objective. Thus \eqref{eq:problem-composite} allows a nonconvex smooth fidelity term and a convex nonsmooth regularizer. The kernel $\phi$ specifies the geometry, while $r_i$ encodes sparsity, shared support, or another structural prior. We write $[N]=\{1,\ldots,N\}$, $d_i=|\mathcal N_i|$, $d_{\min}=\min_i d_i$, $d_{\max}=\max_i d_i$, and $\bar d=N^{-1}\sum_i d_i$.

\begin{definition}[Kernel and Bregman divergence]\label{def:bregman-kernel}
The kernel $\phi:\R^d\to(-\infty,+\infty]$ is proper, closed, convex, finite on $\mathcal X$, and continuously differentiable on $\mathcal C$. For $x\in\mathcal X$ and $y\in\mathcal C$, the induced Bregman divergence is
\[
D_\phi(x,y)=\phi(x)-\phi(y)-\langle\nabla\phi(y),x-y\rangle\ge0.
\]
\end{definition}

\begin{definition}[Relative smoothness and relative weak convexity]\label{def:relative-smoothness}
Let $Q\subseteq\mathcal C$ be convex. A differentiable function $f$ is $L_+$-smooth relative to $\phi$ on $Q$ if
\[
f(x)-f(y)-\langle \nabla f(y),x-y\rangle
\le L_+D_\phi(x,y),\qquad x,y\in Q,
\]
and $L_-$-weakly convex relative to $\phi$ on $Q$ if the same quantity is bounded below by $-L_-D_\phi(x,y)$. Equivalently, $L_+\phi-f$ and $L_-\phi+f$ are convex on $Q$. The upper bound drives descent. Together, the two bounds give local gradient-variation estimates; the lower bound also supplies the curvature needed for the consensus Bregman proximal-point residual.
\end{definition}

\subsection{Smooth Deterministic SDBPG}
Following \citet{pu2024smoothed}, smooth SDBPG combines a Bregman local model with edge multipliers. Each agent maintains a primal variable $x_i$, a smoothing reference $z_i$, and skew-symmetric edge variables $y_{ij}=-y_{ji}$. For $Z_i=\{z_j:j\in\mathcal N_i\cup\{i\}\}$, local-reference coefficient $c_0>0$, and neighbor-reference coefficient $\rho\ge0$, define
\begin{equation}\label{eq:Fi-main}
F_i(x_i,Z_i)=f_i(x_i)+c_0D_\phi(x_i,z_i)+\rho\sum_{j\in\mathcal N_i}D_\phi(x_i,z_j).
\end{equation}
The associated local primal-dual model is
\begin{equation}\label{eq:Ki-main}
K_i(x_i;Y_i,Z_i)=F_i(x_i,Z_i)+\sum_{j\in\mathcal N_i} y_{ij}^\top x_i,
\end{equation}
where $Y_i=\{y_{ij}:j\in\mathcal N_i\}$. The linear term is the local edge-consensus Lagrangian contribution, and the Bregman terms stabilize the moving model.
Throughout, $\nabla K_i$ denotes the gradient of $K_i$ with respect to its first argument.

\subsection{Composite and Stochastic Extensions}
Algorithm~\ref{alg:sdbpg} defines composite SDBPG. The variables $x_i^t$, $z_i^t$, and $y_{ij}^t$ are the primal, smoothing-reference, and edge-dual iterates; $\alpha$, $\lambda$, and $\beta$ denote the dual stepsize, primal stepsize, and smoothing weight. The effective domain of the primal subproblem is determined jointly by $D_\phi(\cdot,x_i^t)$ and $r_i$.

\begin{algorithm}[H]
\caption{Composite Kernel-Domain Smoothed Decentralized Bregman Proximal Gradient}
\label{alg:sdbpg}
\begin{algorithmic}[1]
\Require Graph $G$; local data $(f_i,r_i)$; kernel $\phi$; initial $x_i^0=z_i^0\in\mathcal C$ and $y_{ij}^0=0$; parameters $\alpha,\lambda,\beta,c_0>0$ and $\rho\ge0$.
\For{$t=0,1,2,\ldots$}
  \State Each agent exchanges $(x_i^t,z_i^t)$ with its neighbors.
  \For{each edge $(i,j)\in\mathcal E$}
    \State $y_{ij}^{t+1}=y_{ij}^t+\alpha(x_i^t-x_j^t)$ and $y_{ji}^{t+1}=-y_{ij}^{t+1}$.
  \EndFor
  \For{each agent $i\in\mathcal V$}
    \State Set $\nabla K_i^t:=\nabla_{x_i}K_i(x_i^t;Y_i^{t+1},Z_i^t)$.
    \State $x_i^{t+1}\in\arg\min_{x\in\R^d}\left\{r_i(x)+\langle \nabla K_i^t,x-x_i^t\rangle+\frac1\lambda D_\phi(x,x_i^t)\right\}$.
    \State $z_i^{t+1}=\arg\min_{z\in\R^d}\left\{\frac12\|z-x_i^{t+1}\|^2+\beta D_\phi(z,z_i^t)\right\}$.
  \EndFor
\EndFor
\end{algorithmic}
\end{algorithm}

The subproblems are solved exactly, with any minimizer admissible when the primal solution is nonunique. Setting $r_i\equiv0$ recovers smooth deterministic SDBPG. In the composite method, primal optimality selects a subgradient in $\partial r_i(x_i^{t+1})$, where $\partial r_i$ is the convex subdifferential. Stochastic composite SDBPG replaces $\nabla K_i^t$ by an unbiased estimator $G_i^t$ and keeps the remaining updates unchanged. A convex constraint set $Q$ can be included through its indicator $\iota_Q$, yielding a Bregman projected-gradient step and, for a quadratic kernel, its Euclidean counterpart.

\section{Assumptions, Defects, and Local Verification}
The convergence proof uses three ingredients. Localization provides uniform Bregman constants. Two defects account for nonsmooth structural changes. Conditions on an identified active structure then give finite accumulated defects.

Fix an orientation of $G$ and let $\mathcal B:\R^{|\mathcal E|d}\to\R^{Nd}$ be the block incidence operator, so that $\mathcal B^\top x$ stacks the oriented edge differences. Write $x^t=(x_1^t,\ldots,x_N^t)\in\R^{Nd}$ and $z^t=(z_1^t,\ldots,z_N^t)\in\R^{Nd}$, let $y^t\in\R^{|\mathcal E|d}$ stack the edge multipliers in the chosen orientation, and set $s^t:=\mathcal B^\top x^t$. Define
\[
\bar x^t=\frac1N\sum_{i=1}^N x_i^t,
\qquad
S^t=\|s^t\|^2,
\]
\[
d_x^t=\sum_{i=1}^N D_\phi(x_i^{t+1},x_i^t),\qquad
d_z^t=\sum_{i=1}^N D_\phi(z_i^{t+1},z_i^t),\qquad
e^t=\sum_{i=1}^N\|x_i^{t+1}-z_i^{t+1}\|^2.
\]
Here $S^t$ is the squared consensus violation, $d_x^t$ and $d_z^t$ quantify primal and reference motion, and $e^t$ is the squared smoothing mismatch. These four quantities form the descent terms in the Lyapunov estimates.

\subsection{Localized Geometry and Composite Structure}

\begin{assumption}[Geometry and compact localization]\label{ass:standing-local}
The problem data and the generated trajectory satisfy the following.
\begin{enumerate}[label=(\alph*),leftmargin=2em]
\item \emph{Convex domain and well-posed subproblems.} The common parameter set $\mathcal X$ is closed and convex with nonempty interior $\mathcal C=\operatorname{int}\mathcal X$. The primal, smoothing, and consensus proximal subproblems used below are well posed with selected solutions in $\mathcal C$, in both the smooth and composite settings.
\item \emph{Local kernel regularity.} The kernel in Definition~\ref{def:bregman-kernel} is twice continuously differentiable ($C^2$) and strongly convex on an open convex neighborhood $U\subseteq\mathcal C$ of the localization set $K$ in part (d).
\item \emph{Local two-sided relative smoothness.} Each $f_i$ is continuously differentiable ($C^1$) on $\mathcal C$. There exist $L_+,L_-<\infty$ such that $L_+\phi-f_i$ and $L_-\phi+f_i$ are convex on $U$. The symmetric case is $L_+=L_-=L$.
\item \emph{Admissible compact localization.} There is a compact convex set $K\Subset\mathcal C$, where $\Subset$ denotes compact containment, and $x_i^0=z_i^0\in K$. The set is forward invariant: if the current primal and smoothing references lie in $K$, then the selected primal updates, smoothing updates, and the common-variable proximal comparison points in \eqref{eq:stationarity-measure-main} also lie in $K$ for the generated edge multipliers. All local Bregman and model-gradient constants are measured on $K$.
\end{enumerate}
\end{assumption}

\begin{remark}[Role of compact localization]\label{rem:compact-localization}
Compact localization provides uniform curvature bounds and a finite Bregman diameter, following the use of regularity on bounded sets in Bregman analysis \citep{bolte2018first,takahashi2025approximate}. The curvature conditions are imposed on a neighborhood of $K$, and continuity gives a lower bound for each $f_i$ on $K$. The quadratic and quadratic-plus-quartic kernels used here have uniformly positive and bounded Hessians on compact sets, while $K\Subset\mathcal C$ provides an interior margin for proper kernel domains. In sparse and shared-support least squares, positive $\ell_1$ or group $\ell_{2,1}$ penalties covering all coordinates make the composite objectives coercive, yielding compact sublevel sets even for rank-deficient sensing operators. Forward invariance additionally concerns the evolving dual terms, moving Bregman references, and consensus proximal comparison points; deriving it from the problem data entails joint control of these coupled updates.
\end{remark}

\begin{lemma}[Local constants on $K$]\label{lem:local-variation-from-relative}
Under Assumption~\ref{ass:standing-local}, the generated iterates and the consensus proximal comparison points lie in the compact convex set $K\Subset\mathcal C$, and there exist constants $\mu,L_\phi,L_{F,K}<\infty$ such that
\[
D_\phi(u,v)\ge\frac\mu2\|u-v\|^2,
\qquad
\|\nabla\phi(u)-\nabla\phi(v)\|\le L_\phi\|u-v\|,
\qquad u,v\in K,
\]
and, for reference collections $U_i=\{u_j^z:j\in\mathcal N_i\cup\{i\}\}$ and $V_i=\{v_j^z:j\in\mathcal N_i\cup\{i\}\}$ with all entries in $K$,
\[
\begin{aligned}
&\sum_{i=1}^N\|\nabla f_i(u_i)-\nabla f_i(v_i)\|^2
+\sum_{i=1}^N\|\nabla_xF_i(u_i,U_i)-\nabla_xF_i(v_i,V_i)\|^2\\
&\qquad\le
L_{F,K}\left(\sum_{i=1}^N\|u_i-v_i\|^2+
\sum_{i=1}^N\sum_{j\in\mathcal N_i\cup\{i\}}\|u_j^z-v_j^z\|^2\right).
\end{aligned}
\]
\end{lemma}

\begin{assumption}[Composite regularizers]\label{ass:composite-regularizer}
For the composite model, each $r_i$ is proper, closed, convex, and finite at the iterates and comparison points in $K$. The subdifferential sum rules and first-order optimality relations used in Definition~\ref{def:subgradient-motion-defects} and Proposition~\ref{prop:kkt-bridge} hold at all selected primal updates and consensus comparison points.
\end{assumption}

Properness and closedness give a lower bound for each $r_i$ on the compact set $K$. The sum-rule qualification holds when $r_i$ is finite and continuous near the relevant part of $K$, including the $\ell_1$, squared $\ell_2$, and group $\ell_{2,1}$ penalties. Existence and localization of the selected subproblem solutions are specified in Assumption~\ref{ass:standing-local}(a),(d).

\subsection{Nonsmooth Defects and Finite Accumulation}

For $a\in\R$, let $[a]_+:=\max\{a,0\}$. The motion defect is the selected-subgradient variation left after the controllable primal, reference, and oracle terms are absorbed. Subscripts $T\xi$ and $C\xi$ refer to temporal-motion and consensus-aggregation absorption constants; the superscript $s$ denotes stochastic quantities.

\begin{definition}[Composite subgradient selections and motion defects]\phantomsection\label{def:subgradient-motion-defects}
\begin{enumerate}[label=(\alph*),leftmargin=2em]
\item \emph{Smooth case.} If $r_i\equiv0$, select $\xi_i^t=0$ for all $i,t$.
\item \emph{Deterministic composite case.} Choose $\xi_i^{t+1}\in\partial r_i(x_i^{t+1})$ satisfying
\[
0=
\nabla K_i(x_i^t;Y_i^{t+1},Z_i^t)
+\frac1\lambda\big(\nabla\phi(x_i^{t+1})-\nabla\phi(x_i^t)\big)
+\xi_i^{t+1}.
\]
For any fixed absorption constant $M_{T\xi}\ge0$, define the deterministic subgradient-motion defect by
\[
\begin{aligned}
\eta_t^r:=
\Big[
&\sum_{i=1}^N\|\xi_i^{t+1}-\xi_i^t\|^2\\
&-M_{T\xi}\big(d_x^t+d_x^{t-1}+d_z^{t-1}\big)
\Big]_+,\qquad t\ge1.
\end{aligned}
\]
\item \emph{Stochastic composite case.} Let $\E$ denote expectation and define the stochastic oracle error $\delta_i^t=G_i^t-\nabla K_i(x_i^t;Y_i^{t+1},Z_i^t)$, with $\delta^t=(\delta_1^t,\ldots,\delta_N^t)$. Choose $\xi_i^{t+1}\in\partial r_i(x_i^{t+1})$ satisfying
\[
0=
G_i^t
+\frac1\lambda\big(\nabla\phi(x_i^{t+1})-\nabla\phi(x_i^t)\big)
+\xi_i^{t+1}.
\]
For fixed absorption constants $M_{T\xi}^s,M_{\delta\xi}\ge0$, define the expected stochastic subgradient-motion defect by
\[
\begin{aligned}
\eta_t^s:=
\Big[
&\E\sum_{i=1}^N\|\xi_i^{t+1}-\xi_i^t\|^2\\
&-M_{T\xi}^s\E\big(d_x^t+d_x^{t-1}+d_z^{t-1}\big)\\
&-M_{\delta\xi}\E\big(\|\delta^t\|^2+\|\delta^{t-1}\|^2\big)
\Big]_+,\qquad t\ge1.
\end{aligned}
\]
\end{enumerate}
The defect sum records the part of the selected-subgradient variation that exceeds the motion and oracle-error terms.
\end{definition}

Choosing $M_{\delta\xi}=0$ leaves all oracle error outside the motion defect; otherwise the absorbed part is controlled by the variance term that remains explicit in the stochastic bounds.

\begin{definition}[Aggregation defects]\label{def:aggregation-defects}
For a vector $v$ and set $A$, write $\operatorname{dist}(v,A):=\inf_{a\in A}\|v-a\|$. The sum $\sum_i\partial r_i(\bar x^t)$ is a Minkowski sum: it contains all $\sum_i v_i$ with $v_i\in\partial r_i(\bar x^t)$. The aggregation defect measures the selected-subgradient mismatch at the network average that remains after consensus has been absorbed.
\begin{enumerate}[label=(\alph*),leftmargin=2em]
\item \emph{Deterministic aggregation.} For $t\ge1$, with selected $\xi_i^t$ from Definition~\ref{def:subgradient-motion-defects}(b) and any fixed $M_{C\xi}\ge0$, define
\[
\zeta_t^r:=\left[\operatorname{dist}^2\!\left(\sum_{i=1}^N\xi_i^t,\sum_{i=1}^N\partial r_i(\bar x^t)\right)-M_{C\xi}S^t\right]_+\ \ge0 .
\]
Then, by construction (using $a\le[a-b]_++b$ for $b\ge0$),
\[
\operatorname{dist}^2\!\left(\sum_{i=1}^N\xi_i^t,\sum_{i=1}^N\partial r_i(\bar x^t)\right)
\le M_{C\xi}S^t+\zeta_t^r.
\]
\item \emph{Stochastic aggregation.} For $t\ge1$, with the selected stochastic subgradients from Definition~\ref{def:subgradient-motion-defects}(c) and any fixed $M_{C\xi}^s\ge0$, define
\[
\zeta_t^s:=\left[\E\operatorname{dist}^2\!\left(\sum_{i=1}^N\xi_i^t,\sum_{i=1}^N\partial r_i(\bar x^t)\right)-M_{C\xi}^s\,\E S^t\right]_+\ \ge0 ,
\]
so that
\[
\E\operatorname{dist}^2\!\left(\sum_{i=1}^N\xi_i^t,\sum_{i=1}^N\partial r_i(\bar x^t)\right)
\le M_{C\xi}^s\E S^t+\zeta_t^s.
\]
\end{enumerate}
\end{definition}

The aggregation defect appears only when the node-wise lifted residual is converted to the consensus proximal-point residual.

\begin{assumption}[Finite defect accumulation]\label{ass:finite-defect-accumulation}
For fixed absorption constants, the defects satisfy the condition corresponding to the oracle setting:
\begin{enumerate}[label=(\alph*),leftmargin=2em]
\item \emph{Deterministic case.}
\[
\sum_{t=1}^{\infty}\eta_t^r<\infty
\qquad\text{and}\qquad
\sum_{t=1}^{\infty}\zeta_t^r<\infty .
\]
\item \emph{Stochastic case.}
\[
\sum_{t=1}^{\infty}\eta_t^s<\infty
\qquad\text{and}\qquad
\sum_{t=1}^{\infty}\zeta_t^s<\infty .
\]
\end{enumerate}
\end{assumption}

\begin{remark}[Finite-horizon bounds and rate conditions]\label{rem:finite-defect-role}
The general deterministic and stochastic bounds retain the accumulated defects explicitly. Parts (a) and (b) of Assumption~\ref{ass:finite-defect-accumulation} make their contribution uniformly bounded in the respective rate corollaries. Deterministic verified selections control the defect tails, and stochastic verified selections give finite expected accumulation. Example~\ref{ex:df-vs-verified} illustrates finite accumulation for subgradients that vary at a fixed primal iterate.
\end{remark}

\subsection{Deterministic Verification by Active Identification}

\begin{definition}[Verified local subgradient selection]\label{def:verified-local-selection}
The deterministic subgradients in Definition~\ref{def:subgradient-motion-defects}(b) admit a \emph{verified local selection} if there exist an index $T_{\rm id}\ge1$, sets $U_i$ containing the tail points $\{x_i^t,\bar x^t:t\ge T_{\rm id}\}$, and single-valued maps $\Xi_i:U_i\to\R^d$ such that, for all $t\ge T_{\rm id}$,
\[
\xi_i^t=\Xi_i(x_i^t)\in\partial r_i(x_i^t),
\qquad
\Xi_i(\bar x^t)\in\partial r_i(\bar x^t),
\]
and each $\Xi_i$ is Lipschitz on $U_i$. The sets $U_i$ may be ordinary neighborhoods in the smooth case or relative neighborhoods on an identified face or active manifold. The finite prefix $t<T_{\rm id}$ is absorbed into the constants of the final best-iterate bound.
\end{definition}

Verified selection gives tail defect control and deterministic Lyapunov lower boundedness. Smooth convex regularizers with Lipschitz gradients satisfy it directly; polyhedral and partly smooth regularizers may do so after identification. For $\ell_1$ and group $\ell_{2,1}$ penalties, common active-structure identification also permits direct verification of the defects from the primal optimality equations, as shown in Proposition~\ref{prop:stability-sufficient}. The next example illustrates finite defect accumulation for a selected subgradient that varies at a fixed iterate.

\begin{example}[Finite accumulation with varying subgradients]\label{ex:df-vs-verified}
Consider the scalar model $f_i\equiv0$, $r_i(x)=\omega|x|$ with $\omega>0$, and $\phi(x)=x^2/2$. Choose admissible parameters with $\rho=0$ and initialize every node at $x_i^0=z_i^0=a$, where
\[
0<a<\frac{\lambda\omega}{1+2\lambda\alpha d_{\max}}.
\]
Set $q_\beta=\beta/(1+\beta)\in(0,1)$. Soft-thresholding and the smoothing update give
\[
x_i^t=0,\qquad z_i^t=a q_\beta^t,\qquad y_{ij}^t=0,\qquad t\ge1.
\]
The selected subgradients are
\[
\xi_i^1=\frac a\lambda,\qquad
\xi_i^{t+1}=c_0a q_\beta^t,\qquad t\ge1,
\]
and lie in $\partial r_i(0)=[-\omega,\omega]$. All iterates and consensus proximal points lie in $K=[-a,a]$: the threshold bound maps primal points in $K$ to zero for the generated multipliers, the smoothing update preserves $K$, and the consensus proximal point is zero. The geometric subgradient variation gives
\[
\sum_{t=1}^{\infty}\eta_t^r
\le\sum_{t=1}^{\infty}\sum_i|\xi_i^{t+1}-\xi_i^t|^2<\infty,
\qquad \zeta_t^r=0\quad(t\ge1).
\]
Thus Assumption~\ref{ass:finite-defect-accumulation}(a) holds. Since the scalar primal iterate remains zero while the selected subgradient changes at every subsequent step, a single-valued map at zero cannot represent these selections. Finite defect accumulation therefore includes this trajectory beyond the verified-selection condition.
\end{example}

\subsection{Stochastic Defects, Oracle, and Verification}

\begin{assumption}[Stochastic oracle]\label{ass:stochastic-oracle}
Let $\mathcal F_t$ denote the information available before $G^t=(G_1^t,\ldots,G_N^t)$ is drawn. With $\delta_i^t=G_i^t-\nabla K_i(x_i^t;Y_i^{t+1},Z_i^t)$, assume
\[
\E[\delta_i^t\mid\mathcal F_t]=0,
\qquad
\E[\|\delta_i^t\|^2\mid\mathcal F_t]\le (\sigma_i^t)^2<\infty.
\]
\end{assumption}

For each finite horizon $T$, denote the accumulated variance bound by
\[
V_T:=\sum_{t=0}^{T}\sum_{i=1}^N(\sigma_i^t)^2<\infty.
\]
We use $\delta^{-1}=0$ and $\sigma_i^{-1}=0$.

The stochastic rate corollaries use Assumption~\ref{ass:finite-defect-accumulation}(b). A sufficient verification uses identification at a stopping time adapted to the oracle history, together with integrability of the preceding subgradient variations.

\begin{definition}[Stochastic verified local subgradient selection]\label{def:stochastic-verified-local-selection}
The stochastic subgradients in Definition~\ref{def:subgradient-motion-defects}(c) admit a \emph{stochastic verified local selection} if the following properties hold. On an event of probability one there exist a finite $\{\mathcal F_t\}$-stopping time $T_{\rm id}^s$, satisfying $\{T_{\rm id}^s\le t\}\in\mathcal F_t$, and measurable single-valued maps $\Xi_i$ on tail neighborhoods containing $\{x_i^t,\bar x^t:t\ge T_{\rm id}^s\}$ such that, for all $t\ge T_{\rm id}^s$,
\[
\xi_i^t=\Xi_i(x_i^t)\in\partial r_i(x_i^t),
\qquad
\Xi_i(\bar x^t)\in\partial r_i(\bar x^t),
\]
and the maps $\Xi_i$ have a common deterministic Lipschitz constant on those neighborhoods. In addition,
\[
\sup_t\E\sum_i\|\xi_i^t\|^2<\infty
\]
and
\begin{equation}\label{eq:integrable-identification-prefix}
\E\sum_{1\le t<T_{\rm id}^s}
\left[
\sum_i\|\xi_i^{t+1}-\xi_i^t\|^2+
\operatorname{dist}^2\!\left(\sum_i\xi_i^t,\sum_i\partial r_i(\bar x^t)\right)
\right]<\infty .
\end{equation}
\end{definition}

\begin{remark}[Noise decay and active identification]\label{rem:noise-identification}
The sufficient condition in Definition~\ref{def:stochastic-verified-local-selection} combines almost-sure identification, Lipschitz compatibility of the selected subgradients on the identified tail, and integrable pre-identification defects. These conditions yield finite expected defect accumulation. Summability of oracle variance is an additional condition for the $O(T^{-1})$ expected squared best-iterate residual. Corollary~\ref{cor:stochastic-proximal} describes the joint contribution of expected defects and oracle variance for general noise sequences.
\end{remark}

\subsection{Lyapunov Lower Bounds and Sufficient Regimes}

Fix an auxiliary weight $\tau>0$ and define the composite primal energy and dual-corrected Lyapunov sequence by
\[
H^t:=\sum_{i=1}^N\big(F_i(x_i^t,Z_i^t)+r_i(x_i^t)\big),
\qquad
\mathcal L_r^t:=H^t+\langle y^t,s^t\rangle+\tau S^t.
\]
For $r_i\equiv0$, this reduces to the smooth Lyapunov sequence $\mathcal L^t$.

\begin{assumption}[Lower bounded Lyapunov sequence]\label{ass:lyapunov-lower-bound}
In the deterministic composite setting, $\inf_t\mathcal L_r^t> -\infty$. In the stochastic composite setting, $\inf_t\E\mathcal L_r^t> -\infty$. The smooth deterministic analysis establishes this property directly.
\end{assumption}

\begin{proposition}[Sufficient lower-bound condition]\label{prop:lyapunov-lower-sufficient}
Let Assumption~\ref{ass:standing-local} hold. In the deterministic composite setting, if the selected subgradients $\xi_i^t$ are uniformly bounded, then Assumption~\ref{ass:lyapunov-lower-bound} holds. In the stochastic composite setting, if $\sup_t\E\sum_i\|\xi_i^t\|^2<\infty$ and $\sup_t\sum_i(\sigma_i^t)^2<\infty$, then the expected Lyapunov sequence is bounded below.
\end{proposition}

For $r_i(x)=\omega_i\|x\|_1$, every subgradient has norm at most $\omega_i\sqrt d$. For $r_i(x)=\omega_i\sum_{g\in\mathcal P_i}\|x_g\|_2$, where $\mathcal P_i$ is a partition of the coordinates, the corresponding bound is $\omega_i\sqrt{|\mathcal P_i|}$. Squared $\ell_2$ penalties have bounded gradients on $K$. Thus, under the standing geometry, Proposition~\ref{prop:lyapunov-lower-sufficient} supplies the deterministic Lyapunov lower bound for these penalties directly, and the expected lower bound when the oracle variance is uniformly bounded. Summable variance in the stochastic $O(T^{-1})$ corollaries satisfies this variance requirement.

\begin{corollary}[Verified selection gives deterministic lower boundedness]\label{cor:verified-selection-lower-bound}
Under Assumptions~\ref{ass:standing-local} and~\ref{ass:composite-regularizer}, a verified local selection in the sense of Definition~\ref{def:verified-local-selection} gives Assumption~\ref{ass:lyapunov-lower-bound} in the deterministic setting.
\end{corollary}

\begin{corollary}[Stochastic verified selection gives lower boundedness]\label{cor:stochastic-verified-selection-lower-bound}
Under Assumptions~\ref{ass:standing-local}, \ref{ass:composite-regularizer}, and~\ref{ass:stochastic-oracle}, a stochastic verified local selection in the sense of Definition~\ref{def:stochastic-verified-local-selection}, together with $\sup_t\sum_i(\sigma_i^t)^2<\infty$, gives Assumption~\ref{ass:lyapunov-lower-bound} in expectation.
\end{corollary}

\begin{proposition}[Sufficient cases for local nonsmooth stability]\label{prop:stability-sufficient}
Let Assumption~\ref{ass:standing-local} hold. In each regime below, fixed absorption constants give $\eta_t^r=\zeta_t^r=0$ on an identified tail and $\sum_t\eta_t^r<\infty$. Parts (a) and (b) use the selected-subgradient representations in Definition~\ref{def:verified-local-selection}; part (c) controls the defects directly. In parts (a) and (c), the pre-identification aggregation defects are finite, giving Assumption~\ref{ass:finite-defect-accumulation}(a). The corresponding stochastic verification uses an almost-sure identified tail and integrable pre-identification subgradient variations and aggregation mismatch, as in \eqref{eq:integrable-identification-prefix}.
\begin{enumerate}[label=(\alph*),leftmargin=2em]
\item If each $r_i$ is $C^1$ with Lipschitz gradient on a neighborhood of $K$, take $\Xi_i=\nabla r_i$.
\item For polyhedral or partly smooth $r_i$, use a verified local selection on the common identified active face or manifold.
\item For $r_i(x)=\omega_i\|x\|_1$, the node and averaged iterates identify a common support and sign pattern. For group $\ell_{2,1}$ penalties, they identify common active groups, and the active-group norms at both node and averaged iterates are bounded away from zero. The active subgradients are fixed or Lipschitz on these structures, and the primal optimality equations control their variation on the common inactive coordinates.
\end{enumerate}
\end{proposition}
A finite identification prefix contributes a horizon-independent constant to the proximal-residual bound. The general finite-horizon bounds express the contribution of structural transitions through the accumulated defects.

\begin{remark}[Identification condition and numerical evidence]\label{rem:identification-diagnostic}
Proposition~\ref{prop:stability-sufficient} takes common active-structure identification and the stated local subgradient conditions as hypotheses and derives tail defect control. Section~\ref{sec:exp-defect} studies the corresponding numerical behavior through an observed stabilization index and defect values on $0\le t\le T_{\rm obs}$.
\end{remark}

\section{Main Theoretical Results}\label{sec:main-results}
The main bounds are obtained in three steps. Lyapunov descent controls primal motion, reference motion, smoothing mismatch, and consensus. These quantities bound a lifted primal--dual residual. The aggregation defect is used only when this bound is converted to the consensus composite proximal residual. Finite accumulated defects give the same deterministic rate as in the smooth case; the stochastic bound also contains the accumulated oracle variance.

\subsection{Residual Interpretation and Smooth Baseline}

Given $z^t$, select the consensus Bregman proximal comparison point
\[
u_r^t\in\arg\min_{u\in\mathcal C}\sum_{i=1}^N\big(F_i(u,Z_i^t)+r_i(u)\big),
\]
which lies in $K$ by Assumption~\ref{ass:standing-local}, and define
\begin{equation}\label{eq:stationarity-measure-main}
\mathcal G_r(z^t)=\max_{i\in[N]}\|z_i^t-u_r^t\|.
\end{equation}
For $r_i\equiv0$, write $\mathcal G_0$. The residual $\mathcal G_r(z^t)$ measures the distance from the references to this common comparison point; we call $z^t$ $\epsilon$-stationary in the proximal-residual sense if $\mathcal G_r(z^t)\le\epsilon$. Proposition~\ref{prop:kkt-bridge} connects a small fixed-point residual to first-order stationarity.

\begin{proposition}[Connection with stationarity]\label{prop:kkt-bridge}
Let $u_r(z)$ be a selected consensus Bregman proximal point in $K$. Under Assumptions~\ref{ass:standing-local} and~\ref{ass:composite-regularizer},
\[
\operatorname{dist}\!\left(0,\sum_{i=1}^N\nabla f_i(u_r(z))+\sum_{i=1}^N\partial r_i(u_r(z))\right)
\le N L_\phi(c_0+\rho d_{\max})\mathcal G_r(z).
\]
Consequently, $\mathcal G_r(z)=0$ implies first-order stationarity of the consensus composite objective at $u_r(z)\in\mathcal C$. If $\mathcal C=\R^d$, this is the usual unconstrained stationarity condition.
\end{proposition}

To connect this Bregman residual directly to the Euclidean composite stationarity metric used in the experiments, define
\[
\bar f:=\frac1N\sum_{i=1}^N f_i,
\qquad
\bar r:=\frac1N\sum_{i=1}^N r_i,
\qquad
P_\gamma(x):=\frac1\gamma\left[x-\operatorname{prox}_{\gamma\bar r}\bigl(x-\gamma\nabla\bar f(x)\bigr)\right].
\]
Here $\operatorname{prox}_{\gamma\bar r}$ denotes the Euclidean proximal map of $\bar r$ with stepsize $\gamma$.

\begin{proposition}[Bridge to the empirical Euclidean stationarity metric]\label{prop:empirical-stationarity-bridge}
Let $L_{\bar f,K}$ be a Lipschitz constant of $\nabla\bar f$ on $K$. For any $\gamma>0$ and $t\ge1$,
\begin{equation}\label{eq:empirical-stationarity-bridge}
\|P_\gamma(\bar x^t)\|
\le \left[L_\phi(c_0+\rho d_{\max})+L_{\bar f,K}+\frac2\gamma\right]\mathcal G_r(z^t)
+\left(L_{\bar f,K}+\frac2\gamma\right)\sqrt{\frac{e^{t-1}}{N}}.
\end{equation}
\end{proposition}

\begin{proof}
Let $\bar z^t=N^{-1}\sum_i z_i^t$. The subdifferential sum rule, the proximal fixed-point characterization, and Proposition~\ref{prop:kkt-bridge} give $\|P_\gamma(u_r^t)\|\le L_\phi(c_0+\rho d_{\max})\mathcal G_r(z^t)$, while proximal nonexpansiveness makes $P_\gamma$ $(2/\gamma+L_{\bar f,K})$-Lipschitz on $K$. Combining these facts with
\[
\|\bar x^t-u_r^t\|
\le \|\bar x^t-\bar z^t\|+\|\bar z^t-u_r^t\|
\le \sqrt{e^{t-1}/N}+\mathcal G_r(z^t),
\]
proves \eqref{eq:empirical-stationarity-bridge}.
\end{proof}

\begin{lemma}[Proximal-residual admissible parameter region]\label{lem:admissible-window-main}
For the auxiliary Lyapunov weight $\tau>0$ fixed above, define
\[
a_x:=\frac1\lambda-(L_++c_0+\rho d_{\max}),\qquad
a_z:=c_0+\rho d_{\min},\qquad
a_e:=\frac{c_0}{\beta},
\]
\[
q:=\tau+\frac{\rho}{\beta},\qquad b:=\frac{q}{\alpha^2}.
\]
Let $\widehat\Gamma_x^+,\widehat\Gamma_x^-,\widehat\Gamma_z$ be the finite constants in Lemma~\ref{det:lem:consensus-by-steps}; these constants are independent of the dual stepsize $\alpha$. If
\begin{equation}\label{eq:main-admissible-region}
c_0+\rho\bar d>L_-,
\qquad
\tau>\alpha,
\qquad
a_x>b(\widehat\Gamma_x^++\widehat\Gamma_x^-),
\qquad
a_z>b\widehat\Gamma_z,
\end{equation}
then the parameter tuple is admissible for the smooth proximal-residual theorem. Treating the local constants induced by Lemma~\ref{lem:local-variation-from-relative} as fixed, the region is nonempty: choose $c_0>L_-$, choose $\lambda<(L_++c_0)^{-1}$, take $\rho>0$ small enough that $a_x>0$, set $\tau=\kappa\alpha$ for any fixed $\kappa>1$, and then take $\alpha$ large enough.
\end{lemma}

\begin{theorem}[Kernel-domain deterministic SDBPG proximal residual]\label{thm:deterministic-main}
Suppose $r_i\equiv0$. For any deterministic SDBPG trajectory generated under Assumption~\ref{ass:standing-local} with parameters in the admissible region of Lemma~\ref{lem:admissible-window-main}, the method satisfies, for every $T\ge1$,
\[
\min_{1\le k\le T}\mathcal G_0(z^k)^2
\le
\frac{\Delta}{T},
\]
where $\Delta<\infty$ depends on the initialization, graph, localization constants, and chosen parameters, but not on $T$. Hence SDBPG produces an iterate with $\mathcal G_0(z^k)\le\epsilon$ within $O(\epsilon^{-2})$ iterations.
\end{theorem}

\begin{corollary}[Best-iterate stationarity residual]\label{cor:kkt-rate}
Under Theorem~\ref{thm:deterministic-main},
\[
\min_{1\le k\le T}\operatorname{dist}\!\left(0,\sum_{i=1}^N\nabla f_i(u_0^k)\right)=O(T^{-1/2}),
\]
where $u_0^k$ is the common component of the selected consensus Bregman proximal point for $z^k$.
\end{corollary}

\subsection{Composite Residual Bounds}

Unlike the smooth system, the composite optimality conditions contain selected subgradients that may jump. The next bound keeps the unabsorbed motion and aggregation mismatch explicit instead of imposing Lipschitz variation.

\begin{lemma}[Composite admissible parameter region]\label{lem:composite-admissible}
Assume the constants in Assumptions~\ref{ass:standing-local} and~\ref{ass:composite-regularizer}, together with the absorption constant $M_{T\xi}$ in Definition~\ref{def:subgradient-motion-defects}(b), are fixed. Let $\widehat\Gamma_{r,x}^+,\widehat\Gamma_{r,x}^-,\widehat\Gamma_{r,z}$ be the finite constants from Lemma~\ref{comp:lem:consensus-by-steps}. With $a_x,a_z,a_e,q,b$ defined above, suppose
\begin{equation}\label{eq:composite-admissible-region}
\tau>\alpha,
\qquad
a_x>b(\widehat\Gamma_{r,x}^++\widehat\Gamma_{r,x}^-),
\qquad
a_z>b\widehat\Gamma_{r,z}.
\end{equation}
Then the parameter tuple is admissible for the deterministic composite estimates. The region is nonempty by the same choice order as in Lemma~\ref{lem:admissible-window-main}, without requiring $c_0+\rho\bar d>L_-$ unless the proximal-point residual is invoked.
\end{lemma}

\begin{theorem}[Deterministic composite SDBPG proximal residual with transition terms]\label{thm:composite-main}
Consider any deterministic composite SDBPG trajectory satisfying Assumptions~\ref{ass:standing-local}, \ref{ass:composite-regularizer}, and~\ref{ass:lyapunov-lower-bound}, with the deterministic aggregation defect defined by Definition~\ref{def:aggregation-defects}(a). Suppose the parameters are compatible with Lemma~\ref{lem:composite-admissible} for the absorption constant used in Definition~\ref{def:subgradient-motion-defects}(b), and suppose
\[
c_0+\rho\bar d>L_-.
\]
Then there exist constants $\Delta_r,C_r<\infty$, independent of $T$, such that, for every $T\ge1$,
\[
\min_{1\le k\le T}\mathcal G_r(z^k)^2
\le
\frac{\Delta_r+C_r\sum_{t=1}^{T}(\eta_t^r+\zeta_t^r)}{T}.
\]
Consequently, with $T=O(\epsilon^{-2})$, the method produces an iterate satisfying $\mathcal G_r(z^k)\le\epsilon$ whenever $T^{-1}\sum_{t=1}^T(\eta_t^r+\zeta_t^r)=O(\epsilon^2)$.
\end{theorem}

The horizon-average defects are the additional terms caused by nonsmooth transitions. Only $\eta_t^r$ enters the lifted estimate of Proposition~\ref{prop:deterministic-lifted-estimate}; $\zeta_t^r$ is introduced when that estimate is converted to $\mathcal G_r$.

\begin{corollary}[Deterministic composite rate under finite defect accumulation]\label{cor:finite-transition-composite-rate}
Under Theorem~\ref{thm:composite-main}, suppose Assumption~\ref{ass:finite-defect-accumulation}(a) holds with the same absorption constants. Then there exists $\bar\Delta_r<\infty$, independent of $T$, such that
\[
\min_{1\le k\le T}\mathcal G_r(z^k)^2
\le
\frac{\bar\Delta_r}{T},
\qquad T\ge1.
\]
The corresponding best-iterate residual satisfies $\min_{1\le k\le T}\mathcal G_r(z^k)=O(T^{-1/2})$. Thus deterministic composite SDBPG produces an iterate with $\mathcal G_r(z^k)\le\epsilon$ within $O(\epsilon^{-2})$ iterations.
\end{corollary}

Proposition~\ref{prop:stability-sufficient} gives defect control in identified regimes. For $\ell_1$ and group $\ell_{2,1}$ penalties, bounded selected subgradients give the Lyapunov lower bound through Proposition~\ref{prop:lyapunov-lower-sufficient}; verified local selections give it through Corollary~\ref{cor:verified-selection-lower-bound}.

\begin{corollary}[Structured nonsmooth composite rate]\label{cor:structured-nonsmooth-rate}
Assume the remaining conditions of Theorem~\ref{thm:composite-main}. Suppose either Assumption~\ref{ass:finite-defect-accumulation}(a) holds, or the selected subgradients admit a verified local selection in the sense of Definition~\ref{def:verified-local-selection}. The parameters are compatible with the absorption constants supplied by the chosen verification. Then there exists $\bar\Delta_r<\infty$ such that
\[
\min_{1\le k\le T}\mathcal G_r(z^k)^2
\le
\frac{\bar\Delta_r}{T},
\qquad T\ge1.
\]
Thus the identified polyhedral, partly smooth, $\ell_1$, and group $\ell_{2,1}$ regimes in Proposition~\ref{prop:stability-sufficient} retain the smooth theorem's $O(\epsilon^{-2})$ best-iterate order for the composite residual.
\end{corollary}

\subsection{Stochastic Residual Bounds}

The stochastic analysis first uses a lifted residual that separates primal--dual stationarity from consensus before subgradients are aggregated at the network average. For stacked primal variables $x=(x_1,\ldots,x_N)$, set
\[
\nabla f(x):=(\nabla f_1(x_1),\ldots,\nabla f_N(x_N)),
\qquad
\partial r(x):=\prod_{i=1}^N\partial r_i(x_i).
\]
The Cartesian product $\partial r(x)$ collects one node subgradient per block.
For edge multipliers $y$, define the lifted primal-dual residual
\begin{equation}\label{eq:lifted-kkt-residual}
\mathcal R_{\mathrm{lift},r}(x,y)^2
:=
\operatorname{dist}^2\!\left(
0,\nabla f(x)+\mathcal B y+\partial r(x)
\right)
+\|\mathcal B^\top x\|^2,
\end{equation}
The two terms in this residual measure primal--dual stationarity and consensus, respectively. In the smooth case $r_i\equiv0$, write $\mathcal R_{\mathrm{lift},0}$.

\begin{lemma}[Stochastic admissible parameter region]\label{lem:stochastic-admissible}
Assume the constants in Assumptions~\ref{ass:standing-local} and~\ref{ass:composite-regularizer}, together with the absorption constants $M_{T\xi}^s,M_{\delta\xi}$ in Definition~\ref{def:subgradient-motion-defects}(c), are fixed. Choose an auxiliary splitting parameter $\chi\in(0,\mu a_x)$ and set $a_x^s=a_x-\chi/\mu$. Let $\widehat\Gamma_{s,x}^+,\widehat\Gamma_{s,x}^-,\widehat\Gamma_{s,z},\widehat\Gamma_{s,\delta}$ be the finite constants from Lemma~\ref{stoch:lem:consensus-by-steps}. If
\begin{equation}\label{eq:stochastic-admissible-region}
0<\chi<\mu a_x,
\qquad
\tau>\alpha,
\qquad
a_x^s>b(\widehat\Gamma_{s,x}^++\widehat\Gamma_{s,x}^-),
\qquad
a_z>b\widehat\Gamma_{s,z},
\end{equation}
then the parameter tuple is admissible for the stochastic composite lifted-residual theorem. The admissible region is nonempty after fixing the localized constants.
\end{lemma}

\begin{theorem}[Stochastic composite lifted complexity]\label{thm:stochastic-main}
For any stochastic composite SDBPG trajectory satisfying Assumptions~\ref{ass:standing-local}, \ref{ass:composite-regularizer}, \ref{ass:stochastic-oracle}, and~\ref{ass:lyapunov-lower-bound}, with expected subgradient-motion defects defined by Definition~\ref{def:subgradient-motion-defects}(c) and parameters compatible with Lemma~\ref{lem:stochastic-admissible},
\[
\E\left[\min_{1\le k\le T}\mathcal R_{\mathrm{lift},r}(x^k,y^k)^2\right]
\le
\frac{\Delta_s+C_s\sum_{t=0}^{T}\sum_{i=1}^N(\sigma_i^t)^2+C_{\eta,s}\sum_{t=1}^{T}\eta_t^s}{T}.
\]
This holds for every $T\ge1$. If the per-step variance bounds and expected motion defects are summable, the expected squared best-iterate residual is $O(T^{-1})$. If their horizon averages are $O(\epsilon^2)$ with $T=O(\epsilon^{-2})$, then the displayed expected squared residual is $O(\epsilon^2)$.
\end{theorem}

\begin{corollary}[Expected asymptotic lifted residual]\label{cor:stochastic-asymptotic}
Under Theorem~\ref{thm:stochastic-main}, if $\sum_{t=0}^{\infty}\sum_i(\sigma_i^t)^2<\infty$ and $\sum_t\eta_t^s<\infty$, then $\E\,\mathcal R_{\mathrm{lift},r}(x^t,y^t)^2\to0$.
\end{corollary}

\begin{corollary}[Stochastic proximal-point residual with transition terms]\label{cor:stochastic-proximal}
Under Theorem~\ref{thm:stochastic-main}, define the stochastic aggregation defect as in Definition~\ref{def:aggregation-defects}(b) and suppose $c_0+\rho\bar d>L_-$. Then
\[
\E\left[\min_{1\le k\le T}\mathcal G_r(z^k)^2\right]
\le
\frac{\Delta_{s,r}+C_{s,r}\sum_{t=0}^{T}\sum_i(\sigma_i^t)^2+C'_{s,r}\sum_{t=1}^{T}(\eta_t^s+\zeta_t^s)}{T}.
\]
\end{corollary}

\begin{corollary}[Stochastic transition-free rate under finite expected defect accumulation]\label{cor:stochastic-finite-accumulation}
Under Corollary~\ref{cor:stochastic-proximal}, suppose in addition Assumption~\ref{ass:finite-defect-accumulation}(b). Then there exists $\bar\Delta_{s,r}<\infty$, independent of $T$, such that
\[
\E\left[\min_{1\le k\le T}\mathcal G_r(z^k)^2\right]
\le
\frac{\bar\Delta_{s,r}+C_{s,r}\sum_{t=0}^{T}\sum_i(\sigma_i^t)^2}{T}.
\]
If furthermore $\sum_{t=0}^{\infty}\sum_i(\sigma_i^t)^2<\infty$, the expected squared best-iterate composite proximal residual is $O(T^{-1})$, matching the deterministic finite-defect order of Corollary~\ref{cor:finite-transition-composite-rate}. For general variance sequences, its bound is $O(T^{-1}+V_T/T)$.
\end{corollary}

\begin{remark}[Variance accumulation and batch size]\label{rem:variance-batch-rate}
Under the conditions of Corollary~\ref{cor:stochastic-finite-accumulation}, a uniform network variance bound $\sum_i(\sigma_i^t)^2\le\bar\sigma^2$ yields an expected squared best-iterate residual bound of $O(T^{-1})+O(\bar\sigma^2)$. For averages of $b_t$ conditionally independent samples whose individual conditional variances have a uniform network sum bound $v$, one has $V_T\le v\sum_{t=0}^T b_t^{-1}$. A fixed batch size $B$ therefore gives the bound $O(T^{-1})+O(v/B)$. A schedule satisfying $\sum_{t=0}^{\infty}b_t^{-1}<\infty$, such as geometrically increasing batches, gives the $O(T^{-1})$ squared-residual order. In this summable-variance regime, $\bigl(\E[\min_{1\le k\le T}\mathcal G_r(z^k)^2]\bigr)^{1/2}=O(T^{-1/2})$; the associated $O(\epsilon^{-2})$ bound counts iterations, with sampling cost determined by the batch schedule.
\end{remark}

The next corollary derives the variance-dependent bound from the stochastic selection condition, which gives finite expected defect accumulation through Lemma~\ref{stoch:lem:verified-selection-absorption}.

\begin{corollary}[Stochastic proximal-point residual under verified selection]\label{cor:stochastic-verified-proximal}
Consider stochastic composite SDBPG under Assumptions~\ref{ass:standing-local}, \ref{ass:composite-regularizer}, and~\ref{ass:stochastic-oracle}, with a stochastic verified local selection in the sense of Definition~\ref{def:stochastic-verified-local-selection}. Suppose $\sup_t\sum_i(\sigma_i^t)^2<\infty$, the parameters are compatible with Lemma~\ref{lem:stochastic-admissible} for the absorption constants induced by the stochastic verified selection, and
\[
c_0+\rho\bar d>L_-.
\]
Then there exist constants $\bar\Delta_{s,r},C_{s,r}<\infty$, independent of $T$, such that
\[
\E\left[\min_{1\le k\le T}\mathcal G_r(z^k)^2\right]
\le
\frac{\bar\Delta_{s,r}+C_{s,r}\sum_{t=0}^{T}\sum_i(\sigma_i^t)^2}{T}.
\]
If $\sum_{t=0}^{\infty}\sum_i(\sigma_i^t)^2<\infty$, the expected squared composite proximal-point residual has the same $O(T^{-1})$ best-iterate order as the deterministic finite-defect corollary. More generally, accumulated variance is the only nonconstant numerator term.
\end{corollary}
The constants may depend on the finite expected prefix in Definition~\ref{def:stochastic-verified-local-selection}, but not on $T$.

\begin{corollary}[Structured stochastic composite rate]\label{cor:structured-stochastic-rate}
Consider stochastic composite SDBPG under Assumptions~\ref{ass:standing-local}, \ref{ass:composite-regularizer}, and~\ref{ass:stochastic-oracle}, with $\sup_t\sum_i(\sigma_i^t)^2<\infty$, parameters compatible with Lemma~\ref{lem:stochastic-admissible} for the absorption constants used in the verification, and $c_0+\rho\bar d>L_-$. Suppose either a stochastic verified local selection in the sense of Definition~\ref{def:stochastic-verified-local-selection} holds, or the $\ell_1$ or group $\ell_{2,1}$ regime in Proposition~\ref{prop:stability-sufficient}(c) holds on an almost-sure identified tail with the integrable prefix in \eqref{eq:integrable-identification-prefix}. Then the variance-dependent bound of Corollary~\ref{cor:stochastic-verified-proximal} holds. If $\sum_{t=0}^{\infty}\sum_i(\sigma_i^t)^2<\infty$, the expected squared composite proximal-point residual has $O(T^{-1})$ best-iterate order.
\end{corollary}

\section{Experiments}\label{sec:experiments}
The experiments examine defect decay after observed active-structure stabilization, composite sparse recovery, the effect of a geometry-matched quartic kernel, and reconstruction--communication tradeoffs against recent decentralized methods. The regularized recovery studies use common penalties across nodes, $r_i=r$, while local sensing operators and observations account for data heterogeneity. The active-structure experiment uses the quadratic kernel with $\mathcal X=\R^d$, whereas Section~\ref{sec:quartic-phase-retrieval} studies an unregularized phase-retrieval model with a nonlinear quartic kernel. Appendix~\ref{app:experimental-details} gives the problem definitions, additional trajectories, and sensitivity results.

\subsection{Setup and Evaluation Metrics}\label{sec:experimental-setup}
Baselines include decentralized gradient descent (DGD), decentralized gradient tracking (DGT), their proximal variants (ProxDGD and ProxDGT), Prox-EXTRA, ProxGT \citep{xin2021stochastic}, and Prox-DASA-GT \citep{xiao2023one}, as appropriate. Hyperparameters are selected on validation instances and kept fixed on disjoint evaluation instances. Paired comparisons use common data, graphs, and initialization. Same-penalty proximal comparisons isolate algorithmic effects, while unregularized controls quantify the contribution of structural regularization. Appendix~\ref{app:experimental-details} specifies the budgets and regularization weights for each study. Tables report evaluation means; 95\% percentile-bootstrap confidence intervals (CIs) use paired resampling for paired differences.

The quadratic losses in linear, shared-support, and Fourier recovery satisfy two-sided relative smoothness with a quadratic kernel, including for rank-deficient local sensing operators. The quartic kernel in Section~\ref{sec:quartic-phase-retrieval} matches the global curvature growth of the phase-retrieval loss, whose Euclidean curvature is unbounded; quadratic-kernel runs provide a Euclidean reference. Appendix~\ref{app:experimental-details} relates the experimental parameter selection to the analytical sufficient conditions.

For recent-algorithm comparisons, cumulative cost comprises sequential communication rounds, transmitted volume in mebibytes (MiB), consumed oracle samples, and wall time. One round permits simultaneous neighbor transmission of all states available before that exchange, and the sample count includes the data used by local-gradient calls. Full gradients used only for post-run stationarity evaluation are excluded from algorithmic oracle totals. Appendix~\ref{app:modern-baseline-costs} gives method-specific recursions, transmitted states, parameters, and cost identities.

Let $\widehat x=\bar x^T$ denote the final averaged iterate. Vector recovery error is $\|\widehat x-x^\star\|/\|x^\star\|$; for phase retrieval it is sign invariant, $\min_{\nu\in\{-1,1\}}\|\widehat x-\nu x^\star\|/\|x^\star\|$, and group recovery uses the analogous Frobenius relative error. For true and estimated supports $\mathcal S^\star$ and $\widehat{\mathcal S}$,
\[
P=\frac{|\widehat{\mathcal S}\cap\mathcal S^\star|}{|\widehat{\mathcal S}|},
\qquad
R=\frac{|\widehat{\mathcal S}\cap\mathcal S^\star|}{|\mathcal S^\star|},
\qquad
\mathrm{F1}=\frac{2PR}{P+R},
\]
where $P$ and $R$ are support precision and recall, with zero-denominator values set to zero. Coordinate and group support use fixed thresholds $|\widehat x_k|\ge0.05$ and row norm at least $0.12$, respectively; support size is $|\widehat{\mathcal S}|$. Consensus error is $N^{-1}\sum_i\|x_i^T-\bar x^T\|^2$. With the full gradient and reference step $\gamma_{\rm exp}=0.25$, the reported Euclidean composite stationarity metric is
\[
\mathsf S_{\rm exp}^t
:=\left\|\frac{\bar x^t-\operatorname{prox}_{\gamma_{\rm exp}\bar r}
\bigl(\bar x^t-\gamma_{\rm exp}\nabla\bar f(\bar x^t)\bigr)}{\gamma_{\rm exp}}\right\|^2
=\|P_{\gamma_{\rm exp}}(\bar x^t)\|^2.
\]
Proposition~\ref{prop:empirical-stationarity-bridge} bounds this metric by the theoretical Bregman residual and the smoothing mismatch; Appendix~\ref{app:residual-components} displays the lifted-residual components.

\subsection{Defect Accumulation and Active-Structure Stabilization}\label{sec:exp-defect}
We examine the relation between observed support stabilization and defect decay for the $\ell_1$ and group $\ell_{2,1}$ penalties in Proposition~\ref{prop:stability-sufficient}. For each penalty, we generate ten decentralized sparse least-squares instances with $N=6$ agents, $d=40$ candidate coordinates, true active size $s=5$, and $25$ measurements per agent; the group problem contains four signals with shared row support. The graph is a ring with one chord. We use the Euclidean kernel, penalty weight $\omega=0.05$, observation-noise standard deviation $10^{-3}$, and $(\lambda,\alpha,\beta,c_0,\rho)=(0.04,0.02,0.5,0.5,0.1)$.

This study uses local losses $\tfrac12\|\widetilde A_i x-\widetilde b_i\|^2$ and their matrix-valued counterparts, with independent sensing entries distributed as $\mathcal N(0,1/25)$. The active coordinates are $g+\operatorname{sign}(g)$ for $g\sim\mathcal N(0,1)$; active group rows have normalized Gaussian directions and independent lengths uniform on $[1.4,2.4]$. These amplitudes provide separated active components. All primal and reference variables start at zero. The motion and aggregation defects use absorption coefficient $3$ in Definitions~\ref{def:subgradient-motion-defects} and~\ref{def:aggregation-defects}. Let $\mathcal A^t$ denote the active support of $\bar x^t$, augmented with signs for the $\ell_1$ case; a coordinate or group is active when its magnitude or row norm exceeds $10^{-7}$. For the observation horizon $T_{\rm obs}=3000$, define
\[
\widehat T_{\rm stab}(T_{\rm obs})
:=1+\max\!\left(\{-1\}\cup\{0\le t<T_{\rm obs}:\mathcal A^t\ne\mathcal A^{T_{\rm obs}}\}\right),
\]
the first recorded iteration after which the active structure remains unchanged up to $T_{\rm obs}$. We plot $\eta_t^r$, $\zeta_t^r$, their cumulative sum, and the active size of $\bar x^t$ alongside this observed stabilization index.

At $T_{\rm obs}$, all 20 runs recover the true active structure. The median values of $\widehat T_{\rm stab}(T_{\rm obs})$ are 1040 for $\ell_1$ and 962 for $\ell_{2,1}$. We evaluate strict complementarity using the aggregate gradient $\sum_i\nabla f_i(\bar x^{T_{\rm obs}})$ and regularization radius $N\omega=0.3$. The radius minus the largest inactive-coordinate magnitude or inactive-group norm is positive in every run, with minima $0.131$ and $0.182$; the corresponding minimum active magnitudes are $0.968$ and $1.373$. From 50 iterations after $\widehat T_{\rm stab}(T_{\rm obs})$, the largest value of $\eta_t^r+\zeta_t^r$ is below $3.4\times10^{-15}$. The cumulative defects then remain constant to numerical precision, with median final values $6.27\times10^2$ and $6.86\times10^2$, and the maximum final edge-consensus error $S^{T_{\rm obs}}$ is below $1.9\times10^{-23}$.

\begin{figure}[!htbp]
\centering
\includegraphics[width=\textwidth]{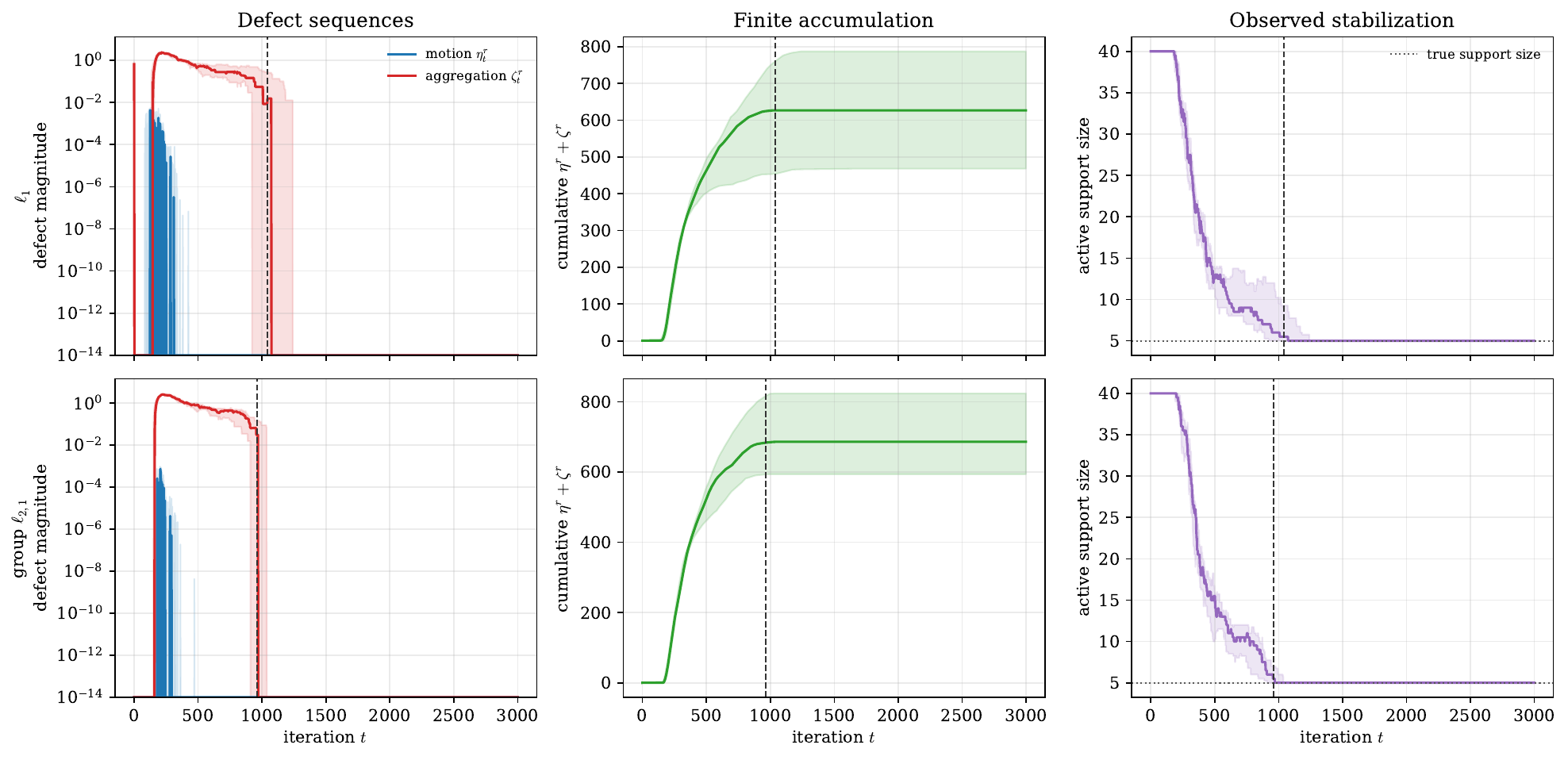}
\caption{Defect and active-structure stabilization for $\ell_1$ (top) and group $\ell_{2,1}$ (bottom) recovery. Solid curves show medians over ten instances and shaded bands show interquartile ranges. The dashed line marks the median $\widehat T_{\rm stab}(T_{\rm obs})$; thereafter the defects decay to numerical zero, the cumulative sums plateau, and the active size remains at $s=5$.}
\label{fig:defect-identification}
\end{figure}
\FloatBarrier

\subsection{Composite Sparse Recovery}
Linear sparse recovery is used to assess both recovery error and support identification. ProxDGD and ProxDGT use the same $\ell_1$ penalty, so they separate the effect of the SDBPG update from that of regularization alone.

\begin{table}[!htbp]
\centering
\caption{Composite sparse recovery in noisy undersampled and noisy moderate regimes, with true support sizes 4 and 5, respectively. Bold entries mark the lowest recovery error and highest support F1. Support size describes the estimated structure.}
\label{tab:sparse-recovery-main}

\scriptsize

\resizebox{\textwidth}{!}{%
\begin{tabular}{lcccccc}
\toprule
& \multicolumn{3}{c}{Noisy undersampled} & \multicolumn{3}{c}{Noisy moderate}\\
\cmidrule(lr){2-4}\cmidrule(lr){5-7}
Algorithm & Recovery error & Support F1 & Support size & Recovery error & Support F1 & Support size\\
\midrule
$\ell_1$-SDBPG & \best{0.657} & \best{0.724} & 2.9 & \best{0.632} & \best{0.620} & 2.6\\
SDBPG & 0.814 & 0.227 & 27.1 & 0.819 & 0.229 & 31.9\\
DGD & 1.058 & 0.223 & 27.7 & 1.061 & 0.181 & 33.7\\
DGT & 0.949 & 0.245 & 27.5 & 0.957 & 0.207 & 32.8\\
ProxDGD & 0.838 & 0.385 & 8.0 & 0.808 & 0.338 & 8.7\\
ProxDGT & 0.781 & 0.518 & 5.7 & 0.751 & 0.457 & 5.4\\
\bottomrule
\end{tabular}}
\end{table}
\FloatBarrier

In both regimes of Table~\ref{tab:sparse-recovery-main}, $\ell_1$-SDBPG has the lowest recovery error and highest support F1. Against the same-penalty ProxDGT baseline, the paired recovery differences are $-0.125$ (95\% CI $[-0.151,-0.096]$) and $-0.119$ (95\% CI $[-0.144,-0.099]$); the support-F1 differences are $0.206$ (95\% CI $[0.114,0.286]$) and $0.163$ (95\% CI $[0.115,0.213]$). The unregularized methods remain dense, while the proximal baselines recover sparser but less accurate solutions.

We repeat the comparison in 24 settings that vary observation noise, graph topology, data heterogeneity, sample count, and signal sparsity. For each comparator, Figure~\ref{fig:sparse-recovery-coverage} counts joint statistical gains: the paired 95\% intervals favor $\ell_1$-SDBPG for both recovery error and support F1, the corresponding two-sided sign-flip tests have $p<0.05$, and the success rate does not significantly decrease. This occurs in 23 settings against each of SDBPG, DGD, DGT, and ProxDGD, and in 22 against ProxDGT. Mean recovery errors are lower and mean support F1 values are higher in all 24 settings. At the largest tested sparsity, some support-F1 intervals contain zero. Appendix~\ref{app:sparse-recovery-trajectories} shows the corresponding tradeoff and trajectories. In the separately tuned EXTRA comparison of Appendix~\ref{app:extra-sparse-baselines}, PG-EXTRA has a slightly lower error in the undersampled regime but remains dense. $\ell_1$-SDBPG has the highest support F1 in both regimes and the lowest error in the moderate regime.

\begin{figure}[!htbp]
\centering
\includegraphics[width=0.96\textwidth]{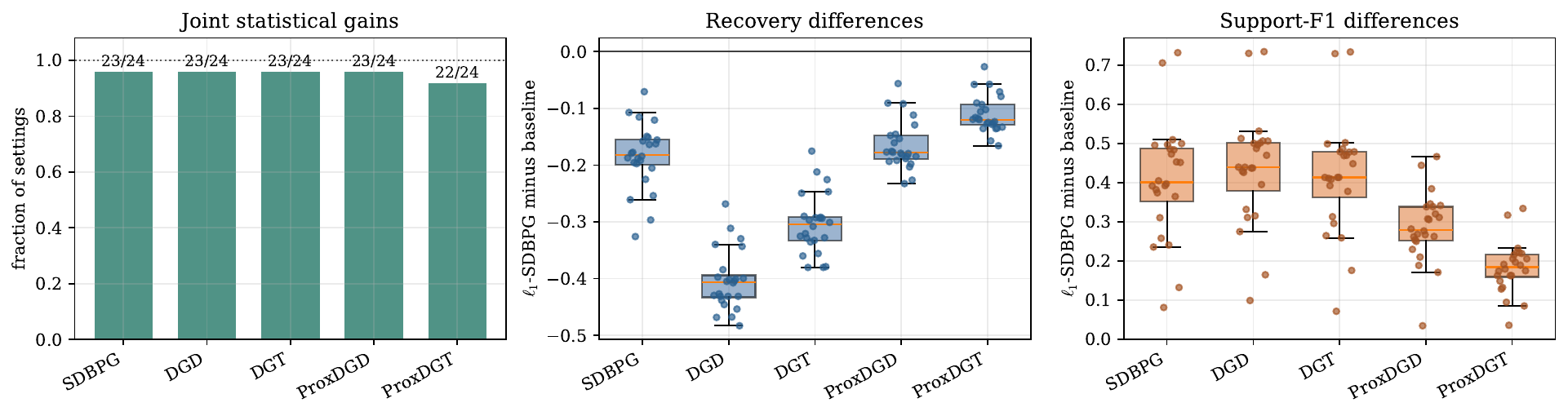}
\caption{Sparse-recovery robustness across 24 settings. Left: fraction satisfying the joint statistical-gain criterion stated in the text. Center and right: distributions of paired mean recovery-error and support-F1 differences across settings; each point represents one setting. Negative recovery differences and positive support differences favor $\ell_1$-SDBPG.}
\label{fig:sparse-recovery-coverage}
\end{figure}
\FloatBarrier

\subsection{Shared-Support Composite Recovery}
The group $\ell_{2,1}$ experiment considers shared row support across multiple signals rather than coordinate sparsity. Its block soft-thresholding update is the second structured case covered by Proposition~\ref{prop:stability-sufficient}.

\begin{table}[!htbp]
\centering
\caption{Shared-support composite recovery with the convex $\ell_{2,1}$ penalty and true active-group counts 5 and 4 in the undersampled and heterogeneous regimes, respectively. Bold entries mark the lowest recovery error and highest support F1.}
\label{tab:group-sparse-main}

\scriptsize

\resizebox{\textwidth}{!}{%
\begin{tabular}{lcccccc}
\toprule
& \multicolumn{3}{c}{Noisy undersampled} & \multicolumn{3}{c}{Noisy heterogeneous}\\
\cmidrule(lr){2-4}\cmidrule(lr){5-7}
Algorithm & Recovery error & Support F1 & Support size & Recovery error & Support F1 & Support size\\
\midrule
$\ell_{2,1}$-SDBPG & \best{0.534} & \best{0.829} & 3.6 & \best{0.380} & \best{0.986} & 3.9\\
SDBPG & 0.662 & 0.667 & 9.9 & 0.623 & 0.630 & 8.9\\
DGD & 1.021 & 0.381 & 18.0 & 0.961 & 0.333 & 17.2\\
DGT & 0.872 & 0.496 & 14.3 & 0.817 & 0.459 & 13.2\\
\bottomrule
\end{tabular}}
\end{table}
\FloatBarrier

In both regimes, $\ell_{2,1}$-SDBPG has the lowest recovery error, the highest support F1, and a support size closest to the true one. Relative to unregularized SDBPG, its paired recovery improvements are $0.128$ (95\% CI $[0.094,0.156]$) and $0.243$ (95\% CI $[0.210,0.276]$), with support-F1 gains of $0.162$ (95\% CI $[0.060,0.263]$) and $0.356$ (95\% CI $[0.295,0.414]$). The $\ell_{2,1}$ results in Figure~\ref{fig:defect-identification} show that the observed group support stabilizes before the defects vanish.

\subsection{Geometry-Matched Non-Euclidean Phase Retrieval}\label{sec:quartic-phase-retrieval}
The sparse-recovery studies above use a quadratic kernel. For phase retrieval, we instead use a kernel with the same quartic growth as the loss:
\begin{equation}\label{eq:quartic-kernel-experiment}
\phi_4(x)=\frac14\|x\|^4+\frac12\|x\|^2,
\qquad
\nabla\phi_4(x)=(1+\|x\|^2)x.
\end{equation}
The resulting mirror map $\nabla\phi_4$ is nonlinear and cannot be reduced to a weighted Euclidean transformation. Its inverse and the SDBPG smoothing subproblem are monotone radial cubic equations that can be solved exactly. Appendix~\ref{app:quartic-kernel-details} gives the formulas, parameters, and detailed results.

We use six agents, dimension $d=16$, 60 quadratic measurements per agent, observation-noise standard deviation $0.02$, and a ring with two chords. For heterogeneity factor $h$, the Gaussian sensing vectors at agent $i\in\{0,\ldots,N-1\}$ are multiplied by $1+h i/(N-1)$; we evaluate $h=2.5$, $3.0$, and $3.25$. Quadratic SDBPG, quartic SDBPG, and deterministic ProxGT are tuned on five validation instances at $h=3.0$; the parameters are then kept unchanged on ten held-out evaluation instances at all three levels for 600 iterations.

At $h=3.0$, quartic SDBPG lowers mean recovery error from $0.01505$ to $0.01197$ relative to quadratic SDBPG and from $0.02452$ relative to ProxGT. The paired quartic-minus-quadratic difference is $-0.00308$ with 95\% bootstrap interval $[-0.00438,-0.00184]$, while the difference from ProxGT is $-0.01254$ with interval $[-0.02179,-0.00413]$; the quartic kernel gives the lower error on nine of ten evaluation instances in both comparisons. At $h=3.25$, its mean error is $0.01113$, compared with $0.01384$ for quadratic SDBPG. At $h=2.5$, the quadratic kernel gives the lower mean, $0.01753$ versus $0.02056$.

Appendix~\ref{app:sparse-phase-retrieval} also considers the quartic phase-retrieval loss with nonsmooth $\ell_1$ regularization. Table~\ref{tab:sparse-phase-main} shows that $\ell_1$-SDBPG has the lowest mean recovery errors, $0.864$ and $0.904$, in the two noisy regimes and produces much smaller supports than the unregularized methods. ProxDGT has the highest support F1, so no method is best on both recovery error and support F1.

\begin{figure}[!htbp]
\centering
\includegraphics[width=0.96\textwidth]{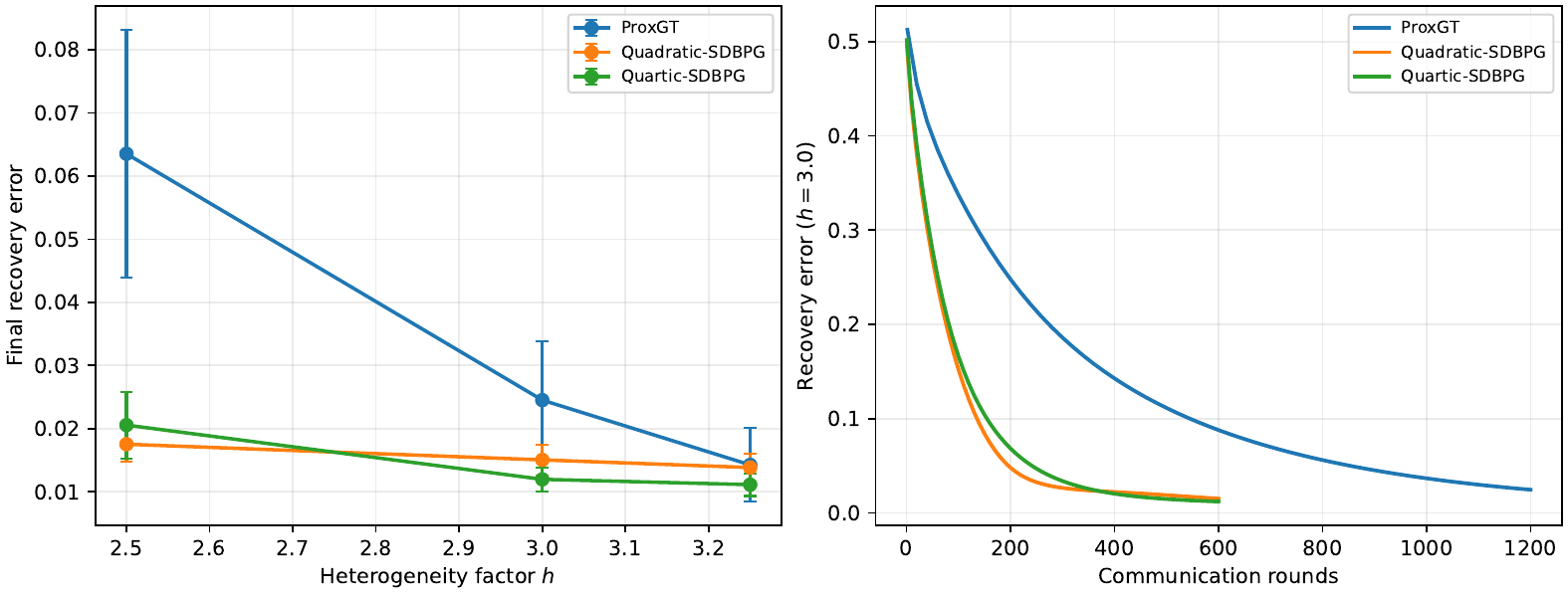}
\caption{Geometry-matched phase retrieval. Left: final recovery error with parameters kept unchanged across the heterogeneity-factor sweep; bars show normal-approximation 95\% intervals over ten evaluation instances. Right: mean recovery trajectories versus sequential communication rounds at $h=3.0$.}
\label{fig:quartic-geometry-modern}
\end{figure}
\FloatBarrier

\subsection{Structured Fourier Imaging}\label{sec:structured-fourier-imaging}
We also consider structured Fourier sensing. Let $W^\top$ synthesize an image from orthonormal wavelet coefficients, $\mathcal F$ be the orthonormal two-dimensional Fourier transform, and $P_\Omega$ select the measured frequencies. The consensus problem is equivalent to
\begin{equation}
\min_x\ \frac12\big\|P_\Omega\mathcal F W^\top x-b\big\|^2+\lambda_r\|x\|_1,
\label{eq:structured-fourier-imaging}
\end{equation}
where $\lambda_r$ is the regularization weight. Disjoint subsets of $\Omega$ and their observations are assigned to the agents. This convex nonsmooth problem replaces the Gaussian designs used above with a structured complex-valued operator.

We reconstruct $64\times64$ modified Shepp--Logan phantoms from $35\%$ variable-density Fourier samples, including a fully sampled central width of $8\%$ on each axis, with $1\%$ relative complex Gaussian noise. The representation uses level-3 orthonormal Daubechies-2 (db2) wavelets. Six agents communicate on a ring with two chords, start from zero, and run for 1000 iterations. ProxGT uses $K=2$ consensus steps in each of its gradient-tracker and proximal-point mixing blocks, hence four sequential communication rounds per iteration. All methods use the same $\ell_1$ weight $0.01$. Parameters are selected on three validation instances and kept unchanged on ten held-out evaluation instances. We report normalized root-mean-square error (NRMSE), peak signal-to-noise ratio (PSNR), and structural similarity index (SSIM). Appendix~\ref{app:structured-fourier-imaging} reports a separate extended benchmark on independently generated masks and noise draws, adding a proximal primal--dual baseline (ProxPDA), Prox-EXTRA, the centralized fast iterative shrinkage-thresholding algorithm (FISTA), and Cameraman images.

\begin{table}[!htbp]
\centering
\caption{Structured-Fourier comparison over ten evaluation phantom masks. Bold entries mark the best displayed reconstruction means and lowest costs to target. The target is $1.10$ times the lowest final NRMSE attained by a compared method on the same instance.}
\label{tab:structured-fourier-modern}

\footnotesize
\setlength{\tabcolsep}{2pt}

\begin{subtable}[t]{0.48\textwidth}
\centering
\caption{Final reconstruction.}
\begin{tabular}{lrrr}
\toprule
Algorithm & NRMSE $\downarrow$ & \shortstack{PSNR $\uparrow$\\(dB)} & SSIM $\uparrow$\\
\midrule
$\ell_1$-SDBPG & \best{0.17250} & \best{27.393} & \best{0.84058}\\
ProxGT ($K=2$) & 0.17349 & 27.344 & 0.83852\\
ProxDGT & 0.17905 & 27.070 & 0.82589\\
\bottomrule
\end{tabular}
\end{subtable}\hfill
\begin{subtable}[t]{0.50\textwidth}
\centering
\caption{Cost to the common target.}
\begin{tabular}{lrrr}
\toprule
Algorithm & Rounds $\downarrow$ & MiB $\downarrow$ & \shortstack{Oracle\\samples $\downarrow$}\\
\midrule
$\ell_1$-SDBPG & \best{407.5} & \best{407.5} & \best{2445}\\
ProxGT ($K=2$) & 1870 & 935 & 2805\\
ProxDGT & 695 & 695 & 4176\\
\bottomrule
\end{tabular}
\end{subtable}
\end{table}

The paired NRMSE difference between $\ell_1$-SDBPG and ProxGT is $-9.83\times10^{-4}$ with 95\% bootstrap interval $[-1.17\times10^{-3},-7.64\times10^{-4}]$, and $\ell_1$-SDBPG gives the lower NRMSE on all ten instances. At the full budget, $\ell_1$-SDBPG and ProxDGT each use 1000 sequential communication rounds and 1000 MiB; ProxGT uses 4000 rounds and 2000 MiB. Relative to ProxGT, $\ell_1$-SDBPG reaches the prescribed reconstruction level with $78.2\%$ fewer rounds, $56.4\%$ fewer transmitted MiB, and $12.8\%$ fewer oracle samples.

\FloatBarrier

\subsection{Stochastic Sparse Recovery}
For stochastic sparse recovery, we compare $\ell_1$-SDBPG with ProxGT-SA and Prox-DASA-GT \citep{xin2021stochastic,xiao2023one}. All methods use mini-batches of size four, a common $\ell_1$ weight of $0.02$, and 900 iterations. Parameters are selected on three validation instances and kept unchanged on ten held-out evaluation instances.

\begin{table}[!htbp]
\centering
\caption{Stochastic sparse recovery over ten evaluation instances with mini-batch size four and 900 iterations. Bold entries mark the best displayed final-quality means; paired intervals are given in the text.}
\label{tab:modern-stochastic-main}

\footnotesize
\setlength{\tabcolsep}{2pt}

\begin{subtable}[t]{0.58\textwidth}
\centering
\caption{Final quality.}
\begin{tabular}{lrrrr}
\toprule
Algorithm & \shortstack{Recovery\\error $\downarrow$} & \shortstack{Support\\F1 $\uparrow$} & \shortstack{Squared\\stationarity\\residual $\downarrow$} & \shortstack{Consensus\\error $\downarrow$}\\
\midrule
$\ell_1$-SDBPG & 0.6529 & 0.5536 & $9.26\times10^{-4}$ & \best{$4.13\times10^{-4}$}\\
ProxGT-SA & \best{0.6510} & 0.5714 & $3.38\times10^{-3}$ & $1.27\times10^{-3}$\\
Prox-DASA-GT & 0.6654 & \best{0.5853} & \best{$3.28\times10^{-4}$} & $7.95\times10^{-4}$\\
\bottomrule
\end{tabular}
\end{subtable}\hfill
\begin{subtable}[t]{0.40\textwidth}
\centering
\caption{Fixed-budget costs.}
\begin{tabular}{lrrr}
\toprule
Algorithm & Rounds & MiB & \shortstack{Oracle\\samples\\\strut}\\
\midrule
$\ell_1$-SDBPG & 900 & 13.733 & 28800\\
ProxGT-SA & 3600 & 27.466 & 28800\\
Prox-DASA-GT & 900 & 20.599 & 28832\\
\bottomrule
\end{tabular}
\end{subtable}
\end{table}

The paired recovery-error differences ($\ell_1$-SDBPG minus baseline) are $0.00189$ against ProxGT-SA and $-0.01252$ against Prox-DASA-GT. Their respective 95\% intervals are $[-0.01771,0.01987]$ and $[-0.03698,0.01205]$. $\ell_1$-SDBPG uses one quarter of the communication rounds and one half of the transmitted MiB of ProxGT-SA. Relative to Prox-DASA-GT, it uses the same number of rounds and $33.3\%$ fewer transmitted MiB. It also gives the smallest consensus error. The cost-to-target results in Appendix~\ref{app:modern-baseline-costs} show that ProxGT-SA and Prox-DASA-GT reach the relaxed $110\%$ recovery target earlier.

Figure~\ref{fig:stochastic-noise-alignment} varies the mini-batch size and the variance of additive gradient noise. Larger mini-batches reduce both the mean squared oracle error and $e^t$, whereas stronger additive noise increases both. This agrees with the oracle-variance term in the finite-horizon bound.

\begin{figure}[!htbp]
\centering
\includegraphics[width=0.95\textwidth]{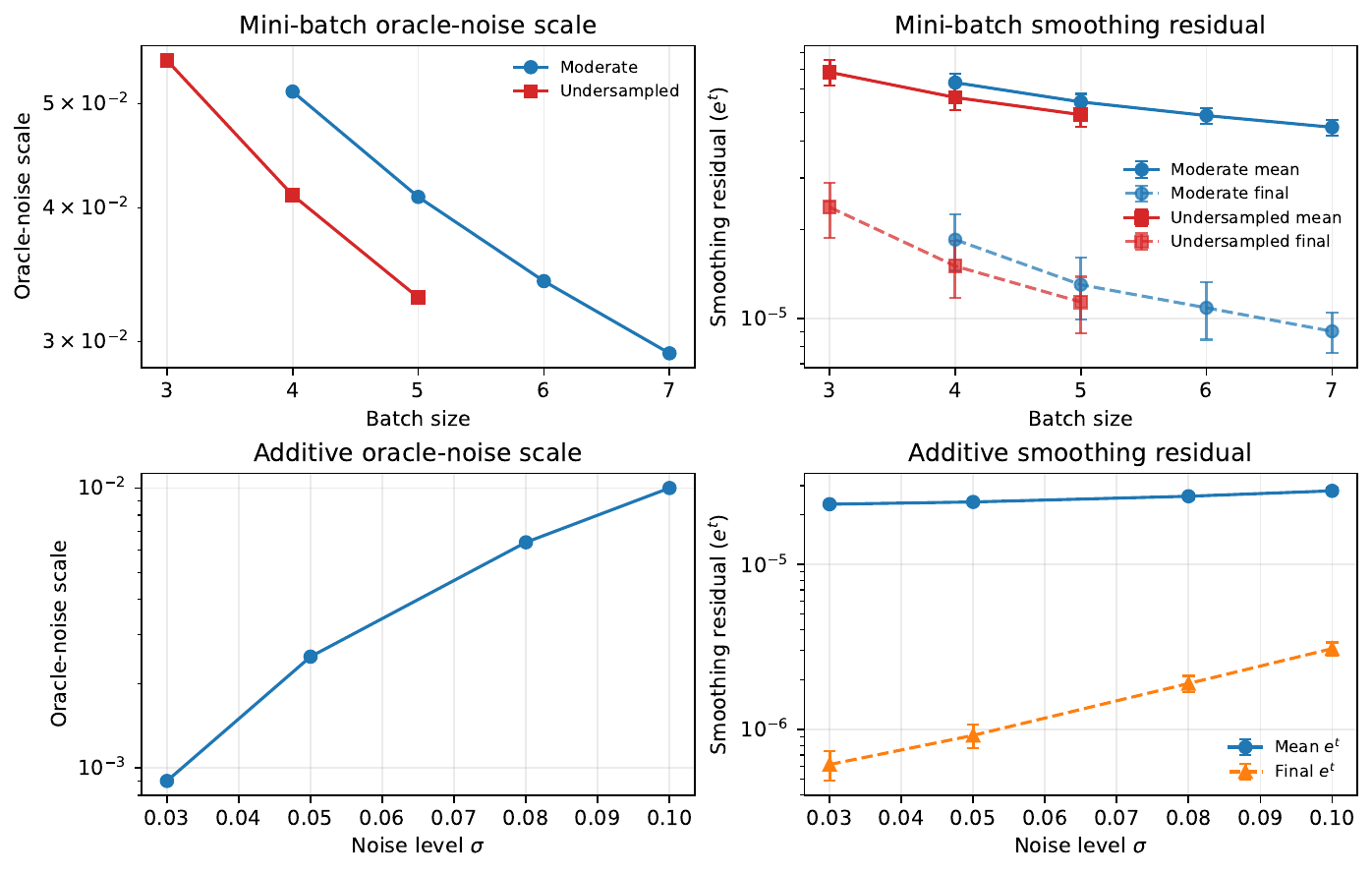}
\caption{Stochastic sparse-recovery oracle-error alignment for
$\ell_1$-SDBPG. Mini-batch panels vary the batch size, and additive
panels vary the injected noise level. Smoothing-residual panels report time-averaged and final $e^t$ over 900 iterations, using the same exact-update configuration and 20 evaluation instances as Figure~\ref{fig:stochastic-residual-trajectories}. Error bars show normal-approximation 95\% confidence intervals.}
\label{fig:stochastic-noise-alignment}
\end{figure}
\FloatBarrier

Appendix~\ref{app:stochastic-sparse-extra} adds unregularized stochastic SDBPG and ProxDGD. Across the reported oracle settings, stochastic $\ell_1$-SDBPG has lower recovery error and higher support F1 than unregularized SDBPG. It also improves both metrics over ProxDGD in the four matched moderate-regime comparisons. The fixed-batch and additive-noise experiments characterize residuals under persistent oracle error, while the increasing-batch experiment traces the transition to exact local gradients.

\section{Conclusion}
We developed composite and stochastic SDBPG methods for node-specific convex regularizers and unbiased local gradients while retaining the Bregman primal--dual smoothing updates. The analysis separates temporal subgradient motion, aggregation mismatch, and oracle variance. Under compact localization and admissible parameter conditions, summable defects give an $O(T^{-1})$ bound on the squared best-iterate proximal residual. The stochastic method has the same order in expectation when the expected defects and oracle variance are summable, and a variance-dependent residual bound under persistent noise. Common active-structure identification and the stated local subgradient conditions provide defect verification through primal optimality or local subgradient selections.

Experiments connect active-structure stabilization with defect decay and demonstrate the benefits of coordinate and group regularization for recovery under heterogeneous sensing. Geometry-matched phase retrieval illustrates how kernel choice influences reconstruction. Partial-Fourier imaging combines improved phantom reconstruction with lower communication cost against gradient-tracking baselines, while the leading proximal methods give similar Cameraman reconstruction quality. Stochastic recovery achieves accuracy comparable to recent proximal gradient-tracking methods with lower transmitted volume at the fixed iteration budget. The cost-to-accuracy comparisons further characterize the different transient behaviors of these methods.

\bibliographystyle{abbrvnat}
\bibliography{references}

\begin{thebibliography}{74}
\providecommand{\natexlab}[1]{#1}
\providecommand{\url}[1]{\texttt{#1}}
\expandafter\ifx\csname urlstyle\endcsname\relax
  \providecommand{\doi}[1]{doi: #1}\else
  \providecommand{\doi}{doi: \begingroup \urlstyle{rm}\Url}\fi

\bibitem[ATTOUCH(2011)]{attouch2011convergence}
H.~ATTOUCH.
\newblock Convergence of descent methods for semi-algebraic and tame problems.
\newblock 2011.

\bibitem[Auslender and Teboulle(2006)]{auslender2006interior}
A.~Auslender and M.~Teboulle.
\newblock Interior gradient and proximal methods for convex and conic
  optimization.
\newblock \emph{SIAM Journal on Optimization}, 16\penalty0 (3):\penalty0
  697--725, 2006.

\bibitem[Bauschke et~al.(2017)Bauschke, Bolte, and
  Teboulle]{bauschke2017descent}
H.~H. Bauschke, J.~Bolte, and M.~Teboulle.
\newblock A descent lemma beyond lipschitz gradient continuity: first-order
  methods revisited and applications.
\newblock \emph{Mathematics of Operations Research}, 42\penalty0 (2):\penalty0
  330--348, 2017.

\bibitem[Beck and Teboulle(2003)]{beck2003mirror}
A.~Beck and M.~Teboulle.
\newblock Mirror descent and nonlinear projected subgradient methods for convex
  optimization.
\newblock \emph{Operations Research Letters}, 31\penalty0 (3):\penalty0
  167--175, 2003.

\bibitem[Beck and Teboulle(2009)]{beck2009fast}
A.~Beck and M.~Teboulle.
\newblock A fast iterative shrinkage-thresholding algorithm for linear inverse
  problems.
\newblock \emph{SIAM journal on imaging sciences}, 2\penalty0 (1):\penalty0
  183--202, 2009.

\bibitem[Bolte et~al.(2014)Bolte, Sabach, and Teboulle]{bolte2014proximal}
J.~Bolte, S.~Sabach, and M.~Teboulle.
\newblock Proximal alternating linearized minimization for nonconvex and
  nonsmooth problems.
\newblock \emph{Mathematical Programming}, 146\penalty0 (1):\penalty0 459--494,
  2014.

\bibitem[Bolte et~al.(2018)Bolte, Sabach, Teboulle, and
  Vaisbourd]{bolte2018first}
J.~Bolte, S.~Sabach, M.~Teboulle, and Y.~Vaisbourd.
\newblock First order methods beyond convexity and lipschitz gradient
  continuity with applications to quadratic inverse problems.
\newblock \emph{SIAM Journal on Optimization}, 28\penalty0 (3):\penalty0
  2131--2151, 2018.

\bibitem[Bottou et~al.(2018)Bottou, Curtis, and
  Nocedal]{bottou2018optimization}
L.~Bottou, F.~E. Curtis, and J.~Nocedal.
\newblock Optimization methods for large-scale machine learning.
\newblock \emph{SIAM review}, 60\penalty0 (2):\penalty0 223--311, 2018.

\bibitem[Bregman(1967)]{bregman1967relaxation}
L.~M. Bregman.
\newblock The relaxation method of finding the common point of convex sets and
  its application to the solution of problems in convex programming.
\newblock \emph{USSR computational mathematics and mathematical physics},
  7\penalty0 (3):\penalty0 200--217, 1967.

\bibitem[Candes and Recht(2012)]{candes2012exact}
E.~Candes and B.~Recht.
\newblock Exact matrix completion via convex optimization.
\newblock \emph{Communications of the ACM}, 55\penalty0 (6):\penalty0 111--119,
  2012.

\bibitem[Candes et~al.(2015)Candes, Li, and Soltanolkotabi]{candes2015phase}
E.~J. Candes, X.~Li, and M.~Soltanolkotabi.
\newblock Phase retrieval via wirtinger flow: Theory and algorithms.
\newblock \emph{IEEE Transactions on Information Theory}, 61\penalty0
  (4):\penalty0 1985--2007, 2015.

\bibitem[Chen and Teboulle(1993)]{chen1993convergence}
G.~Chen and M.~Teboulle.
\newblock Convergence analysis of a proximal-like minimization algorithm using
  bregman functions.
\newblock \emph{SIAM Journal on Optimization}, 3\penalty0 (3):\penalty0
  538--543, 1993.

\bibitem[Chen et~al.(2026)Chen, Kuruzov, and Scutari]{chen2026adaptive}
X.~Chen, I.~Kuruzov, and G.~Scutari.
\newblock Adaptive decentralized composite optimization via three-operator
  splitting.
\newblock \emph{arXiv preprint arXiv:2602.17545}, 2026.

\bibitem[Chi et~al.(2019)Chi, Lu, and Chen]{chi2019nonconvex}
Y.~Chi, Y.~M. Lu, and Y.~Chen.
\newblock Nonconvex optimization meets low-rank matrix factorization: An
  overview.
\newblock \emph{IEEE Transactions on Signal Processing}, 67\penalty0
  (20):\penalty0 5239--5269, 2019.

\bibitem[Chraibi et~al.(2025)Chraibi, Iutzeler, Malick, and
  Rogozin]{chraibi2025delay}
S.~Chraibi, F.~Iutzeler, J.~Malick, and A.~Rogozin.
\newblock Delay-tolerant distributed bregman proximal algorithms.
\newblock \emph{Optimization Methods and Software}, 40\penalty0 (5):\penalty0
  1072--1088, 2025.

\bibitem[Combettes and Pesquet(2011)]{combettes2011proximal}
P.~L. Combettes and J.-C. Pesquet.
\newblock Proximal splitting methods in signal processing.
\newblock In \emph{Fixed-point algorithms for inverse problems in science and
  engineering}, pages 185--212. Springer, 2011.

\bibitem[Dai et~al.(2023)Dai, Wang, Curtis, and Robinson]{dai2023variance}
Y.~Dai, G.~Wang, F.~E. Curtis, and D.~P. Robinson.
\newblock A variance-reduced and stabilized proximal stochastic gradient method
  with support identification guarantees for structured optimization.
\newblock In \emph{International Conference on Artificial Intelligence and
  Statistics}, pages 5107--5133. PMLR, 2023.

\bibitem[Davis and Drusvyatskiy(2019)]{davis2019stochastic}
D.~Davis and D.~Drusvyatskiy.
\newblock Stochastic model-based minimization of weakly convex functions.
\newblock \emph{SIAM Journal on Optimization}, 29\penalty0 (1):\penalty0
  207--239, 2019.

\bibitem[Di~Lorenzo and Scutari(2016)]{di2016next}
P.~Di~Lorenzo and G.~Scutari.
\newblock Next: In-network nonconvex optimization.
\newblock \emph{IEEE Transactions on Signal and Information Processing over
  Networks}, 2\penalty0 (2):\penalty0 120--136, 2016.

\bibitem[Ding et~al.(2025)Ding, Li, and Toh]{ding2025nonconvex}
K.~Ding, J.~Li, and K.-C. Toh.
\newblock Nonconvex stochastic bregman proximal gradient method with
  application to deep learning.
\newblock \emph{Journal of Machine Learning Research}, 26\penalty0
  (39):\penalty0 1--44, 2025.

\bibitem[Donoho(2006)]{donoho2006compressed}
D.~L. Donoho.
\newblock Compressed sensing.
\newblock \emph{IEEE Transactions on information theory}, 52\penalty0
  (4):\penalty0 1289--1306, 2006.

\bibitem[Duchi et~al.(2011)Duchi, Agarwal, and Wainwright]{duchi2011dual}
J.~C. Duchi, A.~Agarwal, and M.~J. Wainwright.
\newblock Dual averaging for distributed optimization: Convergence analysis and
  network scaling.
\newblock \emph{IEEE Transactions on Automatic control}, 57\penalty0
  (3):\penalty0 592--606, 2011.

\bibitem[Eckstein(1993)]{eckstein1993nonlinear}
J.~Eckstein.
\newblock Nonlinear proximal point algorithms using bregman functions, with
  applications to convex programming.
\newblock \emph{Mathematics of Operations Research}, 18\penalty0 (1):\penalty0
  202--226, 1993.

\bibitem[Eldar and Mendelson(2014)]{eldar2014phase}
Y.~C. Eldar and S.~Mendelson.
\newblock Phase retrieval: Stability and recovery guarantees.
\newblock \emph{Applied and Computational Harmonic Analysis}, 36\penalty0
  (3):\penalty0 473--494, 2014.

\bibitem[Ghadimi and Lan(2013)]{ghadimi2013stochastic}
S.~Ghadimi and G.~Lan.
\newblock Stochastic first-and zeroth-order methods for nonconvex stochastic
  programming.
\newblock \emph{SIAM journal on optimization}, 23\penalty0 (4):\penalty0
  2341--2368, 2013.

\bibitem[Guo et~al.(2023)Guo, Shi, Cao, and Wang]{guo2023decentralized}
L.~Guo, X.~Shi, J.~Cao, and Z.~Wang.
\newblock Decentralized inexact proximal gradient method with
  network-independent stepsizes for convex composite optimization.
\newblock \emph{IEEE Transactions on Signal Processing}, 71:\penalty0 786--801,
  2023.

\bibitem[Guo et~al.(2025)Guo, Wang, Shi, and Cao]{guo2025proximal}
L.~Guo, L.~Wang, X.~Shi, and J.~Cao.
\newblock A proximal gradient method with probabilistic multi-gossip
  communications for decentralized composite optimization.
\newblock \emph{IEEE Transactions on Signal and Information Processing over
  Networks}, 2025.

\bibitem[Guo et~al.(2026)Guo, Shi, Xu, and Cao]{guo2026local}
L.~Guo, X.~Shi, W.~Xu, and J.~Cao.
\newblock Local adapt-then-combine algorithms for distributed nonsmooth
  optimization: Achieving provable communication acceleration.
\newblock \emph{Automatica}, 188:\penalty0 112965, 2026.

\bibitem[Hong et~al.(2017)Hong, Hajinezhad, and Zhao]{hong2017prox}
M.~Hong, D.~Hajinezhad, and M.-M. Zhao.
\newblock Prox-pda: The proximal primal-dual algorithm for fast distributed
  nonconvex optimization and learning over networks.
\newblock In \emph{International Conference on Machine Learning}, pages
  1529--1538. PMLR, 2017.

\bibitem[Huang et~al.(2024)Huang, Pu, and Nedi{\'c}]{huang2024distributed}
K.~Huang, S.~Pu, and A.~Nedi{\'c}.
\newblock Distributed normal map-based stochastic proximal gradient methods
  over networks.
\newblock \emph{arXiv preprint arXiv:2412.13054}, 2024.

\bibitem[Jegnell and Vlaski(2022)]{jegnell2022distributed}
S.~Jegnell and S.~Vlaski.
\newblock Distributed relatively smooth optimization.
\newblock In \emph{2022 IEEE 61st Conference on Decision and Control (CDC)},
  pages 6511--6517. IEEE, 2022.

\bibitem[Kiwiel(1997)]{kiwiel1997proximal}
K.~C. Kiwiel.
\newblock Proximal minimization methods with generalized bregman functions.
\newblock \emph{SIAM journal on control and optimization}, 35\penalty0
  (4):\penalty0 1142--1168, 1997.

\bibitem[Koloskova et~al.(2020)Koloskova, Loizou, Boreiri, Jaggi, and
  Stich]{koloskova2020unified}
A.~Koloskova, N.~Loizou, S.~Boreiri, M.~Jaggi, and S.~Stich.
\newblock A unified theory of decentralized sgd with changing topology and
  local updates.
\newblock In \emph{International conference on machine learning}, pages
  5381--5393. PMLR, 2020.

\bibitem[Lan(2012)]{lan2012optimal}
G.~Lan.
\newblock An optimal method for stochastic composite optimization.
\newblock \emph{Mathematical Programming}, 133\penalty0 (1):\penalty0 365--397,
  2012.

\bibitem[Lewis(2002)]{lewis2002active}
A.~S. Lewis.
\newblock Active sets, nonsmoothness, and sensitivity.
\newblock \emph{SIAM Journal on Optimization}, 13\penalty0 (3):\penalty0
  702--725, 2002.

\bibitem[Lian et~al.(2017)Lian, Zhang, Zhang, Hsieh, Zhang, and
  Liu]{lian2017can}
X.~Lian, C.~Zhang, H.~Zhang, C.-J. Hsieh, W.~Zhang, and J.~Liu.
\newblock Can decentralized algorithms outperform centralized algorithms? a
  case study for decentralized parallel stochastic gradient descent.
\newblock \emph{Advances in neural information processing systems}, 30, 2017.

\bibitem[Liang et~al.(2017)Liang, Fadili, and Peyr{\'e}]{liang2017activity}
J.~Liang, J.~Fadili, and G.~Peyr{\'e}.
\newblock Activity identification and local linear convergence of
  forward--backward-type methods.
\newblock \emph{SIAM Journal on Optimization}, 27\penalty0 (1):\penalty0
  408--437, 2017.

\bibitem[Liao et~al.(2025)Liao, Liu, Han, and Guo]{liao2025momentum}
S.~Liao, Y.~Liu, C.~Han, and T.~Guo.
\newblock Momentum-based variance-reduced stochastic bregman proximal gradient
  methods for nonconvex nonsmooth optimization.
\newblock \emph{Expert Systems with Applications}, 266:\penalty0 125960, 2025.

\bibitem[Lu et~al.(2018)Lu, Freund, and Nesterov]{lu2018relatively}
H.~Lu, R.~M. Freund, and Y.~Nesterov.
\newblock Relatively smooth convex optimization by first-order methods, and
  applications.
\newblock \emph{SIAM Journal on Optimization}, 28\penalty0 (1):\penalty0
  333--354, 2018.

\bibitem[Neal et~al.(2011)Neal, Eric, Borja, and Jonathan]{neal2011distributed}
P.~Neal, C.~Eric, P.~Borja, and E.~Jonathan.
\newblock Distributed optimization and statistical learning via the alternating
  direction method of multipliers.
\newblock \emph{Foundations and Trends{\textregistered} in Machine learning},
  3\penalty0 (1):\penalty0 1--122, 2011.

\bibitem[Nedi{\'c} and Olshevsky(2014)]{nedic2014distributed}
A.~Nedi{\'c} and A.~Olshevsky.
\newblock Distributed optimization over time-varying directed graphs.
\newblock \emph{IEEE Transactions on Automatic Control}, 60\penalty0
  (3):\penalty0 601--615, 2014.

\bibitem[Nedic and Ozdaglar(2009)]{nedic2009distributed}
A.~Nedic and A.~Ozdaglar.
\newblock Distributed subgradient methods for multi-agent optimization.
\newblock \emph{IEEE Transactions on automatic control}, 54\penalty0
  (1):\penalty0 48--61, 2009.

\bibitem[Nedic et~al.(2017)Nedic, Olshevsky, and Shi]{nedic2017achieving}
A.~Nedic, A.~Olshevsky, and W.~Shi.
\newblock Achieving geometric convergence for distributed optimization over
  time-varying graphs.
\newblock \emph{SIAM Journal on Optimization}, 27\penalty0 (4):\penalty0
  2597--2633, 2017.

\bibitem[Nemirovski et~al.(2009)Nemirovski, Juditsky, Lan, and
  Shapiro]{nemirovski2009robust}
A.~Nemirovski, A.~Juditsky, G.~Lan, and A.~Shapiro.
\newblock Robust stochastic approximation approach to stochastic programming.
\newblock \emph{SIAM Journal on optimization}, 19\penalty0 (4):\penalty0
  1574--1609, 2009.

\bibitem[Parikh and Boyd(2014)]{parikh2014proximal}
N.~Parikh and S.~Boyd.
\newblock Proximal algorithms.
\newblock \emph{Foundations and Trends in optimization}, 1\penalty0
  (3):\penalty0 127--239, 2014.

\bibitem[Poon et~al.(2018)Poon, Liang, and Schoenlieb]{poon2018local}
C.~Poon, J.~Liang, and C.~Schoenlieb.
\newblock Local convergence properties of saga/prox-svrg and acceleration.
\newblock In \emph{International Conference on Machine Learning}, pages
  4124--4132. PMLR, 2018.

\bibitem[Pu et~al.(2024)Pu, Zhang, Zhou, Fu, and Hong]{pu2024smoothed}
W.~Pu, J.~Zhang, R.~Zhou, X.~Fu, and M.~Hong.
\newblock A smoothed bregman proximal gradient algorithm for decentralized
  nonconvex optimization.
\newblock In \emph{ICASSP 2024-2024 IEEE International Conference on Acoustics,
  Speech and Signal Processing (ICASSP)}, pages 8911--8915. IEEE, 2024.

\bibitem[Qiu et~al.(2026)Qiu, Jiang, and Milzarek]{qiu2026normal}
J.~Qiu, L.~Jiang, and A.~Milzarek.
\newblock A normal map-based proximal stochastic gradient method: Convergence
  and identification properties.
\newblock \emph{SIAM Journal on Optimization}, 36\penalty0 (3):\penalty0
  1773--1804, 2026.

\bibitem[Qu and Li(2017)]{qu2017harnessing}
G.~Qu and N.~Li.
\newblock Harnessing smoothness to accelerate distributed optimization.
\newblock \emph{IEEE Transactions on Control of Network Systems}, 5\penalty0
  (3):\penalty0 1245--1260, 2017.

\bibitem[Recht et~al.(2010)Recht, Fazel, and Parrilo]{recht2010guaranteed}
B.~Recht, M.~Fazel, and P.~A. Parrilo.
\newblock Guaranteed minimum-rank solutions of linear matrix equations via
  nuclear norm minimization.
\newblock \emph{SIAM review}, 52\penalty0 (3):\penalty0 471--501, 2010.

\bibitem[Robbins and Monro(1951)]{robbins1951stochastic}
H.~Robbins and S.~Monro.
\newblock A stochastic approximation method.
\newblock \emph{The annals of mathematical statistics}, pages 400--407, 1951.

\bibitem[Rockafellar and Wets(1998)]{rockafellar1998variational}
R.~T. Rockafellar and R.~J. Wets.
\newblock \emph{Variational analysis}.
\newblock Springer, 1998.

\bibitem[Scutari and Sun(2018)]{scutari2018distributed}
G.~Scutari and Y.~Sun.
\newblock Distributed nonconvex constrained optimization over time-varying
  digraphs.
\newblock \emph{arXiv preprint arXiv:1809.01106}, 2018.

\bibitem[Shi et~al.(2015{\natexlab{a}})Shi, Ling, Wu, and Yin]{shi2015extra}
W.~Shi, Q.~Ling, G.~Wu, and W.~Yin.
\newblock Extra: An exact first-order algorithm for decentralized consensus
  optimization.
\newblock \emph{SIAM Journal on Optimization}, 25\penalty0 (2):\penalty0
  944--966, 2015{\natexlab{a}}.

\bibitem[Shi et~al.(2015{\natexlab{b}})Shi, Ling, Wu, and Yin]{shi2015proximal}
W.~Shi, Q.~Ling, G.~Wu, and W.~Yin.
\newblock A proximal gradient algorithm for decentralized composite
  optimization.
\newblock \emph{IEEE Transactions on Signal Processing}, 63\penalty0
  (22):\penalty0 6013--6023, 2015{\natexlab{b}}.

\bibitem[Sun and Hong(2019)]{sun2019distributed}
H.~Sun and M.~Hong.
\newblock Distributed non-convex first-order optimization and information
  processing: Lower complexity bounds and rate optimal algorithms.
\newblock \emph{IEEE Transactions on Signal processing}, 67\penalty0
  (22):\penalty0 5912--5928, 2019.

\bibitem[Sundhar~Ram et~al.(2010)Sundhar~Ram, Nedi{\'c}, and
  Veeravalli]{sundhar2010distributed}
S.~Sundhar~Ram, A.~Nedi{\'c}, and V.~V. Veeravalli.
\newblock Distributed stochastic subgradient projection algorithms for convex
  optimization.
\newblock \emph{Journal of optimization theory and applications}, 147\penalty0
  (3):\penalty0 516--545, 2010.

\bibitem[Takahashi and Takeda(2025)]{takahashi2025approximate}
S.~Takahashi and A.~Takeda.
\newblock Approximate bregman proximal gradient algorithm for relatively smooth
  nonconvex optimization.
\newblock \emph{Computational Optimization and Applications}, 90\penalty0
  (1):\penalty0 227--256, 2025.

\bibitem[Tibshirani(1996)]{tibshirani1996regression}
R.~Tibshirani.
\newblock Regression shrinkage and selection via the lasso.
\newblock \emph{Journal of the Royal Statistical Society Series B: Statistical
  Methodology}, 58\penalty0 (1):\penalty0 267--288, 1996.

\bibitem[Tsitsiklis et~al.(1986)Tsitsiklis, Bertsekas, and
  Athans]{tsitsiklis1986distributed}
J.~Tsitsiklis, D.~Bertsekas, and M.~Athans.
\newblock Distributed asynchronous deterministic and stochastic gradient
  optimization algorithms.
\newblock \emph{IEEE transactions on automatic control}, 31\penalty0
  (9):\penalty0 803--812, 1986.

\bibitem[Tu et~al.(2016)Tu, Boczar, Simchowitz, Soltanolkotabi, and
  Recht]{tu2016low}
S.~Tu, R.~Boczar, M.~Simchowitz, M.~Soltanolkotabi, and B.~Recht.
\newblock Low-rank solutions of linear matrix equations via procrustes flow.
\newblock In \emph{International conference on machine learning}, pages
  964--973. PMLR, 2016.

\bibitem[Wai et~al.(2017)Wai, Lafond, Scaglione, and
  Moulines]{wai2017decentralized}
H.-T. Wai, J.~Lafond, A.~Scaglione, and E.~Moulines.
\newblock Decentralized frank--wolfe algorithm for convex and nonconvex
  problems.
\newblock \emph{IEEE Transactions on Automatic Control}, 62\penalty0
  (11):\penalty0 5522--5537, 2017.

\bibitem[Wang et~al.(2025)Wang, Bao, and Liu]{wang2025decentralized}
L.~Wang, L.~Bao, and X.~Liu.
\newblock A decentralized proximal gradient tracking algorithm for composite
  optimization on riemannian manifolds.
\newblock \emph{Journal of Machine Learning Research}, 26\penalty0
  (106):\penalty0 1--37, 2025.

\bibitem[Wright et~al.(2009)Wright, Nowak, and Figueiredo]{wright2009sparse}
S.~J. Wright, R.~D. Nowak, and M.~A. Figueiredo.
\newblock Sparse reconstruction by separable approximation.
\newblock \emph{IEEE Transactions on signal processing}, 57\penalty0
  (7):\penalty0 2479--2493, 2009.

\bibitem[Xiao et~al.(2023)Xiao, Chen, Balasubramanian, and
  Ghadimi]{xiao2023one}
T.~Xiao, X.~Chen, K.~Balasubramanian, and S.~Ghadimi.
\newblock A one-sample decentralized proximal algorithm for non-convex
  stochastic composite optimization.
\newblock In \emph{Uncertainty in Artificial Intelligence}, pages 2324--2334.
  PMLR, 2023.

\bibitem[Xiao(2021)]{xiao2021unified}
X.~Xiao.
\newblock A unified convergence analysis of stochastic bregman proximal
  gradient and extragradient methods.
\newblock \emph{Journal of optimization theory and applications}, 188\penalty0
  (3):\penalty0 605--627, 2021.

\bibitem[Xin et~al.(2021)Xin, Das, Khan, and Kar]{xin2021stochastic}
R.~Xin, S.~Das, U.~A. Khan, and S.~Kar.
\newblock A stochastic proximal gradient framework for decentralized non-convex
  composite optimization: Topology-independent sample complexity and
  communication efficiency.
\newblock \emph{arXiv preprint arXiv:2110.01594}, 2021.

\bibitem[Xu et~al.(2021)Xu, Tian, Sun, and Scutari]{xu2021distributed}
J.~Xu, Y.~Tian, Y.~Sun, and G.~Scutari.
\newblock Distributed algorithms for composite optimization: Unified framework
  and convergence analysis.
\newblock \emph{IEEE Transactions on Signal Processing}, 69:\penalty0
  3555--3570, 2021.

\bibitem[Yan et~al.(2023)Yan, Chen, Chen, Cui, Lu, and Xu]{yan2023compressed}
Y.~Yan, J.~Chen, P.-Y. Chen, X.~Cui, S.~Lu, and Y.~Xu.
\newblock Compressed decentralized proximal stochastic gradient method for
  nonconvex composite problems with heterogeneous data.
\newblock In \emph{International Conference on Machine Learning}, pages
  39035--39061. PMLR, 2023.

\bibitem[Yang et~al.(2016)Yang, Scutari, Palomar, and
  Pesavento]{yang2016parallel}
Y.~Yang, G.~Scutari, D.~P. Palomar, and M.~Pesavento.
\newblock A parallel decomposition method for nonconvex stochastic multi-agent
  optimization problems.
\newblock \emph{IEEE Transactions on Signal Processing}, 64\penalty0
  (11):\penalty0 2949--2964, 2016.

\bibitem[Ye and Chang(2025)]{ye2025optimal}
H.~Ye and X.~Chang.
\newblock Optimal decentralized composite optimization for strongly convex
  functions.
\newblock \emph{Journal of Machine Learning Research}, 26\penalty0
  (280):\penalty0 1--34, 2025.

\bibitem[Yuan et~al.(2016)Yuan, Ling, and Yin]{yuan2016convergence}
K.~Yuan, Q.~Ling, and W.~Yin.
\newblock On the convergence of decentralized gradient descent.
\newblock \emph{SIAM Journal on Optimization}, 26\penalty0 (3):\penalty0
  1835--1854, 2016.

\bibitem[Yuan and Lin(2006)]{yuan2006model}
M.~Yuan and Y.~Lin.
\newblock Model selection and estimation in regression with grouped variables.
\newblock \emph{Journal of the Royal Statistical Society Series B: Statistical
  Methodology}, 68\penalty0 (1):\penalty0 49--67, 2006.

\bibitem[Zhang(2025)]{zhang2025stochastic}
J.~Zhang.
\newblock Stochastic bregman proximal gradient method revisited: Kernel
  conditioning and painless variance reduction: J. zhang.
\newblock \emph{Mathematical Programming}, pages 1--60, 2025.

\end{thebibliography}

\clearpage
\appendix

\section{Dependencies among the Convergence Results}\label{app:assumption-usage}
The smooth proximal-residual theorem uses the localized geometry of Assumption~\ref{ass:standing-local}, the constants of Lemma~\ref{lem:local-variation-from-relative}, and the admissible region in Lemma~\ref{lem:admissible-window-main}. The deterministic composite theorem additionally uses Assumption~\ref{ass:composite-regularizer}, the motion defect in Definition~\ref{def:subgradient-motion-defects}(b), and the Lyapunov lower bound. Conversion to the consensus proximal residual further requires the aggregation defect in Definition~\ref{def:aggregation-defects}(a) and $c_0+\rho\bar d>L_-$. Assumption~\ref{ass:finite-defect-accumulation}(a) gives Corollary~\ref{cor:finite-transition-composite-rate}. Verified selection gives summable motion defects, vanishing aggregation defects on the identified tail, and Lyapunov lower boundedness through Lemma~\ref{comp:lem:verified-selection-absorption} and Corollary~\ref{cor:verified-selection-lower-bound}; Corollary~\ref{cor:structured-nonsmooth-rate} accounts for the finite prefix through its proximal residuals.

The stochastic results use the corresponding expected defects and Assumption~\ref{ass:stochastic-oracle}. Assumption~\ref{ass:finite-defect-accumulation}(b) bounds the accumulated expected defects in Corollary~\ref{cor:stochastic-finite-accumulation}. The stopping-time selection in Definition~\ref{def:stochastic-verified-local-selection} verifies this condition by Lemma~\ref{stoch:lem:verified-selection-absorption}; the identified $\ell_1$ and group $\ell_{2,1}$ regimes verify it directly by Proposition~\ref{prop:stability-sufficient}(c) with an integrable prefix. Summable oracle variance then gives an $O(T^{-1})$ expected squared best-iterate residual, while general variance sequences contribute $V_T/T$.

For a globally regular quadratic kernel and globally Lipschitz smooth parts, the localized constants reduce to the model constants. The compact localization in Assumption~\ref{ass:standing-local}(d) supplies uniform bounds for the trajectory and comparison points throughout the proofs. Smooth convex composite terms use their Lipschitz gradients as selections. For polyhedral and partly smooth regularizers, common active-structure identification and the local selections in Proposition~\ref{prop:stability-sufficient} give the tail defect control used in Corollary~\ref{cor:structured-nonsmooth-rate}.

\section{Experimental Details and Extended Comparisons}\label{app:experimental-details}
This appendix specifies the test problems and parameter-selection procedure and reports additional results for active-structure stabilization, composite recovery, kernel choice, stochastic noise, network topology, and communication cost. Sparse-recovery tables include recovery error, support F1, and support size to distinguish accurate recovery from dense fitting.

\subsection{Test Problems and Data Generation}

For $M_i$ local measurements, the phase-retrieval and matrix-sensing objectives are
\[
f_i(x)=\frac1{M_i}\sum_{m=1}^{M_i}\big((a_{i,m}^{\top}x)^2-y_{i,m}\big)^2,
\qquad
f_i(U)=\frac1{M_i}\sum_{m=1}^{M_i}\big(\langle A_{i,m},UU^\top\rangle-y_{i,m}\big)^2,
\]
respectively. Here $a_{i,m}$ and $A_{i,m}$ are local sensing vectors and matrices, $y_{i,m}$ is the corresponding observation, and $U$ is the low-rank factor. The linear sparse-recovery comparisons use $f_i(x)=(2M_i)^{-1}\|A_i x-b_i\|^2$, and shared-support recovery uses $f_i(X)=(2M_i)^{-1}\|A_iX-B_i\|_{\mathrm F}^2$ with signal matrix $X$ and row penalty $\|X\|_{2,1}=\sum_j\|X_{j,:}\|$; $A_i$ denotes the local design matrix and $b_i$ or $B_i$ its vector- or matrix-valued observations. Sparse phase retrieval combines the first objective with an $\ell_1$ penalty. Measurement matrices and additive observation noise are Gaussian. For heterogeneity factor $h$, the sensing operator at agent $i\in\{0,\ldots,N-1\}$ is scaled by $1+h i/(N-1)$; the shared-support study additionally multiplies individual columns by $1+0.35h u$, where $u$ is uniform on $[-1,1]$. Thus larger $h$ produces greater variation among the local sensing operators. The recovery-comparison signals have unit Euclidean or Frobenius norm. The defect study uses the amplitudes and loss normalization specified in Section~\ref{sec:exp-defect}.

For structured Fourier imaging, $x\in\R^{4096}$ contains level-3 db2 wavelet coefficients and agent $i$ uses
\[
f_i(x)=\frac{N}{2}\big\|P_{\Omega_i}\mathcal F W^\top x-b_i\big\|^2,
\qquad
\Omega_i\cap\Omega_j=\varnothing\ (i\ne j),
\qquad
\bigcup_{i=1}^N\Omega_i=\Omega.
\]
The fixed-cardinality mask draws $35\%$ of the two-dimensional Fourier coefficients with radially decaying probability after including a centered square whose width is $8\%$ of each axis. The selected coefficients are shuffled and partitioned as evenly as possible among the six agents. Circular complex Gaussian noise is scaled to $1\%$ of the root-mean-square magnitude of the noiseless observations. Images are normalized to $[0,1]$. For reconstruction $\widehat u=W^\top\bar x^T$ and reference $u^\star$, NRMSE is $\|\widehat u-u^\star\|/\|u^\star\|$, and PSNR is $20\log_{10}(R/\operatorname{RMSE})$, where RMSE is root-mean-square error and $R$ is the image range. SSIM uses Gaussian local statistics with standard deviation $1$, $C_1=(0.01R)^2$, and $C_2=(0.03R)^2$.

The defect study uses a ring with one chord, and imaging and communication-aware comparisons use a ring with two chords. The other recovery comparisons use Erd\H{o}s--R\'enyi graphs with edge probability $0.5$, resampled at increasing probability until connected. Phase retrieval and matrix sensing use leading-eigenspace spectral initialization; sparse phase retrieval retains the largest spectral coordinates. Linear and group recovery comparisons use a common normalized Gaussian initialization, with independent node perturbations of standard deviation $0.01$. The defect study and imaging start from zero. Paired methods share the same initialization.

\subsection{Parameter Selection and Evaluation Design}
Tuned primal and dual stepsizes, smoothing weights, and regularization weights are selected on validation instances disjoint from the reported evaluation set. Within each comparison family, methods receive the stated tuning budget and use common evaluation instances and iteration budgets. The noise sweep keeps a single selected SDBPG configuration fixed across oracle settings, and the defect study uses the prescribed parameters in Section~\ref{sec:exp-defect}.

Parameters are chosen from the candidate configurations using the prescribed validation criteria, independently of the sufficient analytical parameter regions. Those regions are expressed in terms of localized curvature, kernel and graph constants, defect absorption constants, and the auxiliary weights $\tau$ and $\chi$. The experiments characterize recovery, stationarity, consensus, and communication cost for the selected configurations over fixed iteration budgets.

\begin{table}[htbp]
\centering
\caption{Experimental settings. Iterations and evaluation counts refer to the reported results, per regime or image; validation uses disjoint instances. The noise sweep selects parameters on the moderate regime and keeps them fixed.}
\label{tab:appendix-experiment-settings}

\footnotesize
\setlength{\tabcolsep}{4pt}

\begin{tabular}{>{\raggedright\arraybackslash}p{0.21\textwidth}>{\raggedright\arraybackslash}p{0.38\textwidth}rrr}
\toprule
Study & Network and data & Iterations & Validation & Evaluation\\
\midrule
Defects ($\ell_1$ / $\ell_{2,1}$) & $N=6$, $d=40$, 25 samples/node; $\sigma=10^{-3}$; 4 group signals & 3000 & Prescribed & 10\\
Phase retrieval & $N=6$, $d=12$; easy: 90 samples/node, $\sigma=0$; hard: 12 samples/node, $\sigma=0.15$ & 1200 & 5 & 10\\
Quartic phase retrieval & $N=6$, $d=16$, 60 samples/node; $\sigma=0.02$; $h=2.5,3,3.25$ & 600 & 5 & 10\\
Matrix factorization & $N=6$, $8\times2$ factor; easy: 120 samples/node, $\sigma=0$; hard: 16 samples/node, $\sigma=0.08$ & 1200 & 5 & 10\\
Sparse recovery & $N=8$; undersampled: $d=40$, 6 samples/node; moderate: $d=50$, 8 samples/node; $\sigma=0.08$ & 900 & 5 & 10\\
Sparse phase retrieval & $N=8$; undersampled: $d=40$, 8 samples/node; moderate: $d=50$, 10 samples/node; $\sigma=0.08$ & 1200 & 5 & 10\\
Shared-support recovery & $N=8$, 4 signals; undersampled: $d=36$, 8 samples/node; heterogeneous: $d=40$, 9 samples/node; $\sigma=0.08$ & 950 & 5 & 10\\
Extended Fourier imaging & $N=6$, $64\times64$, 35\% Fourier samples, 1\% relative noise; Phantom and Cameraman & 1000 & Preselected & 10\\
Stochastic noise sweep & $N=8$; undersampled: $d=40$, 6 samples/node; moderate: $d=50$, 8 samples/node; $\sigma=0.08$ & 900 & 5 & 20\\
Modern Fourier comparison & $N=6$, $64\times64$, 35\% Fourier samples, 1\% relative noise & 1000 & 3 & 10\\
Modern stochastic comparison & $N=8$, $d=50$, 8 samples/node; batch size 4, $\sigma=0.08$ & 900 & 3 & 10\\
\bottomrule
\end{tabular}
\end{table}

For phase retrieval and matrix factorization, the heterogeneity factor is $h=0.5$ in the easy regime and $h=2.0$ in the hard regime. Sparse and sparse-phase recovery use $h=1.6$. Shared-support recovery uses $h=1.8$ in the undersampled regime and $h=2.2$ in the heterogeneous regime. The geometry-matched study separately uses $h=2.5$, $3.0$, and $3.25$. The validation sweeps cover the primal and dual stepsizes, smoothing weights, regularization weights, and the consensus depth of ProxGT; the stochastic noise study holds the algorithmic parameters fixed while varying mini-batch size or additive-noise magnitude.

For unregularized phase retrieval and matrix factorization, final recovery error is the primary criterion, supplemented by stationarity and consensus residuals. For matrix factorization, $\widehat U$ is the factor reshaped from $\widehat x$ and the recovery error is $\|\widehat U\widehat U^\top-X^\star\|_{\mathrm F}/\|X^\star\|_{\mathrm F}$. Smooth stationarity is measured by the running minimum of $\|N^{-1}\sum_i\nabla f_i(\bar x^t)\|^2$. Sparse and sparse-phase recovery pair recovery error with support F1, while support size describes the degree of sparsity relative to the true active size.

\subsection{Residual Components and Active-Set Variation}\label{app:residual-components}
Figure~\ref{fig:lifted-residual-diagnostics} shows the lifted primal--dual residual and its stationarity and consensus components, together with the smoothing mismatch $e^t$, oracle error, and support variation. These diagnostics use ten sparse-recovery instances over 700 iterations, with $N=8$ and dimensions 40 and 50 for the deterministic and stochastic cases, respectively. The deterministic support panel records stabilization under the evaluation threshold $0.05$, complementing the separate defect experiment in Figure~\ref{fig:defect-identification}.

\begin{figure}[htbp]
\centering
\includegraphics[width=\textwidth]{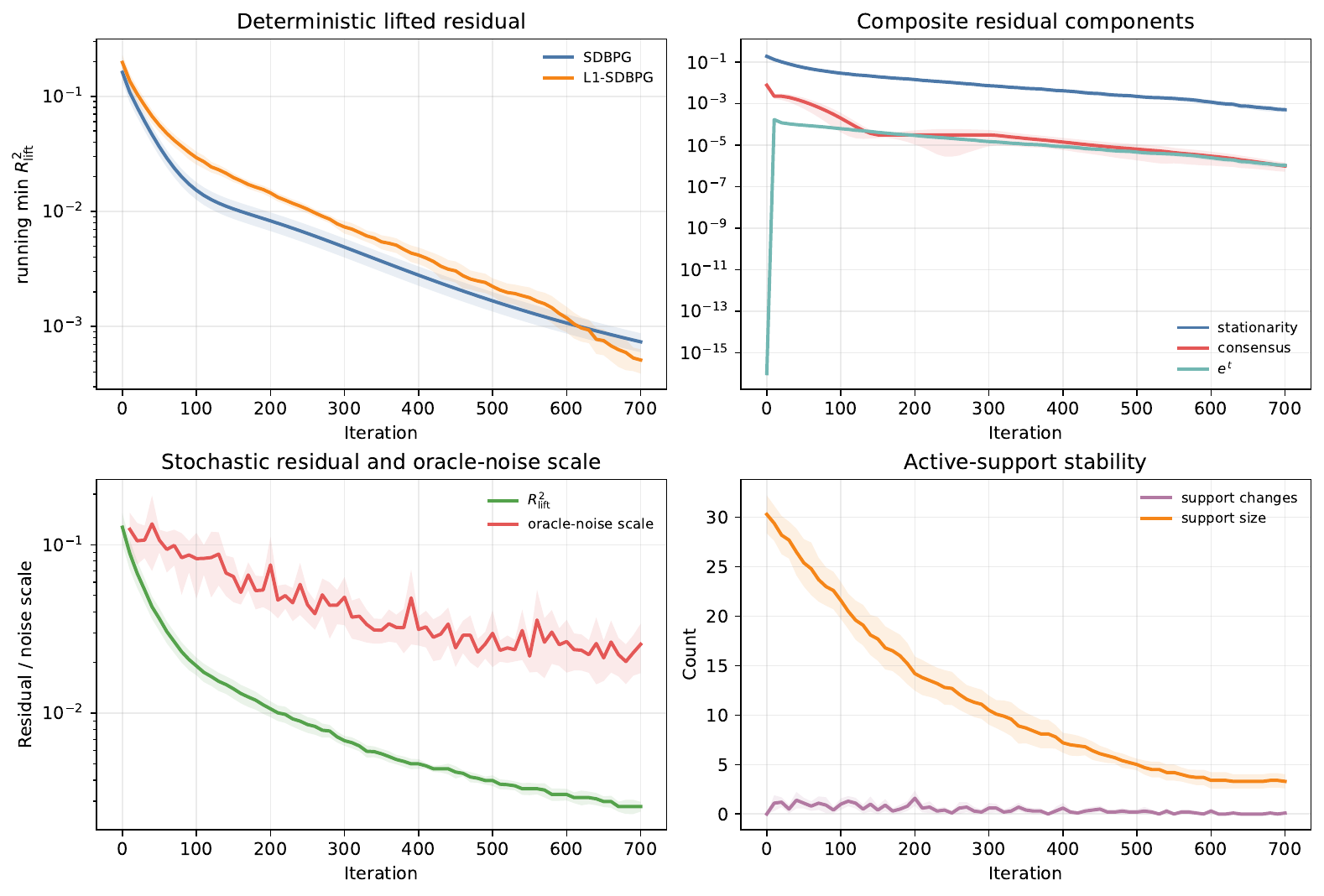}
\caption{Sparse-recovery diagnostics: deterministic lifted residual (top left); stationarity, consensus, and smoothing mismatch (top right); stochastic lifted residual and squared oracle error (bottom left); and deterministic support size and changes (bottom right). Lifted, stationarity, and consensus residual curves show running minima of instance means; the remaining curves show pointwise means.}
\label{fig:lifted-residual-diagnostics}
\end{figure}
\FloatBarrier

\subsection{Smooth SDBPG Baseline}\label{app:base-dynamics}
Table~\ref{tab:main-convergence} places the composite and stochastic extensions against the smooth SDBPG baseline by reporting stationarity and consensus after 2000 iterations.

\begin{table}[htbp]
\centering
\caption{Smooth SDBPG baseline: stationarity, consensus, and recovery with parameters selected on validation instances.}
\label{tab:smooth-backbone-combined}

\scriptsize
\begin{subtable}{0.92\textwidth}
\centering
\caption{Stationarity and consensus after 2000 iterations.}
\label{tab:main-convergence}
\begin{tabular}{lcc}
\toprule
Algorithm & Smallest squared gradient norm observed & Final consensus error\\
\midrule
SDBPG (unregularized) & \best{$7.051\times10^{-7}$} & $1.836\times10^{-10}$\\
DGD & $1.813\times10^{-5}$ & $7.349\times10^{-10}$\\
DGT & $1.508\times10^{-5}$ & \best{$2.010\times10^{-14}$}\\
\bottomrule
\end{tabular}
\end{subtable}

\vspace{0.5em}
\begin{subtable}{\textwidth}
\centering
\caption{Final recovery error; brackets give 95\% bootstrap confidence intervals over ten evaluation instances.}
\label{tab:benchmark-main}
\resizebox{\textwidth}{!}{%
\begin{tabular}{llcc}
\toprule
Task & Algorithm & Easy & Hard\\
\midrule
\multirow{3}{*}{Phase retrieval}
& SDBPG (unregularized) & \best{0.0727 [0.0566, 0.0902]} & \best{0.1562 [0.0939, 0.2445]}\\
& DGD & 0.1440 [0.1134, 0.1800] & 0.2206 [0.1098, 0.3808]\\
& DGT & 0.0900 [0.0713, 0.1114] & 0.1886 [0.0967, 0.3296]\\
\midrule
\multirow{3}{*}{Matrix factorization}
& SDBPG (unregularized) & 0.0145 [0.0058, 0.0275] & 0.0383 [0.0290, 0.0480]\\
& DGD & 0.0421 [0.0254, 0.0600] & 0.0475 [0.0326, 0.0635]\\
& DGT & \best{0.0136 [0.0052, 0.0263]} & \best{0.0374 [0.0280, 0.0467]}\\
\bottomrule
\end{tabular}}
\end{subtable}

\end{table}

All methods reach small consensus error within 2000 iterations, while SDBPG attains the lowest reported running-minimum gradient norm squared. Table~\ref{tab:benchmark-main} gives recovery results for unregularized phase retrieval and matrix factorization. SDBPG has the lowest mean phase-retrieval error in both regimes. In matrix factorization, SDBPG improves on DGD, while DGT gives slightly lower means.

\FloatBarrier

\subsection{Sparse-Recovery Trajectories}\label{app:sparse-recovery-trajectories}
Figure~\ref{fig:sparse-recovery-tradeoff} displays the recovery--support tradeoff behind Table~\ref{tab:sparse-recovery-main}. Figure~\ref{fig:sparse-recovery-dynamics} follows the noisy undersampled regime over iterations: as the $\ell_1$-SDBPG support contracts, its support F1 increases and its recovery error decreases.

\begin{figure}[htbp]
\centering
\begin{subfigure}{0.97\textwidth}
\centering
\includegraphics[width=0.95\textwidth]{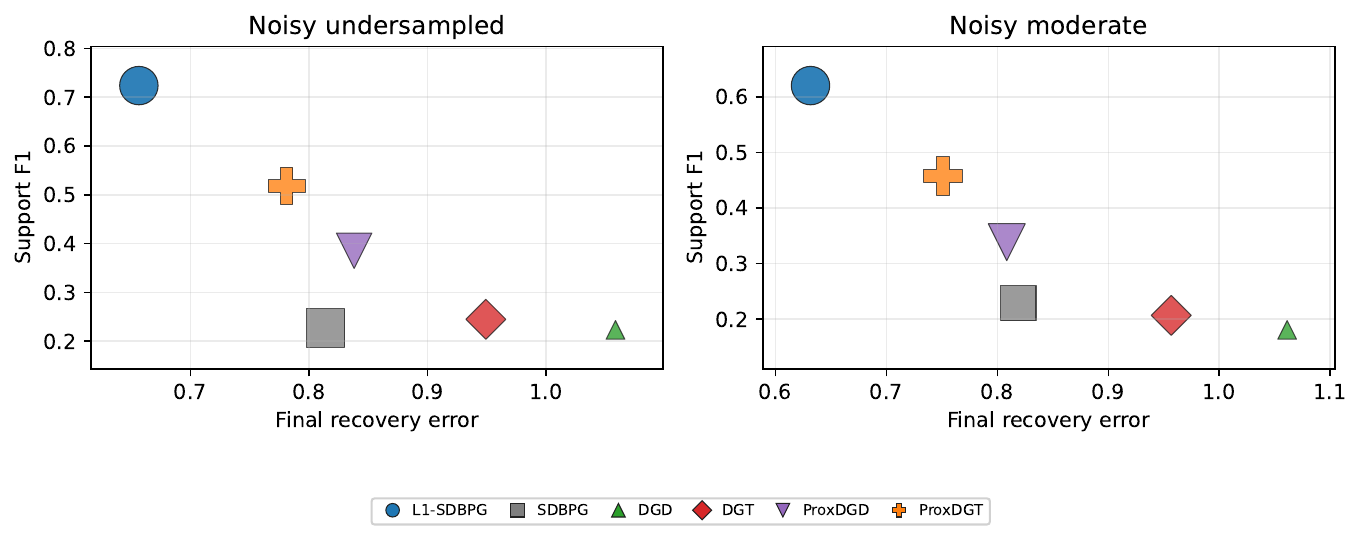}
\caption{Recovery--support tradeoff.}
\label{fig:sparse-recovery-tradeoff}
\end{subfigure}

\vspace{0.5em}
\begin{subfigure}{0.97\textwidth}
\centering
\includegraphics[width=0.95\textwidth]{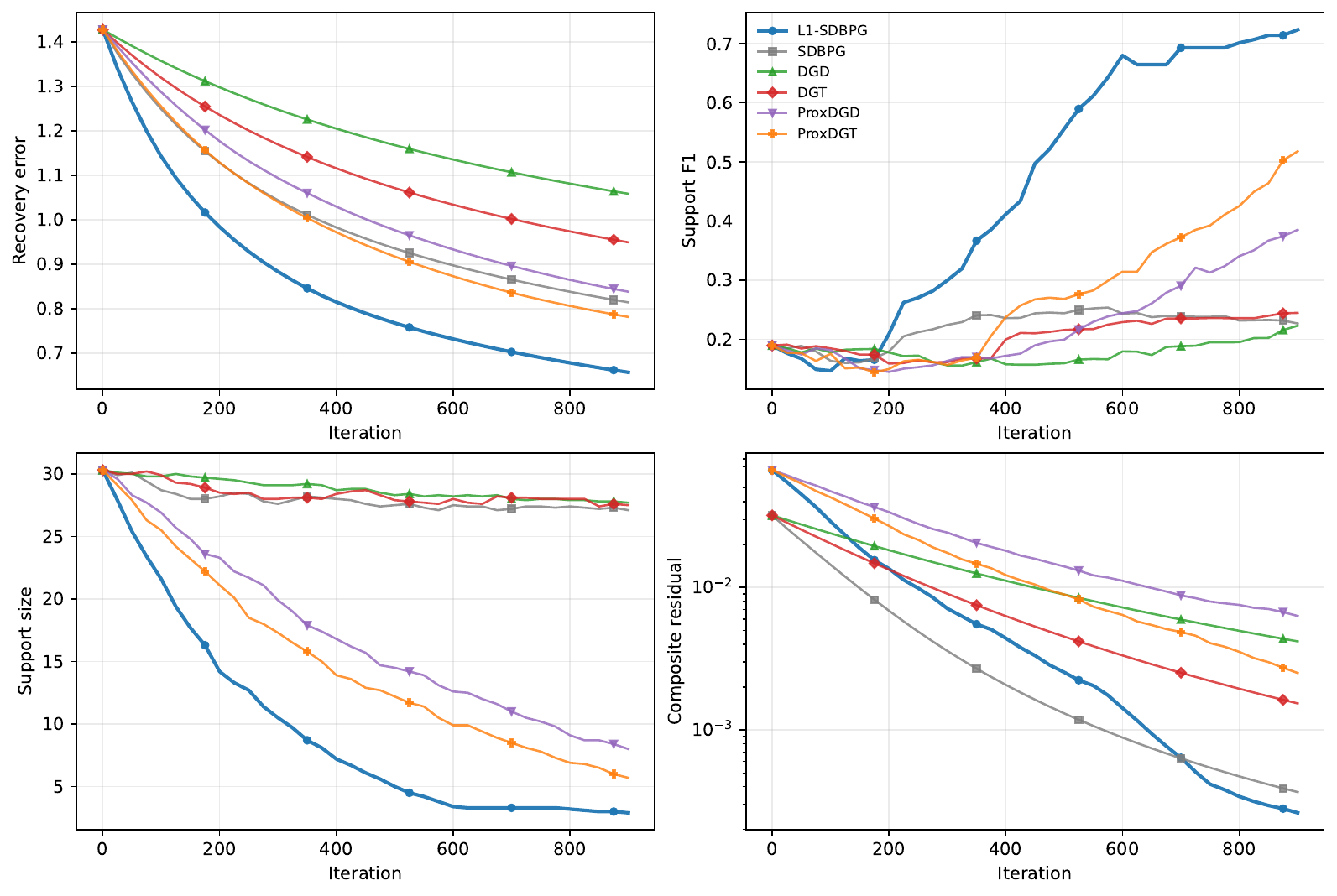}
\caption{Trajectories in the noisy undersampled regime; curves report averages over the evaluation instances.}
\label{fig:sparse-recovery-dynamics}
\end{subfigure}
\caption{Sparse-recovery tradeoff and iteration dynamics.}
\label{fig:sparse-recovery-trajectories}
\end{figure}
\FloatBarrier

\subsection{EXTRA-Family Sparse-Recovery Baselines}\label{app:extra-sparse-baselines}
Table~\ref{tab:appendix-extra-baselines} compares independently tuned PG-EXTRA and Prox-EXTRA on the two representative sparse-recovery regimes. PG-EXTRA omits the regularizer, whereas Prox-EXTRA applies an $\ell_1$ proximal step. PG-EXTRA has the lowest recovery error in the undersampled regime but remains dense; $\ell_1$-SDBPG has the highest support F1 in both regimes and the lowest recovery error in the moderate regime.

\begin{table}[htbp]
\centering
\caption{Numerical results for the independently tuned EXTRA-family baselines. True support sizes are 4 and 5 in the undersampled and moderate regimes. Bold entries mark the lowest recovery error and highest support F1.}
\label{tab:appendix-extra-baselines}

\scriptsize

\resizebox{\textwidth}{!}{%
\begin{tabular}{llccc}
\toprule
Regime & Algorithm & Final recovery error & Support F1 & Support size\\
\midrule
\multirow{6}{*}{Noisy undersampled}
& $\ell_1$-SDBPG & 0.657 & \best{0.724} & 2.9\\
& PG-EXTRA & \best{0.632} & 0.262 & 24.4\\
& Prox-EXTRA & 1.038 & 0.257 & 13.3\\
& ProxDGT & 0.781 & 0.518 & 5.7\\
& ProxDGD & 0.838 & 0.385 & 8.0\\
& SDBPG & 0.814 & 0.227 & 27.1\\
\midrule
\multirow{6}{*}{Noisy moderate}
& $\ell_1$-SDBPG & \best{0.632} & \best{0.620} & 2.6\\
& PG-EXTRA & 0.646 & 0.254 & 29.8\\
& Prox-EXTRA & 1.057 & 0.164 & 21.6\\
& ProxDGT & 0.751 & 0.457 & 5.4\\
& ProxDGD & 0.808 & 0.338 & 8.7\\
& SDBPG & 0.819 & 0.229 & 31.9\\
\bottomrule
\end{tabular}}
\end{table}

\subsection{Sparse Phase Retrieval}\label{app:sparse-phase-retrieval}
Sparse phase retrieval tests the composite update on a nonconvex recovery model. Table~\ref{tab:sparse-phase-main} shows that $\ell_1$-SDBPG has the lowest mean recovery error in both noisy regimes and produces much smaller supports than the unregularized methods. ProxDGT gives the highest support F1, with recovery errors of $0.870$ and $0.916$ compared with $0.864$ and $0.904$ for $\ell_1$-SDBPG.

\begin{table}[htbp]
\centering
\caption{Sparse phase retrieval with composite decentralized methods. Bold entries mark the lowest recovery error and highest support F1 in each regime.}
\label{tab:sparse-phase-main}

\scriptsize

\resizebox{\textwidth}{!}{%
\begin{tabular}{lcccccc}
\toprule
& \multicolumn{3}{c}{Noisy undersampled} & \multicolumn{3}{c}{Noisy moderate}\\
\cmidrule(lr){2-4}\cmidrule(lr){5-7}
Algorithm & Recovery error & Support F1 & Support size & Recovery error & Support F1 & Support size\\
\midrule
$\ell_1$-SDBPG & \best{0.864} & 0.301 & 10.0 & \best{0.904} & 0.232 & 11.3\\
SDBPG & 0.965 & 0.213 & 20.8 & 1.006 & 0.248 & 25.3\\
DGD & 0.975 & 0.206 & 20.6 & 0.991 & 0.247 & 21.9\\
DGT & 0.965 & 0.212 & 20.9 & 0.999 & 0.239 & 23.9\\
ProxDGD & 0.868 & 0.329 & 5.1 & 0.911 & 0.253 & 5.2\\
ProxDGT & 0.870 & \best{0.342} & 4.7 & 0.916 & \best{0.275} & 3.7\\
\bottomrule
\end{tabular}}
\end{table}

The paired recovery-error differences, $\ell_1$-SDBPG minus ProxDGT, are $-0.00596$ (95\% CI $[-0.03165,0.01989]$) in the undersampled regime and $-0.01216$ (95\% CI $[-0.05301,0.02267]$) in the moderate regime. Both intervals contain zero. Figure~\ref{fig:sparse-phase-comparison} displays the recovery--support tradeoff: $\ell_1$-SDBPG gives the lowest mean recovery errors, while ProxDGT gives higher mean support F1.

\begin{figure}[htbp]
\centering
\includegraphics[width=0.95\textwidth]{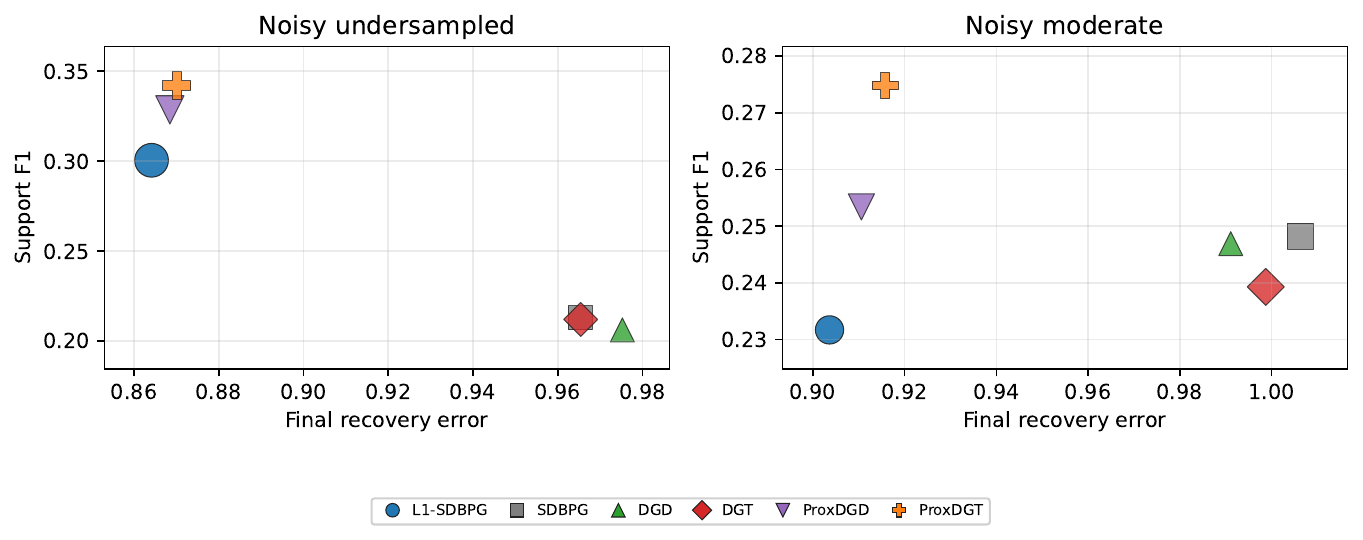}
\caption{Sparse phase-retrieval recovery--support tradeoff. Points show evaluation means for the methods in Table~\ref{tab:sparse-phase-main}; lower recovery error and higher support F1 are preferable.}
\label{fig:sparse-phase-tradeoff}
\label{fig:sparse-phase-comparison}
\end{figure}
\FloatBarrier

\subsection{Structured Fourier Imaging: Extended Results}\label{app:structured-fourier-imaging}
Table~\ref{tab:structured-fourier-main} uses new masks and noise draws and adds ProxPDA, Prox-EXTRA, unregularized SDBPG, centralized FISTA, and Cameraman images to the phantom comparison in Section~\ref{sec:structured-fourier-imaging}. $\ell_1$-SDBPG has the lowest mean NRMSE among decentralized methods on both images. On the phantom, its NRMSE is $3.47\%$ lower than that of ProxDGT and $3.51\%$ lower than that of ProxPDA. On Cameraman, both paired intervals contain zero. Relative to unregularized SDBPG, its NRMSE decreases by $52.8\%$ on the phantom and $29.5\%$ on Cameraman.

\begin{table}[htbp]
\centering
\caption{Extended structured-Fourier results over ten masks and noise draws per image: reconstruction quality and paired NRMSE differences.}
\label{tab:structured-fourier-combined}

\footnotesize
\setlength{\tabcolsep}{4pt}

\begin{subtable}{\textwidth}
\centering
\caption{Reconstruction means. Bold marks the best displayed decentralized value, including ties at the displayed precision. Centralized FISTA is a quality reference.}
\label{tab:structured-fourier-main}
\begin{tabular}{lrrrrrr}
\toprule
& \multicolumn{3}{c}{Phantom} & \multicolumn{3}{c}{Cameraman}\\
\cmidrule(lr){2-4}\cmidrule(lr){5-7}
Algorithm & NRMSE $\downarrow$ & PSNR $\uparrow$ & SSIM $\uparrow$ & NRMSE $\downarrow$ & PSNR $\uparrow$ & SSIM $\uparrow$\\
\midrule
$\ell_1$-SDBPG & \best{0.1771} & \best{27.17} & \best{0.8348} & \best{0.04304} & \best{31.25} & \best{0.8731}\\
ProxDGT & 0.1835 & 26.86 & 0.8192 & 0.04305 & \best{31.25} & 0.8729\\
ProxPDA & 0.1836 & 26.85 & 0.8191 & 0.04306 & \best{31.25} & 0.8729\\
SDBPG & 0.3753 & 20.64 & 0.3527 & 0.06101 & 28.23 & 0.7012\\
Prox-EXTRA & 0.3780 & 20.58 & 0.6077 & 0.10448 & 23.53 & 0.6795\\
\midrule
Centralized FISTA & 0.1646 & 27.81 & 0.8338 & 0.04021 & 31.84 & 0.8699\\
\bottomrule
\end{tabular}
\end{subtable}

\vspace{0.5em}
\begin{subtable}{\textwidth}
\centering
\caption{Paired NRMSE differences, $\ell_1$-SDBPG minus comparator. Negative values favor $\ell_1$-SDBPG; intervals and win counts use the same ten instances.}
\label{tab:structured-fourier-paired}
\begin{tabular}{llrcr}
\toprule
Image & Comparator & Mean difference & 95\% bootstrap interval & Wins\\
\midrule
\multirow{4}{*}{Phantom}
& ProxDGT & $-0.006361$ & $[-0.007783,-0.004562]$ & 10/10\\
& ProxPDA & $-0.006446$ & $[-0.007886,-0.004634]$ & 10/10\\
& SDBPG & $-0.198227$ & $[-0.205550,-0.188623]$ & 10/10\\
& Centralized FISTA & $0.012511$ & $[0.010253,0.014305]$ & 0/10\\
\midrule
\multirow{4}{*}{Cameraman}
& ProxDGT & $-0.000013$ & $[-0.000180,0.000150]$ & 5/10\\
& ProxPDA & $-0.000019$ & $[-0.000188,0.000146]$ & 6/10\\
& SDBPG & $-0.017973$ & $[-0.020141,-0.015805]$ & 10/10\\
& Centralized FISTA & $0.002823$ & $[0.002579,0.003094]$ & 0/10\\
\bottomrule
\end{tabular}
\end{subtable}
\end{table}

Table~\ref{tab:structured-fourier-configs} lists the parameters selected before evaluation. They remain fixed across all reported masks and noise draws. Every decentralized method uses one local-gradient evaluation per agent and iteration. Communication counts include both directions of each edge and distinguish methods that transmit one state vector from those that transmit two.

\begin{table}[htbp]
\centering
\caption{Parameters and per-edge communication payloads for the extended Fourier benchmark. The step is the primal or proximal stepsize; $\alpha,\beta,c_0,\rho$ denote SDBPG parameters. ProxPDA additionally uses augmented-consensus coefficient $0.15$.}
\label{tab:structured-fourier-configs}

\footnotesize
\setlength{\tabcolsep}{5pt}

\begin{tabular}{lrrrrrrr}
\toprule
Algorithm & $\lambda_r$ & Step & $\alpha$ & $\beta$ & $c_0$ & $\rho$ & States\\
\midrule
$\ell_1$-SDBPG & 0.01 & 0.16 & 0.12 & 0.30 & 0.75 & 1 & 2\\
SDBPG & 0 & 0.12 & 0.16 & 0.85 & 1.5 & 2 & 2\\
ProxDGT & 0.01 & 0.08 & -- & -- & -- & -- & 2\\
ProxPDA & 0.01 & 0.08 & -- & -- & -- & -- & 1\\
Prox-EXTRA & $10^{-4}$ & 0.05 & -- & -- & -- & -- & 1\\
Centralized FISTA & 0.005 & 1 & -- & -- & -- & -- & --\\
\bottomrule
\end{tabular}
\end{table}

Table~\ref{tab:structured-fourier-paired} gives instance-wise NRMSE differences. On the phantom, $\ell_1$-SDBPG gives lower NRMSE than ProxDGT and ProxPDA on all ten instances. On Cameraman, the mean differences are two orders of magnitude smaller and the intervals contain zero. Centralized FISTA uses the separately selected regularization weight $0.005$ and provides a reconstruction reference. It has lower NRMSE, while $\ell_1$-SDBPG has slightly higher mean SSIM.

The mean squared composite-stationarity values of $\ell_1$-SDBPG are $1.22\times10^{-5}$ on the phantom and $6.34\times10^{-7}$ on Cameraman, compared with $1.40\times10^{-4}$ and $5.88\times10^{-6}$ for ProxDGT and $1.43\times10^{-4}$ and $6.08\times10^{-6}$ for ProxPDA. For an NRMSE threshold equal to $1.10$ times the best final decentralized value on each instance, the mean transmitted volumes on the phantom are 396 MiB for $\ell_1$-SDBPG, 678 MiB for ProxDGT, and 339 MiB for ProxPDA; the corresponding Cameraman values are 122, 188, and 101 MiB. Thus $\ell_1$-SDBPG reaches this reconstruction level with less communication than ProxDGT. ProxPDA uses still less communication at the relaxed target but has a slightly worse final reconstruction.

\begin{figure}[!htbp]
\centering
\includegraphics[width=0.94\textwidth]{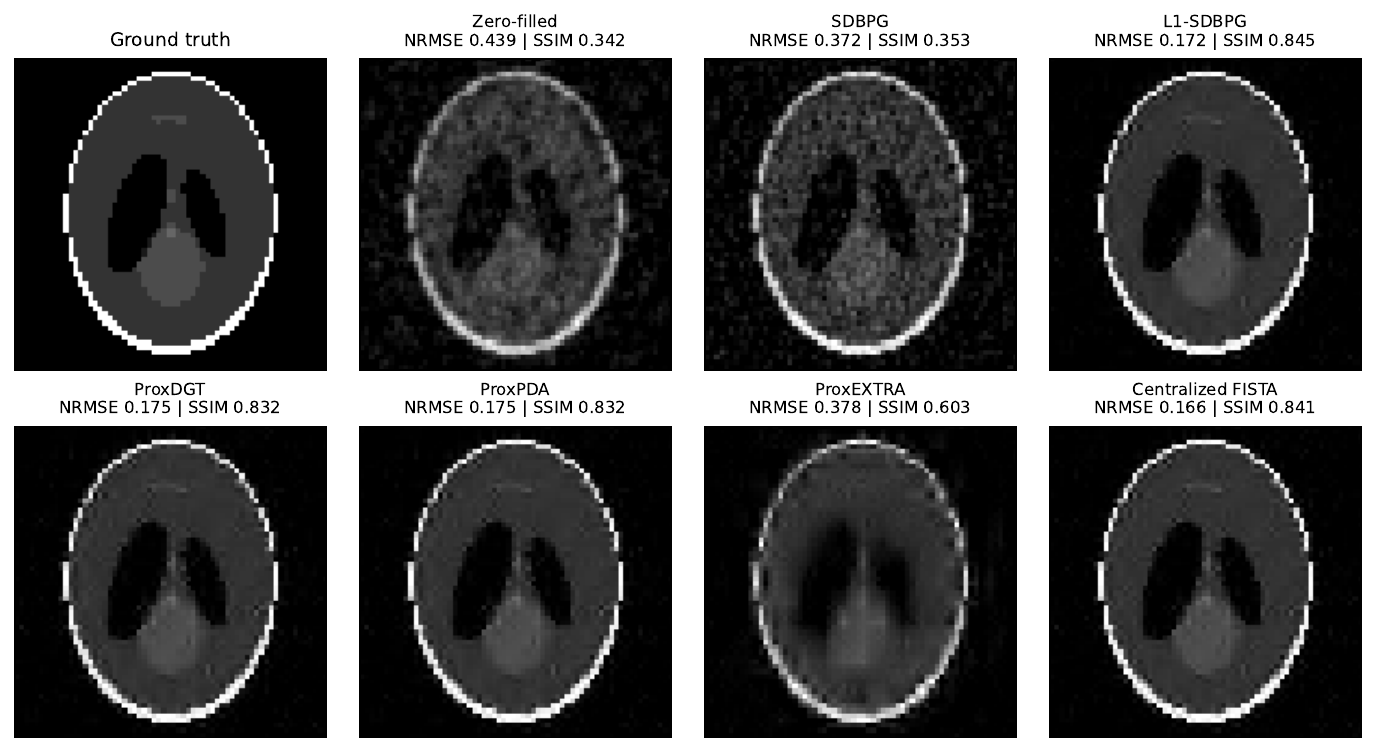}
\caption{Representative phantom reconstruction from the extended comparison. $\ell_1$-SDBPG reduces the aliasing visible in the zero-filled and unregularized SDBPG images and attains the lowest NRMSE among the decentralized proximal methods shown. Centralized FISTA provides a reconstruction reference; aggregate results appear in Table~\ref{tab:structured-fourier-main}.}
\label{fig:structured-fourier-phantom}
\end{figure}

The images have compressible wavelet representations. The phantom reconstructions in Figure~\ref{fig:structured-fourier-phantom} illustrate the reduction in aliasing under structural regularization; the Cameraman results in Table~\ref{tab:structured-fourier-combined} show similar reconstruction quality among the leading decentralized proximal methods.
\FloatBarrier

\subsection{Exact Quartic-Kernel Steps and Parameters}\label{app:quartic-kernel-details}
For the mirror map in~\eqref{eq:quartic-kernel-experiment},
\[
\nabla\phi_4(x)=(1+\|x\|^2)x.
\]
Consequently its inverse is radial. Given a dual vector $q\ne0$, let $r\geq0$ be the unique root of $r^3+r=\|q\|$; then $(\nabla\phi_4)^{-1}(q)=r q/\|q\|$. For a Bregman proximal-gradient step with local model gradient $g$, stepsize $\eta$, and $\ell_1$ weight $\lambda_r$, soft-thresholding is first applied in mirror coordinates,
\[
q=\operatorname{soft}\!\left(\nabla\phi_4(x)-\eta g,\eta\lambda_r\right),
\qquad
x^+= (\nabla\phi_4)^{-1}(q).
\]
Here $\operatorname{soft}(\cdot,\theta)$ denotes coordinatewise soft-thresholding at level $\theta$.
The phase-retrieval geometry study is unregularized, so $\lambda_r=0$. Its smoothing update is also exact: for
$q=x^++\beta\nabla\phi_4(z)$, the new radius is the unique nonnegative root of
\[
\beta r^3+(1+\beta)r=\|q\|,
\qquad
z^+=\begin{cases}r q/\|q\|,&q\ne0,\\0,&q=0.\end{cases}
\]
Both scalar equations are monotone and are solved to machine precision; no inner optimization tolerance is introduced.

\begin{table}[htbp]
\centering
\caption{Quartic-kernel phase retrieval: mean recovery error over ten evaluation instances and parameters selected on five validation instances at $h=3.0$, then kept fixed across the sweep.}
\label{tab:quartic-phase-combined}

\footnotesize
\setlength{\tabcolsep}{2pt}
\begin{subtable}[t]{0.44\textwidth}
\centering
\caption{Mean recovery error.}
\label{tab:quartic-phase-retrieval}
\begin{tabular}{lccc}
\toprule
Algorithm & \shortstack{$h=2.5$\\\strut} & \shortstack{$h=3.0$\\\strut} & \shortstack{$h=3.25$\\\strut}\\
\midrule
Quadratic SDBPG & \best{0.01753} & 0.01505 & 0.01384\\
Quartic SDBPG & 0.02056 & \best{0.01197} & \best{0.01113}\\
ProxGT & 0.06353 & 0.02452 & 0.01426\\
\bottomrule
\end{tabular}
\end{subtable}\hfill
\begin{subtable}[t]{0.54\textwidth}
\centering
\caption{Selected parameters.}
\label{tab:quartic-phase-configs}
\begin{tabular}{lrrrrrr}
\toprule
Algorithm & \shortstack{Primal/prox.\\step} & \shortstack{Dual\\step} & $\beta$ & $c_0$ & $\rho$ & $K$\\
\midrule
Quadratic SDBPG & 0.011 & 0.01 & 0.8 & 1.5 & 2 & --\\
Quartic SDBPG & 0.019 & 0.01 & 0.8 & 1.5 & 2 & --\\
ProxGT & 0.004 & -- & -- & -- & -- & 1\\
\bottomrule
\end{tabular}
\end{subtable}

\end{table}

At $h=3.0$, both SDBPG variants use 600 sequential communication rounds, $2.344$ MiB, and 216000 oracle samples at the full budget; ProxGT uses 1200 rounds with the same transmitted MiB and oracle samples. Mean wall times are $0.0753$, $0.1642$, and $0.0276$ seconds for quadratic SDBPG, quartic SDBPG, and ProxGT, respectively. Thus the kernel-matched update improves recovery at this heterogeneity level with greater local computation time.
\FloatBarrier

\subsection{Recent Algorithms and Cost Accounting}\label{app:modern-baseline-costs}
We implement ProxGT and ProxGT-SA according to Algorithm~1 of \citet{xin2021stochastic}. With $X_t$, $Y_t$, $V_t$, $W$, and $K$ denoting the primal state, gradient tracker, exact or stochastic gradient estimate, mixing matrix, and consensus depth, one iteration has the two mixing blocks
\[
Y_{t+1}=W^K(Y_t+V_t-V_{t-1}),\qquad
X_{t+1}=W^K\operatorname{prox}_{\alpha r}(X_t-\alpha Y_{t+1}).
\]
Here $\operatorname{prox}_{\alpha r}$ is the proximal operator of the regularizer $r$ with stepsize $\alpha$.
Prox-DASA-GT follows Algorithm~2 of \citet{xiao2023one}: it forms a proximal point from $(X_t,Z_t)$, averages it into the primal state with weight $\alpha$, updates the gradient-tracking and averaged-gradient states, and mixes all three states for $m$ rounds. Its initial stochastic gradient, used to set the two tracking states, is included in the oracle count.

Under the cost definitions in Section~\ref{sec:experiments}, SDBPG sends $(x,z)$ together in one round and transmits two state vectors; ProxGT uses $2K$ rounds and transmits one state vector per round; Prox-DASA-GT uses $m$ rounds and transmits three state vectors per round. These state counts produce the cumulative communication and oracle costs reported below.

\begin{table}[htbp]
\centering
\caption{Selected parameters and cost-to-accuracy results for the recent-algorithm comparisons.}
\label{tab:modern-baseline-combined}

\footnotesize
\begin{subtable}{\textwidth}
\centering
\caption{Selected parameters. The step is the primal/proximal stepsize, denoted by $\gamma$ for Prox-DASA-GT; its $\alpha$ is the averaging weight and $m$ is the mixing depth.}
\label{tab:modern-baseline-configs}
\begin{tabular}{llrr>{\raggedright\arraybackslash}p{0.34\textwidth}}
\toprule
Task & Algorithm & $\lambda_r$ & Step & Other parameters\\
\midrule
\multirow{3}{*}{Fourier}
& $\ell_1$-SDBPG & 0.01 & 0.16 & $\alpha=0.12$, $\beta=0.30$, $c_0=0.75$, $\rho=1$\\
& ProxGT & 0.01 & 0.12 & $K=2$\\
& ProxDGT & 0.01 & 0.08 & --\\
\midrule
\multirow{3}{*}{Stochastic}
& $\ell_1$-SDBPG & 0.02 & 0.18 & $\alpha=0.24$, $\beta=0.85$, $c_0=1.5$, $\rho=2$\\
& ProxGT-SA & 0.02 & 2.0 & $K=2$\\
& Prox-DASA-GT & 0.02 & 25.6 & $\alpha=0.2$, $m=1$\\
\bottomrule
\end{tabular}
\end{subtable}

\vspace{0.5em}
\begin{subtable}{\textwidth}
\centering
\caption{Cost to a common per-instance threshold equal to $1.10$ times the lowest final error. Values average ten evaluation instances; every method reaches the threshold on every instance.}
\label{tab:modern-cost-to-target}
\begin{tabular}{llrrrr}
\toprule
Task & Algorithm & Rounds & MiB & Oracle samples & Wall time (s)\\
\midrule
\multirow{3}{*}{Fourier imaging}
& $\ell_1$-SDBPG & 407.5 & 407.5 & 2445 & 1.0134\\
& ProxGT & 1870 & 935.0 & 2805 & 0.9355\\
& ProxDGT & 695 & 695.0 & 4176 & 1.5089\\
\midrule
\multirow{3}{*}{Stochastic recovery}
& $\ell_1$-SDBPG & 422.5 & 6.447 & 13520 & 0.1108\\
& ProxGT-SA & 140 & 1.068 & 1120 & 0.0049\\
& Prox-DASA-GT & 40 & 0.916 & 1312 & 0.0057\\
\bottomrule
\end{tabular}
\end{subtable}

\end{table}

The support-F1 difference between $\ell_1$-SDBPG and ProxGT-SA is $-0.01786$ with 95\% interval $[-0.08929,0.05357]$; the difference from Prox-DASA-GT is $-0.03175$ with interval $[-0.10714,0.04762]$.

\FloatBarrier

On Fourier imaging, $\ell_1$-SDBPG reaches the target with fewer rounds, transmitted MiB, and oracle samples than both gradient-tracking baselines. On stochastic recovery, ProxGT-SA and Prox-DASA-GT reach the relaxed target with fewer rounds, transmitted MiB, oracle samples, and wall time. Their final-iterate recovery intervals overlap those of $\ell_1$-SDBPG in Table~\ref{tab:modern-stochastic-main}.

\begin{figure}[htbp]
\centering
\includegraphics[width=0.98\textwidth]{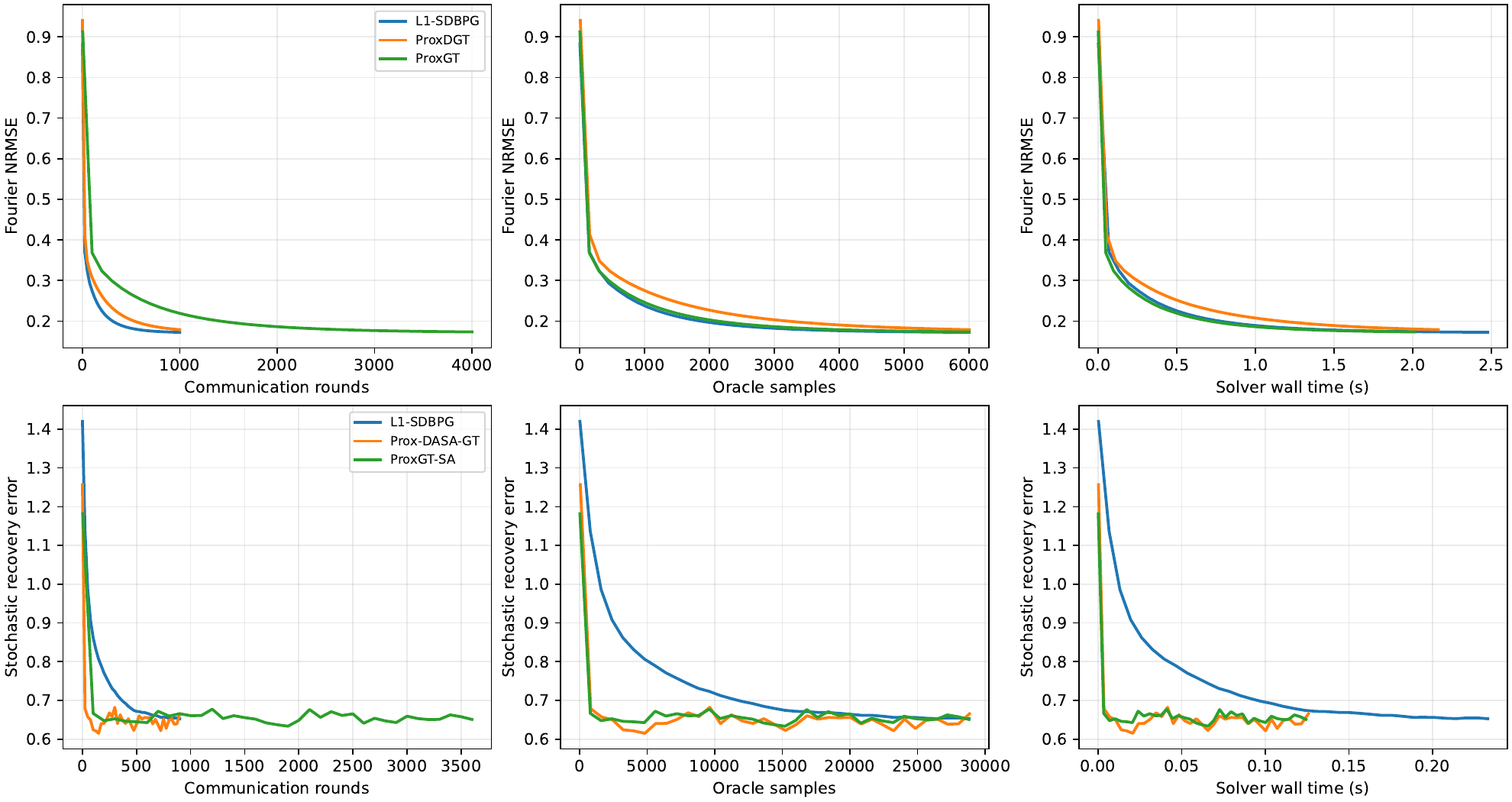}
\caption{Quality against cost over ten held-out evaluation instances. The top row shows Fourier-imaging NRMSE; the bottom row shows stochastic recovery error. Columns use sequential communication rounds, oracle samples, and wall time. Curves are evaluation means at common cost values; lower is preferable.}
\label{fig:modern-cost-quality-axes}
\end{figure}
\FloatBarrier

\subsection{Stochastic Baselines and Residual Trajectories}\label{app:stochastic-sparse-extra}
The noise sweep uses exact quadratic primal and smoothing updates with $(\lambda,\alpha,\beta,c_0,\rho)=(0.0225,0.03,0.85,1.5,2)$ and $\ell_1$ weight $0.02$. These parameters are selected on five moderate-regime validation instances and held fixed across all listed noise settings and both sampling regimes. The unregularized control uses the same parameters with zero penalty. Each setting uses 20 paired evaluation instances and 900 iterations, with the support threshold $0.05$ from Section~\ref{sec:experimental-setup}. Mini-batches sample with replacement, and a batch covering the local sample count uses the exact gradient. Additive Gaussian gradient noise has covariance $(\sigma^2/d)I$, so its expected squared norm is $\sigma^2$. Figures~\ref{fig:stochastic-noise-alignment} and~\ref{fig:stochastic-residual-trajectories} and Table~\ref{tab:stochastic-sparse-combined} use this common configuration; the communication-aware comparison in Table~\ref{tab:modern-stochastic-main} uses the separately selected parameters in Table~\ref{tab:modern-baseline-configs}.

\begin{table}[htbp]
\centering
\caption{Stochastic sparse-recovery paired differences over 20 evaluation instances after 900 iterations. Entries give $\ell_1$-SDBPG minus comparator, with 95\% paired bootstrap intervals in brackets. Lower recovery differences and higher support-F1 differences favor $\ell_1$-SDBPG.}
\label{tab:stochastic-sparse-combined}

\footnotesize
\setlength{\tabcolsep}{4pt}

\begin{subtable}{\textwidth}
\centering
\caption{Noise sweep against unregularized stochastic SDBPG; oracle error is the measured mean squared gradient error of $\ell_1$-SDBPG.}
\label{tab:stochastic-sparse-main}
\begin{tabular}{llccr}
\toprule
Regime & Oracle & Recovery difference [95\% CI] & Support-F1 difference [95\% CI] & Oracle error\\
\midrule
\multirow{8}{*}{Moderate} & Batch 4 & $-0.167\;[-0.196,-0.135]$ & $0.397\;[0.327,0.458]$ & 0.0513\\
 & Batch 5 & $-0.167\;[-0.196,-0.135]$ & $0.399\;[0.328,0.462]$ & 0.0409\\
 & Batch 6 & $-0.168\;[-0.198,-0.137]$ & $0.413\;[0.352,0.468]$ & 0.0342\\
 & Batch 7 & $-0.167\;[-0.197,-0.135]$ & $0.410\;[0.351,0.466]$ & 0.0293\\
 & $\sigma=0.030$ & $-0.167\;[-0.196,-0.134]$ & $0.399\;[0.330,0.460]$ & 0.0009\\
 & $\sigma=0.050$ & $-0.167\;[-0.196,-0.134]$ & $0.399\;[0.328,0.462]$ & 0.0025\\
 & $\sigma=0.080$ & $-0.167\;[-0.196,-0.135]$ & $0.402\;[0.333,0.465]$ & 0.0064\\
 & $\sigma=0.100$ & $-0.167\;[-0.197,-0.134]$ & $0.401\;[0.333,0.464]$ & 0.0100\\
\midrule
\multirow{3}{*}{Undersampled} & Batch 3 & $-0.184\;[-0.228,-0.141]$ & $0.444\;[0.371,0.512]$ & 0.0549\\
 & Batch 4 & $-0.185\;[-0.228,-0.140]$ & $0.443\;[0.369,0.510]$ & 0.0411\\
 & Batch 5 & $-0.184\;[-0.227,-0.141]$ & $0.444\;[0.371,0.509]$ & 0.0330\\
\bottomrule
\end{tabular}
\end{subtable}

\vspace{0.5em}
\begin{subtable}{\textwidth}
\centering
\caption{Moderate-regime comparisons against the same-penalty ProxDGD baseline.}
\label{tab:appendix-stochastic-prox-baselines}
\begin{tabular}{lcc}
\toprule
Oracle & Recovery difference [95\% CI] & Support-F1 difference [95\% CI]\\
\midrule
Batch 4 & $-0.154\;[-0.189,-0.121]$ & $0.230\;[0.170,0.290]$\\
Batch 7 & $-0.153\;[-0.187,-0.119]$ & $0.244\;[0.179,0.307]$\\
$\sigma=0.030$ & $-0.153\;[-0.188,-0.120]$ & $0.241\;[0.182,0.298]$\\
$\sigma=0.100$ & $-0.153\;[-0.188,-0.120]$ & $0.243\;[0.184,0.300]$\\
\bottomrule
\end{tabular}
\end{subtable}
\end{table}
\FloatBarrier

For mini-batch size four in the moderate regime, the paired recovery difference is $-0.167$ with 95\% interval $[-0.196,-0.135]$. Recovery differences are negative and support-F1 differences positive across all listed oracle settings.

Table~\ref{tab:appendix-stochastic-prox-baselines} compares stochastic $\ell_1$-SDBPG with tuned stochastic ProxDGD using the same coordinatewise $\ell_1$ proximal step. Negative recovery differences and positive support-F1 differences favor $\ell_1$-SDBPG in all four moderate-regime oracle settings. Stochastic ProxDGT diverged in several mini-batch settings and is omitted from the aggregate table.

Figure~\ref{fig:stochastic-residual-trajectories} shows the corresponding residual trajectories. Larger fixed mini-batches reduce the smoothing mismatch, and stronger additive noise raises its terminal level. The increasing schedule uses batches of 2 for iterations 1--300 and 4 for iterations 301--600, followed by all 8 local samples. The last stage uses exact local gradients, eliminating oracle noise and accelerating the decay of the smoothing mismatch.

\begin{figure}[htbp]
\centering
\includegraphics[width=0.95\textwidth]{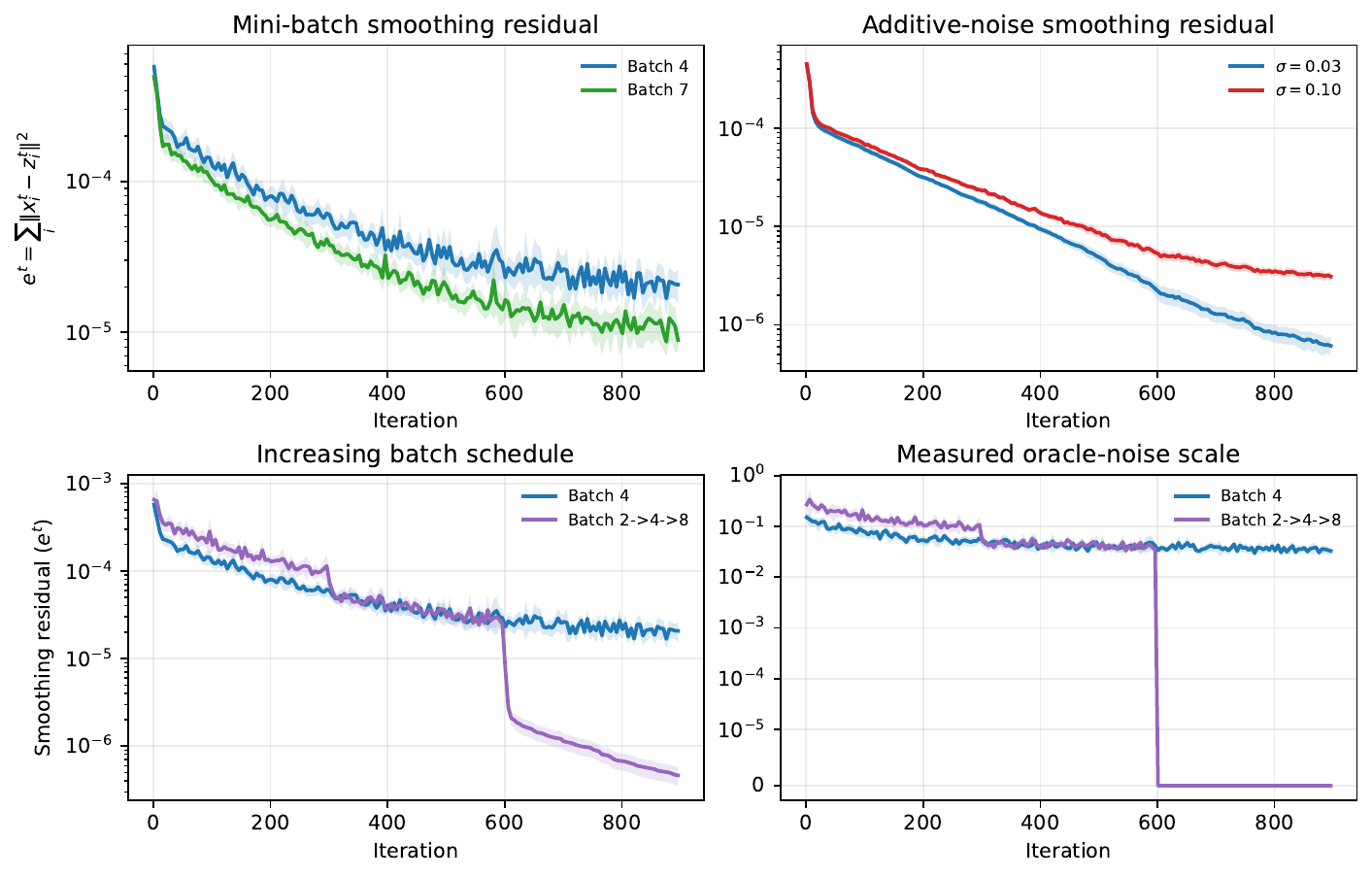}
\caption{Stochastic sparse-recovery smoothing-residual trajectories for
$\ell_1$-SDBPG on the moderate regime over 900 iterations and 20 evaluation instances. Shaded bands show normal-approximation 95\% confidence intervals. The increasing-batch schedule uses $2\to4\to8$ in three 300-iteration stages; the final stage uses exact gradients over all local samples. The fixed-batch curves correspond to the settings summarized in Figure~\ref{fig:stochastic-noise-alignment}.}
\label{fig:stochastic-residual-trajectories}
\end{figure}
\FloatBarrier

\subsection{Sensitivity and Residual Behavior}\label{app:base-sensitivity}
Figure~\ref{fig:appendix-final-controls}(a)--(c) reports squared stationarity residuals as the network size, smoothing weight, and problem dimension vary. The mean residuals remain small over the tested ranges, with variability across instances. Panel~(d) compares empirical hitting times on complete, random, and ring graphs with the theorem's worst-case $O(\epsilon^{-2})$ reference slope.

\begin{figure}[htbp]
\centering
\begin{subfigure}{0.32\textwidth}
\centering
\includegraphics[width=\textwidth]{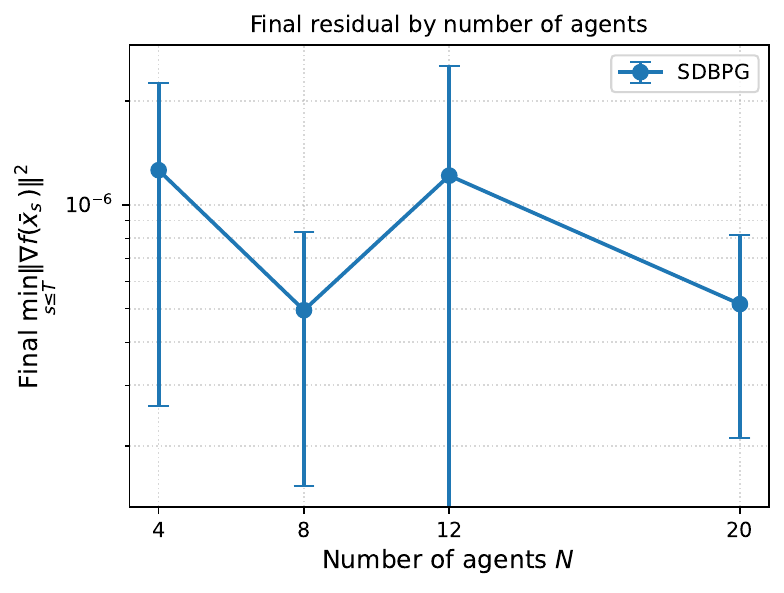}
\caption{Squared stationarity residual versus number of agents $N$.}
\end{subfigure}\hfill
\begin{subfigure}{0.32\textwidth}
\centering
\includegraphics[width=\textwidth]{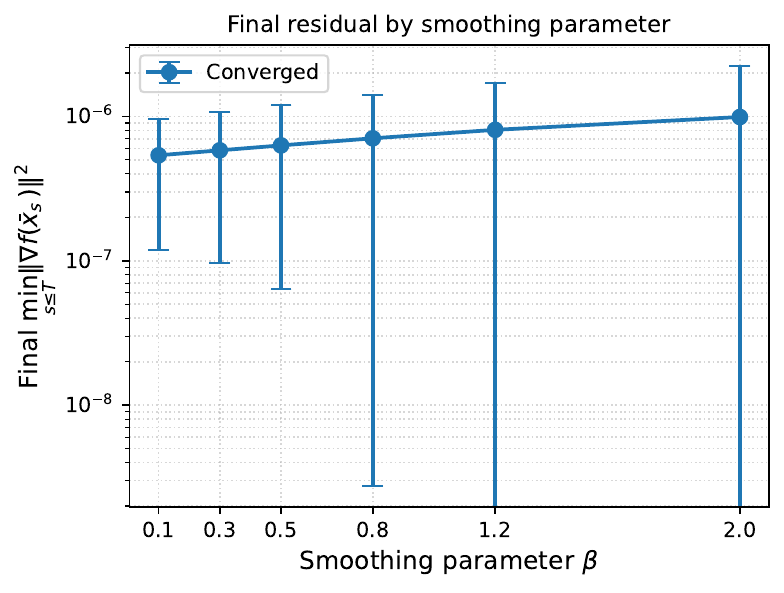}
\caption{Squared stationarity residual versus smoothing weight $\beta$.}
\end{subfigure}\hfill
\begin{subfigure}{0.32\textwidth}
\centering
\includegraphics[width=\textwidth]{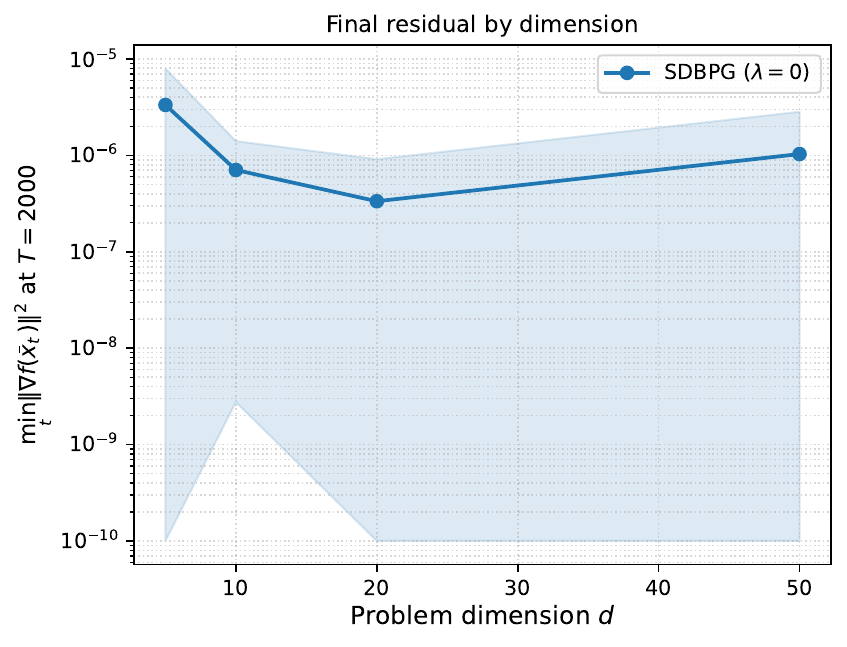}
\caption{Squared stationarity residual versus problem dimension $d$.}
\end{subfigure}
\par\vspace{0.5em}
\begin{subfigure}{0.78\textwidth}
\centering
\includegraphics[width=0.94\textwidth]{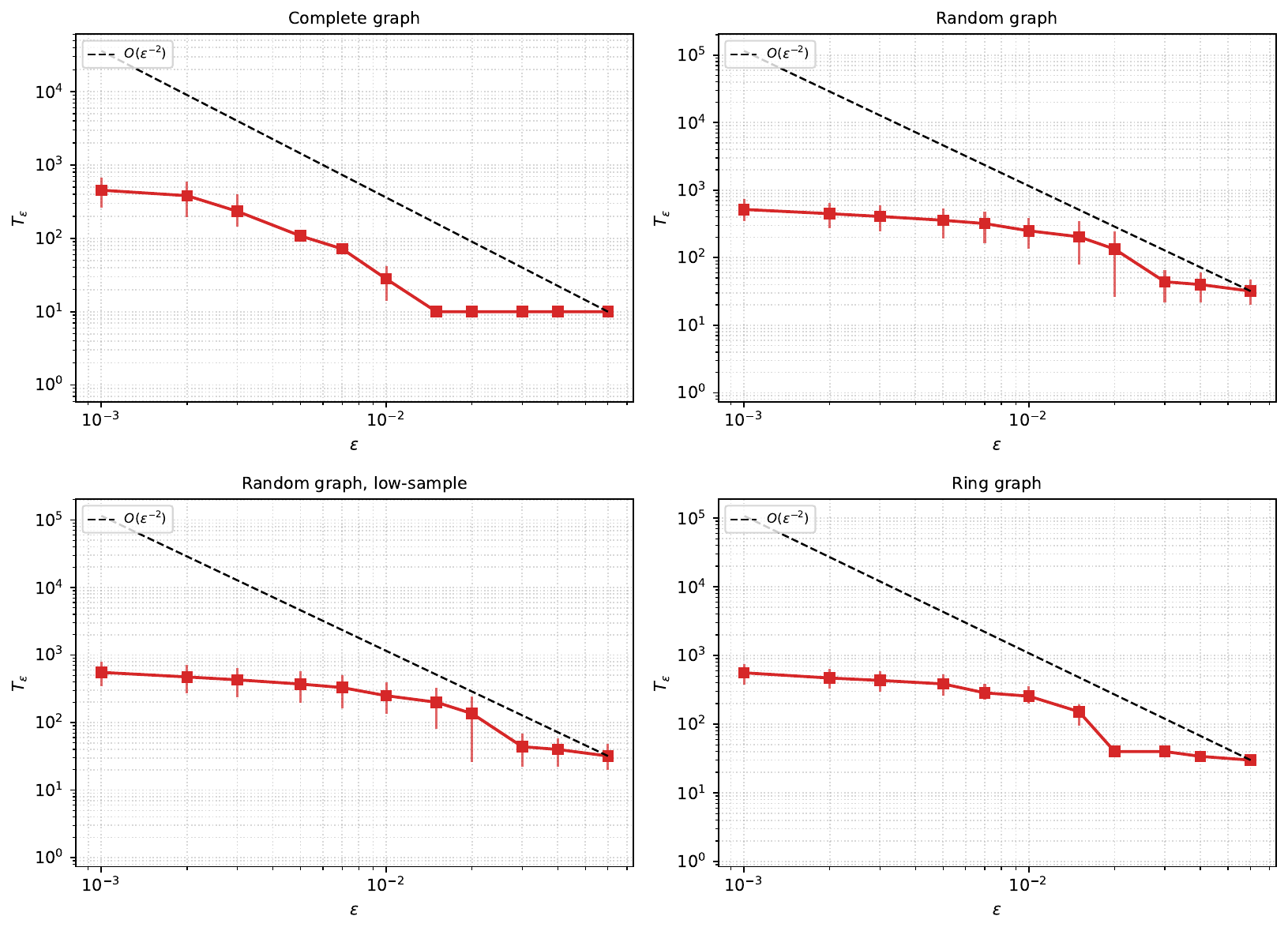}
\caption{Empirical hitting time $T_\epsilon$ on four graph regimes: the first of two consecutive checkpoints at which the maximum nodewise consensus proximal-point distance is at most $\epsilon$. Checkpoints are spaced ten iterations apart after the initial checkpoint. Dashed lines show the theorem's worst-case $O(\epsilon^{-2})$ reference slope; the measured hitting times grow more slowly over the tested range.}
\end{subfigure}
\caption{Sensitivity and residual behavior: squared stationarity residuals versus network size, smoothing weight, and dimension, followed by empirical hitting times.}
\label{fig:appendix-final-controls}
\label{fig:appendix-theory-measures}
\end{figure}
\FloatBarrier

\subsection{Weighted Quadratic Kernel Illustration}\label{app:weighted-kernel-illustration}
Most recovery experiments use the standard quadratic kernel. Figure~\ref{fig:weighted-kernel-illustration} and Table~\ref{tab:weighted-kernel-sweep} instead consider a diagonal quadratic consensus problem with the weighted kernel
\[
\phi_W(x)=\frac12 x^\top W x,
\]
where $W$ is proportional to the diagonal objective curvature and normalized to have average weight one. The Bregman proximal step then applies $W^{-1}$ to the local surrogate gradient, so the kernel acts as a diagonal preconditioner.

\begin{figure}[htbp]
\centering
\includegraphics[width=0.95\textwidth]{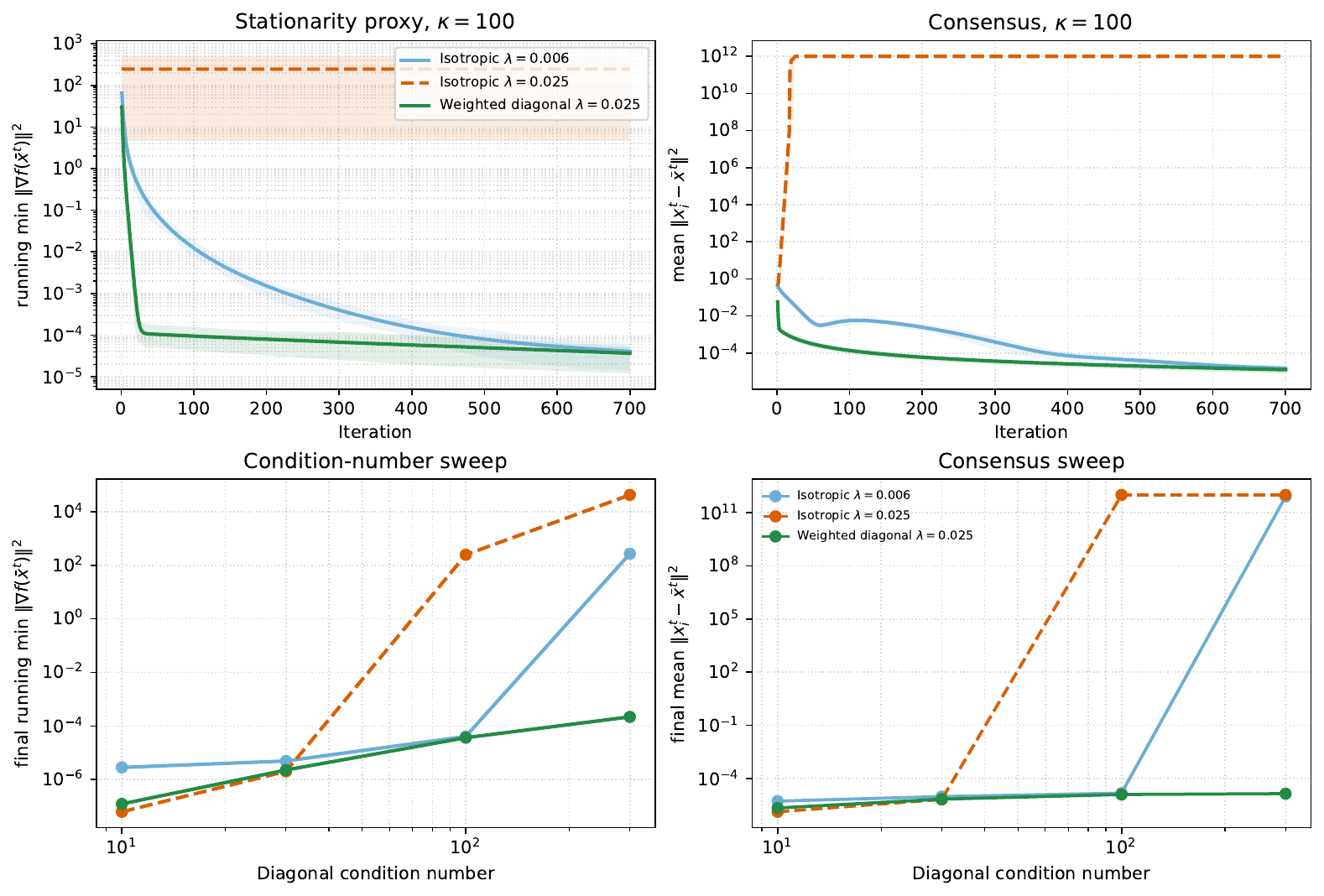}
\caption{Weighted quadratic Bregman kernels on diagonal quadratic consensus problems. At $\kappa=100$ (top), the isotropic kernel is stable with the conservative primal step and unstable with the larger step; the weighted diagonal kernel remains stable with the larger step. The bottom row reports the smallest squared averaged-gradient norm and final consensus error across condition numbers.}
\label{fig:weighted-kernel-illustration}
\end{figure}

On this diagonal quadratic problem, the weighted kernel permits a larger stable nominal primal step. At $\kappa=300$, all weighted-kernel runs remain stable, compared with divergence rates of $80\%$ and $100\%$ for the two isotropic configurations. The weighted-kernel stationarity residual decays more slowly at this condition number.

\begin{table}[htbp]
\centering
\caption{Weighted-kernel condition-number sweep over ten instances. Each entry gives the mean smallest squared averaged-gradient norm observed during the run, followed by the percentage of runs that diverged.}
\label{tab:weighted-kernel-sweep}

\footnotesize
\setlength{\tabcolsep}{5pt}

\begin{tabular}{rccc}
\toprule
& \multicolumn{2}{c}{Isotropic kernel} & Weighted diagonal\\
\cmidrule(lr){2-3}
$\kappa$ & $\lambda=0.006$ & $\lambda=0.025$ & $\lambda=0.025$\\
\midrule
10 & $2.85\times10^{-6}$ (0\%) & $6.31\times10^{-8}$ (0\%) & $1.24\times10^{-7}$ (0\%)\\
30 & $4.96\times10^{-6}$ (0\%) & $2.03\times10^{-6}$ (0\%) & $2.28\times10^{-6}$ (0\%)\\
100 & $4.09\times10^{-5}$ (0\%) & $2.46\times10^{2}$ (100\%) & $3.67\times10^{-5}$ (0\%)\\
300 & $2.69\times10^{2}$ (80\%) & $4.21\times10^{4}$ (100\%) & $2.20\times10^{-4}$ (0\%)\\
\bottomrule
\end{tabular}
\end{table}
\FloatBarrier

\section{Auxiliary Proofs for the Main Statements}\label{app:auxiliary-main-proofs}
The auxiliary results below establish the parameter conditions and are used in the main convergence proofs.

\begin{proof}[Proof of Lemma~\ref{lem:local-variation-from-relative}]
The initial points lie in $K$, and Assumption~\ref{ass:standing-local}(d) makes the selected primal, smoothing, and consensus proximal maps invariant on $K$. Induction places all iterates and comparison points in $K$. Choose a compact convex set $Q$ with $K\subseteq\operatorname{int}Q$ and $Q\Subset U$, where $U$ is the neighborhood in Assumption~\ref{ass:standing-local}(b),(c). Strong convexity of $\phi$ on $U$ gives the Bregman lower bound. Since $\phi\in C^2(U)$, there is $M_\phi<\infty$ with $\|\nabla^2\phi\|\le M_\phi$ on $Q$, which gives the gradient bound for $\phi$. Continuity of each $f_i$ gives its lower bound on $K$.

Set $L_f=M_\phi\max\{L_+,L_-,0\}$. The two relative convexity conditions imply that both $(L_f/2)\|\cdot\|^2-f_i$ and $(L_f/2)\|\cdot\|^2+f_i$ are convex on $Q$. These bounds give local Lipschitz continuity of $\nabla f_i$ using only $f_i\in C^1$: smooth mollifications $f_{i,\varepsilon}$ satisfy $-L_f I\preceq\nabla^2f_{i,\varepsilon}\preceq L_f I$ on $K$ for sufficiently small $\varepsilon$, and their gradients converge uniformly to $\nabla f_i$ on $K$. Passing to the limit in the Lipschitz inequality shows that $\nabla f_i$ is $L_f$-Lipschitz on $K$. Finally, $F_i$ adds terms of the form $\nabla\phi(u)-\nabla\phi(z_j)$ with coefficients $c_0$ and $\rho$. Their local Lipschitz bounds give the stated model-gradient estimate.
\end{proof}

\begin{proof}[Proof of Proposition~\ref{prop:lyapunov-lower-sufficient}]
On the compact set $K$, the smooth objective, composite terms, and Bregman terms appearing in the Lyapunov function are bounded below, except possibly for the dual linear term $\langle y^t,\mathcal B^\top x^t\rangle$. Since $y^0=0$ and $y^{t+1}=y^t+\alpha\mathcal B^\top x^t$, we have $y^t\in\operatorname{range}(\mathcal B^\top)$. On that range, $\|y^t\|\le \sigma_G^{-1}\|\mathcal B y^t\|$, where $\sigma_G$ is the smallest positive singular value of $\mathcal B$. The stacked primal optimality system gives $\mathcal B y^{t+1}=-g^t-p^t-\xi^{t+1}$ in the deterministic composite case, with $g_i^t=\nabla_xF_i(x_i^t,Z_i^t)$ and $p_i^t=\lambda^{-1}(\nabla\phi(x_i^{t+1})-\nabla\phi(x_i^t))$. Compactness and the bounded selected subgradients bound $y^t$, and compactness also bounds $\mathcal B^\top x^t$. The stochastic case is identical after adding the oracle error $\delta^t$ and taking expectations.
\end{proof}

\begin{proof}[Proof of Corollary~\ref{cor:verified-selection-lower-bound}]
On the tail, the points $x_i^t$ lie in the compact set $K$ and $\xi_i^t=\Xi_i(x_i^t)$. Since each $\Xi_i$ is Lipschitz on a set containing the tail points, the selected tail subgradients are bounded. The finite prefix contains only finitely many selected finite vectors, so the full deterministic selected-subgradient sequence is uniformly bounded. Proposition~\ref{prop:lyapunov-lower-sufficient} then gives Assumption~\ref{ass:lyapunov-lower-bound}.
\end{proof}

\begin{proof}[Proof of Corollary~\ref{cor:stochastic-verified-selection-lower-bound}]
The stochastic verified selection gives $\sup_t\E\sum_i\|\xi_i^t\|^2<\infty$. Combining this with the uniform variance bound, Proposition~\ref{prop:lyapunov-lower-sufficient} gives the expected Lyapunov lower bound.
\end{proof}

\begin{proof}[Proof of Proposition~\ref{prop:stability-sufficient}]
Part (a) follows by taking the ordinary gradient selection $\Xi_i=\nabla r_i$, which is Lipschitz on a neighborhood of $K$. In particular, for squared $\ell_2$ regularization $r_i(x)=(\omega_i/2)\|x\|^2$, the selection $\Xi_i(x)=\omega_i x$ gives
\[
\|\Xi_i(x_i^{t+1})-\Xi_i(x_i^t)\|^2
\le \omega_i^2\|x_i^{t+1}-x_i^t\|^2,
\]
and the same Lipschitz bound controls $\|\Xi_i(x_i^t)-\Xi_i(\bar x^t)\|^2$ by $\|x_i^t-\bar x^t\|^2$. Thus the deterministic motion and aggregation defects can be chosen to vanish after fixing the absorption constants.

For part (b), the Lipschitz estimates for the verified active selection absorb its temporal motion and aggregation mismatch on the identified tail, as in Lemma~\ref{comp:lem:verified-selection-absorption}.

For part (c), write $g_i^t=\nabla_xF_i(x_i^t,Z_i^t)$ and $p_i^t=\lambda^{-1}(\nabla\phi(x_i^{t+1})-\nabla\phi(x_i^t))$. Let $P_I$ project onto the common inactive coordinates. On the identified tail, every node is zero on $I$, so $P_I\mathcal B(y^{t+1}-y^t)=0$. Subtracting consecutive primal optimality equations gives
\[
P_I(\xi^{t+1}-\xi^t)=-P_I(g^t-g^{t-1})-P_I(p^t-p^{t-1}).
\]
Lemma~\ref{lem:local-variation-from-relative} bounds the right-hand side by the primal and reference motion terms. For $\ell_1$, the active subgradients are fixed by the common signs, yielding
\[
\sum_i\|\xi_i^{t+1}-\xi_i^t\|^2
\le C(d_x^t+d_x^{t-1}+d_z^{t-1}).
\]
Also $\xi_i^t\in\partial r_i(\bar x^t)$ for each node, so the aggregation defect vanishes on the tail.

For group $\ell_{2,1}$, the active maps $x_g\mapsto\omega_i x_g/\|x_g\|$ are Lipschitz under the stated norm bound and add only a primal-motion term. At $\bar x^t$, choose these active subgradients and keep the selected inactive subgradients, which remain valid on zero groups. The squared aggregation mismatch is then bounded by $C\sum_i\|x_i^t-\bar x^t\|^2\le CS^t$. Fixed absorption constants therefore make both defects zero on the identified tail. The finite temporal prefix consists of finite selected-vector differences. In parts (a),(c), subgradients exist at every averaged point in $K$, so the aggregation prefix is also finite.

In the stochastic case, the inactive-coordinate equation adds $-P_I(\delta^t-\delta^{t-1})$. Its squared norm is controlled by the oracle terms in Definition~\ref{def:subgradient-motion-defects}(c). The same aggregation estimate applies pathwise. Taking expectations and using \eqref{eq:integrable-identification-prefix} bounds the accumulated expected defects. For parts (a),(b), the stochastic verified selection gives the corresponding bounds by the same Lipschitz estimates.
\end{proof}

\begin{proof}[Proof of Proposition~\ref{prop:kkt-bridge}]
The optimality condition for the selected scalar consensus subproblem gives
\[
0\in \sum_i\nabla f_i(u_r)+\sum_i\partial r_i(u_r)
+\sum_i c_0(\nabla\phi(u_r)-\nabla\phi(z_i))
+\sum_i\rho\sum_{j\in\mathcal N_i}(\nabla\phi(u_r)-\nabla\phi(z_j)).
\]
Taking the norm of the Bregman-reference contribution and using the Lipschitz bound on $\nabla\phi$ over $K$ proves the estimate.
\end{proof}

\begin{proof}[Proof of Lemma~\ref{lem:admissible-window-main}]
The inequalities in \eqref{eq:main-admissible-region} give the positive coefficients needed in Lemma~\ref{det:lem:summability}; the condition $c_0+\rho\bar d>L_-$ gives strong convexity on $K$ in Lemma~\ref{det:lem:residual-domination}. For nonemptiness, choose $c_0>L_-$ and $\lambda<(L_++c_0)^{-1}$, then choose $\rho>0$ small enough that $a_x>0$. Set $\tau=\kappa\alpha$ with $\kappa>1$. Since $b=(\tau+\rho/\beta)/\alpha^2=\kappa/\alpha+O(\alpha^{-2})$, the last two inequalities hold for all sufficiently large $\alpha$.
\end{proof}

\begin{proof}[Proof of Lemma~\ref{lem:composite-admissible}]
The same parameter-window argument applies with the constants from Lemma~\ref{comp:lem:consensus-by-steps}. The proximal strong-convexity condition $c_0+\rho\bar d>L_-$ is not needed for the intermediate lifted estimate; it is imposed only when Theorem~\ref{thm:composite-main} is invoked.
\end{proof}

\begin{proof}[Proof of Lemma~\ref{lem:stochastic-admissible}]
Choose $c_0,\rho,\lambda$ so that $a_x>0$, then choose $\chi\in(0,\mu a_x)$, so $a_x^s>0$. With $\tau=\kappa\alpha$ and $\kappa>1$, the same limit $b\to0$ as $\alpha\to\infty$ gives the last two strict inequalities for all sufficiently large $\alpha$.
\end{proof}

\section{Proof of Deterministic SDBPG Convergence}\label{app:deterministic-proof}
The proof of Theorem~\ref{thm:deterministic-main} combines one-step Bregman descent, a dual-corrected Lyapunov recursion, a graph-disagreement bound, and a bound on the proximal residual. We first fix the graph notation.

\paragraph{Graph notation.}
Let $m=|\mathcal E|$ and let $\mathcal B:\R^{md}\to\R^{Nd}$ be the oriented block incidence operator. Set $s^t=\mathcal B^\top x^t$ and $S^t=\|s^t\|^2$. The stacked dual update is
\[
y^{t+1}=y^t+\alpha s^t.
\]
Because $y^0=0$, $y^t\in\operatorname{range}(\mathcal B^\top)$. Connectivity gives $\ker(\mathcal B^\top)=\{\mathbf 1\otimes u:u\in\R^d\}$. Let $\sigma_G>0$ be the smallest positive singular value of $\mathcal B$. The range bound and graph Poincare inequality are
\begin{equation}\label{det:eq:graph-range}
\|w\|\le\sigma_G^{-1}\|\mathcal B w\|,
\qquad w\in\operatorname{range}(\mathcal B^\top),
\end{equation}
and
\begin{equation}\label{det:eq:poincare}
\sum_i\|x_i-\bar x\|^2\le \sigma_G^{-2}\|\mathcal B^\top x\|^2.
\end{equation}
On the compact set $K$, define
\begin{equation}\label{det:eq:compact-diameter}
D_K:=\sup_{u,v\in K}D_\phi(u,v)<\infty.
\end{equation}

The primal descent inequality combines the upper relative-smoothness bound with the Bregman three-point identity.

\begin{lemma}[One-step surrogate descent]\label{det:lem:surrogate-descent}
Let
\[
a_x:=\frac1\lambda-(L_++c_0+\rho d_{\max}),
\qquad
a_z:=c_0+\rho d_{\min},
\qquad
a_e:=\frac{c_0}{\beta}.
\]
For $F^t:=\sum_iF_i(x_i^t,Z_i^t)$,
\begin{equation}\label{det:eq:surrogate-descent}
F^{t+1}-F^t
\le
-a_xd_x^t-a_zd_z^t-a_e e^t
-\langle y^{t+1},s^{t+1}-s^t\rangle
+\frac{\rho}{\beta}S^{t+1}.
\end{equation}
\end{lemma}

\begin{proof}
Fix $Z_i^t$. Relative smoothness of $f_i$ and the identity
\[
D_\phi(u,z)-D_\phi(v,z)-\langle\nabla\phi(v)-\nabla\phi(z),u-v\rangle=D_\phi(u,v)
\]
give
\[
F_i(x_i^{t+1},Z_i^t)-F_i(x_i^t,Z_i^t)
\le
\langle\nabla_xF_i(x_i^t,Z_i^t),x_i^{t+1}-x_i^t\rangle
+(L_++c_0+\rho d_i)D_\phi(x_i^{t+1},x_i^t).
\]
In the smooth case $r_i\equiv0$, the primal optimality condition is
\[
\nabla K_i(x_i^t;Y_i^{t+1},Z_i^t)
+\frac1\lambda(\nabla\phi(x_i^{t+1})-\nabla\phi(x_i^t))=0.
\]
Taking the inner product with $x_i^{t+1}-x_i^t$ gives
\[
\langle\nabla K_i(x_i^t;Y_i^{t+1},Z_i^t),x_i^{t+1}-x_i^t\rangle
=-\frac1\lambda\big(D_\phi(x_i^{t+1},x_i^t)+D_\phi(x_i^t,x_i^{t+1})\big)
\le -\frac1\lambda D_\phi(x_i^{t+1},x_i^t).
\]
Since $\nabla K_i=\nabla_xF_i+\sum_{j\in\mathcal N_i}y_{ij}^{t+1}$, summing over $i$ yields
\begin{equation}\label{det:eq:x-change}
\sum_i\big(F_i(x_i^{t+1},Z_i^t)-F_i(x_i^t,Z_i^t)\big)
\le -a_xd_x^t-\langle y^{t+1},s^{t+1}-s^t\rangle.
\end{equation}
The smoothing optimality condition is
\[
z_i^{t+1}-x_i^{t+1}+\beta(\nabla\phi(z_i^{t+1})-\nabla\phi(z_i^t))=0.
\]
Using the Bregman three-point identity, the own-reference change gives
\[
D_\phi(x_i^{t+1},z_i^{t+1})-D_\phi(x_i^{t+1},z_i^t)
=-D_\phi(z_i^{t+1},z_i^t)-\frac1\beta\|x_i^{t+1}-z_i^{t+1}\|^2.
\]
For neighbor references, the same identity gives
\[
D_\phi(x_i^{t+1},z_j^{t+1})-D_\phi(x_i^{t+1},z_j^t)
=-D_\phi(z_j^{t+1},z_j^t)
+\frac1\beta\langle z_j^{t+1}-x_j^{t+1},x_i^{t+1}-z_j^{t+1}\rangle.
\]
After summing over edges and using Young's inequality, the neighbor-reference contribution is bounded by
\[
-\rho\sum_j d_jD_\phi(z_j^{t+1},z_j^t)+\frac{\rho}{\beta}S^{t+1}.
\]
Combining this estimate with the own-reference term gives
\begin{equation}\label{det:eq:z-change}
\sum_i\big(F_i(x_i^{t+1},Z_i^{t+1})-F_i(x_i^{t+1},Z_i^t)\big)
\le -a_zd_z^t-a_ee^t+\frac{\rho}{\beta}S^{t+1}.
\end{equation}
Adding \eqref{det:eq:x-change} and \eqref{det:eq:z-change} proves the claim.
\end{proof}

Adding the dual update gives the following Lyapunov inequality. The term $\tau S^t$ later absorbs the future consensus error.

\begin{lemma}[Baseline Lyapunov inequality]\label{det:lem:baseline}
For $\tau>0$, define
\[
\mathcal L^t:=F^t+\langle y^t,s^t\rangle+\tau S^t.
\]
Then
\begin{equation}\label{det:eq:baseline}
\mathcal L^{t+1}-\mathcal L^t
\le
-a_xd_x^t-a_zd_z^t-a_e e^t
-(\tau-\alpha)S^t+qS^{t+1},
\qquad q:=\tau+\frac{\rho}{\beta}.
\end{equation}
\end{lemma}

\begin{proof}
The identity
\[
\langle y^{t+1},s^{t+1}\rangle-\langle y^t,s^t\rangle
=\langle y^{t+1},s^{t+1}-s^t\rangle+\alpha S^t
\]
follows from $y^{t+1}-y^t=\alpha s^t$. Add this identity and $\tau(S^{t+1}-S^t)$ to \eqref{det:eq:surrogate-descent}.
\end{proof}

The future disagreement term is bounded by comparing two consecutive primal optimality systems.

\begin{lemma}[Consensus controlled by step errors]\label{det:lem:consensus-by-steps}
There exist finite constants $\widehat\Gamma_x^+,\widehat\Gamma_x^-,\widehat\Gamma_z$ such that, for every $t\ge1$,
\begin{equation}\label{det:eq:consensus-by-steps}
S^t\le\frac1{\alpha^2}\left(\widehat\Gamma_x^+d_x^t+\widehat\Gamma_x^-d_x^{t-1}+\widehat\Gamma_zd_z^{t-1}\right).
\end{equation}
The constants depend only on the localized model constants, $\lambda,c_0,\rho$, and the graph constants, and are independent of the dual stepsize $\alpha$.
\end{lemma}

\begin{proof}
Let
\[
g_i^t:=\nabla_xF_i(x_i^t,Z_i^t),
\qquad
p_i^t:=\frac1\lambda(\nabla\phi(x_i^{t+1})-\nabla\phi(x_i^t)).
\]
The stacked primal optimality system is
\begin{equation}\label{det:eq:stacked-opt}
g^t+\mathcal B y^{t+1}+p^t=0.
\end{equation}
Subtracting the relation at time $t-1$ yields
\[
\mathcal B(y^{t+1}-y^t)=-(g^t-g^{t-1})-(p^t-p^{t-1}).
\]
Since $y^{t+1}-y^t=\alpha s^t$ and $s^t\in\operatorname{range}(\mathcal B^\top)$, \eqref{det:eq:graph-range} gives
\[
\alpha^2\sigma_G^2S^t\le 2\|g^t-g^{t-1}\|^2+2\|p^t-p^{t-1}\|^2.
\]
Lemma~\ref{lem:local-variation-from-relative} and the Bregman lower bound imply
\[
\|g^t-g^{t-1}\|^2\le C_1d_x^{t-1}+C_2d_z^{t-1},
\qquad
\|p^t-p^{t-1}\|^2\le C_3(d_x^t+d_x^{t-1}).
\]
Absorbing constants proves the estimate.
\end{proof}

The telescoping argument also requires a uniform lower bound for the Lyapunov sequence.

\begin{lemma}[Lower boundedness]\label{det:lem:lower-bound}
There exists $\mathcal L_\star>-\infty$ such that $\mathcal L^t\ge\mathcal L_\star$ for all $t$ in the smooth deterministic case.
\end{lemma}

\begin{proof}
Continuity on the compact set $K$ bounds each $f_i$ below, and the Bregman terms are nonnegative, so $F^t$ is bounded below. It remains to control the dual linear term. From \eqref{det:eq:stacked-opt} at time $t-1$,
\[
\mathcal B y^t=-g^{t-1}-p^{t-1}.
\]
The right-hand side is bounded on $K$. Since $y^t\in\operatorname{range}(\mathcal B^\top)$, \eqref{det:eq:graph-range} bounds $y^t$. Compactness bounds $s^t$, and therefore $\langle y^t,s^t\rangle$ is bounded below. The term $\tau S^t$ is nonnegative.
\end{proof}

Under the admissible parameter inequalities, these estimates make the accumulated primal and smoothing motion finite.

\begin{lemma}[Summability of step errors]\label{det:lem:summability}
Under \eqref{eq:main-admissible-region}, there is $M_1<\infty$ such that, for every $T\ge1$,
\begin{equation}\label{det:eq:summable}
\sum_{t=0}^{T-1}(d_x^t+d_z^t+e^t)\le M_1.
\end{equation}
\end{lemma}

\begin{proof}
Sum \eqref{det:eq:baseline} from $t=0$ to $T-1$ and drop $-(\tau-\alpha)S^t$. Lemma~\ref{det:lem:consensus-by-steps}, with $b=q/\alpha^2$, gives
\[
q\sum_{t=0}^{T-1}S^{t+1}
\le b(\widehat\Gamma_x^++\widehat\Gamma_x^-)\sum_{t=0}^{T-1}d_x^t
+b\widehat\Gamma_z\sum_{t=0}^{T-1}d_z^t
+b\widehat\Gamma_x^+d_x^T.
\]
The boundary term is at most $b\widehat\Gamma_x^+ND_K$. Using Lemma~\ref{det:lem:lower-bound} and the strict positivity in \eqref{eq:main-admissible-region} proves \eqref{det:eq:summable}.
\end{proof}

The following lemma bounds the residual in Theorem~\ref{thm:deterministic-main} by the summable quantities.

\begin{lemma}[Residual domination]\label{det:lem:residual-domination}
There exists $H<\infty$ such that, for every $t\ge1$,
\begin{equation}\label{det:eq:residual-domination}
\mathcal G_0(z^t)^2\le H(e^{t-1}+d_x^{t-1}+d_z^{t-1}+S^t).
\end{equation}
\end{lemma}

\begin{proof}
Let $u_0^t$ be the common component of the consensus proximal point for $z^t$ and define
\[
\psi_t(u):=\sum_{i=1}^N F_i(u,Z_i^t).
\]
The total coefficient of $\phi(u)$ in $\psi_t$ is $Nc_0+\rho\sum_i d_i=N(c_0+\rho\bar d)$. Relative weak convexity and $c_0+\rho\bar d>L_-$ therefore make $\psi_t$ strongly convex on $K$ with modulus
\[
m=N\mu(c_0+\rho\bar d-L_-)>0.
\]
For a differentiable $m$-strongly convex function, $\|v-u_0^t\|\le m^{-1}\|\nabla\psi_t(v)\|$. Apply this with $v=\bar x^t$. Summing \eqref{det:eq:stacked-opt} at time $t-1$ cancels the edge multipliers and gives
\[
\sum_i\nabla_xF_i(x_i^{t-1},Z_i^{t-1})+\sum_i p_i^{t-1}=0.
\]
Thus $\|\nabla\psi_t(\bar x^t)\|^2$ is bounded by $C(d_x^{t-1}+d_z^{t-1}+S^t)$ using the local model-gradient variation, the Bregman lower bound, and \eqref{det:eq:poincare}. Hence
\[
\|\bar x^t-u_0^t\|^2\le C(d_x^{t-1}+d_z^{t-1}+S^t).
\]
Finally,
\[
\|z_i^t-u_0^t\|^2\le3\|z_i^t-x_i^t\|^2+3\|x_i^t-\bar x^t\|^2+3\|\bar x^t-u_0^t\|^2,
\]
and the definitions of $e^{t-1}$ and $S^t$ prove the claim.
\end{proof}

\begin{proof}[Proof of Theorem~\ref{thm:deterministic-main}]
Lemma~\ref{det:lem:summability} bounds $\sum_{t=0}^{T-1}(d_x^t+d_z^t+e^t)$. Lemma~\ref{det:lem:consensus-by-steps} then bounds $\sum_{t=1}^{T}S^t$. Summing Lemma~\ref{det:lem:residual-domination} gives $\sum_{t=1}^{T}\mathcal G_0(z^t)^2\le M_2$ for a constant independent of $T$. Therefore
\[
\min_{1\le k\le T}\mathcal G_0(z^k)^2\le \frac1T\sum_{t=1}^{T}\mathcal G_0(z^t)^2\le\frac{M_2}{T}.
\]
The $O(\epsilon^{-2})$ statement follows by taking $T\ge M_2/\epsilon^2$.
\end{proof}

\section{Proof of Composite Deterministic SDBPG Convergence}\label{app:composite-proof}
The deterministic composite proof adds the convex regularizers to the primal energy and otherwise uses the Lyapunov sequence, graph notation, and step quantities from Appendix~\ref{app:deterministic-proof}. Define
\[
H^t:=\sum_i\big(F_i(x_i^t,Z_i^t)+r_i(x_i^t)\big),
\qquad
\mathcal L_r^t:=H^t+\langle y^t,s^t\rangle+\tau S^t.
\]

Convexity of $r_i$ preserves the descent inequality when the smooth optimality equation is replaced by the composite variational inequality.

\begin{lemma}[Composite one-step surrogate descent]\label{comp:lem:surrogate-descent}
For every $t\ge0$,
\begin{equation}\label{comp:eq:surrogate-descent}
H^{t+1}-H^t
\le
-a_xd_x^t-a_zd_z^t-a_e e^t
-\langle y^{t+1},s^{t+1}-s^t\rangle
+\frac{\rho}{\beta}S^{t+1}.
\end{equation}
\end{lemma}

\begin{proof}
The descent argument of Lemma~\ref{det:lem:surrogate-descent} carries over with the composite variational inequality
\[
\left\langle
\nabla K_i(x_i^t;Y_i^{t+1},Z_i^t)+\frac1\lambda(\nabla\phi(x_i^{t+1})-\nabla\phi(x_i^t)),
 x-x_i^{t+1}\right\rangle+r_i(x)-r_i(x_i^{t+1})\ge0.
\]
Taking $x=x_i^t$ gives the same descent inequality with the difference $r_i(x_i^{t+1})-r_i(x_i^t)$ included. The reference-change estimate is unchanged.
\end{proof}

The same dual correction gives the composite Lyapunov recursion.

\begin{lemma}[Composite baseline Lyapunov inequality]\label{comp:lem:baseline}
For every $t\ge0$,
\begin{equation}\label{comp:eq:baseline}
\mathcal L_r^{t+1}-\mathcal L_r^t
\le
-a_xd_x^t-a_zd_z^t-a_e e^t
-(\tau-\alpha)S^t+qS^{t+1}.
\end{equation}
\end{lemma}

\begin{proof}
Add the dual identity used in Lemma~\ref{det:lem:baseline} and $\tau(S^{t+1}-S^t)$ to \eqref{comp:eq:surrogate-descent}.
\end{proof}

The graph-disagreement estimate now includes the change in the selected subgradients. The part not controlled by step errors is $\eta_t^r$.

\begin{lemma}[Composite consensus controlled by step errors]\label{comp:lem:consensus-by-steps}
There exist finite constants $\widehat\Gamma_{r,x}^+,\widehat\Gamma_{r,x}^-,\widehat\Gamma_{r,z},\widehat\Gamma_{r,\eta}$ such that, for $t\ge1$,
\begin{equation}\label{comp:eq:consensus-by-steps}
S^t\le\frac1{\alpha^2}\left(\widehat\Gamma_{r,x}^+d_x^t+\widehat\Gamma_{r,x}^-d_x^{t-1}+\widehat\Gamma_{r,z}d_z^{t-1}+\widehat\Gamma_{r,\eta}\eta_t^r\right).
\end{equation}
\end{lemma}

\begin{proof}
With $g_i^t=\nabla_xF_i(x_i^t,Z_i^t)$ and $p_i^t=\lambda^{-1}(\nabla\phi(x_i^{t+1})-\nabla\phi(x_i^t))$, the composite optimality system is
\begin{equation}\label{comp:eq:stacked-opt}
g^t+\mathcal B y^{t+1}+p^t+\xi^{t+1}=0.
\end{equation}
Subtract this relation at time $t-1$ and use $y^{t+1}-y^t=\alpha s^t$. The graph range inequality gives
\[
\alpha^2\sigma_G^2S^t\le C\big(\|g^t-g^{t-1}\|^2+\|p^t-p^{t-1}\|^2+\|\xi^{t+1}-\xi^t\|^2\big).
\]
The first two terms are bounded as in Lemma~\ref{det:lem:consensus-by-steps}; the last term is split into the absorbed part and the defect $\eta_t^r$ by Definition~\ref{def:subgradient-motion-defects}(b).
\end{proof}

The general lifted and proximal estimates use Assumption~\ref{ass:lyapunov-lower-bound}. In the structured finite-defect regimes, Proposition~\ref{prop:lyapunov-lower-sufficient} gives this lower bound for bounded selected subgradients, and Corollary~\ref{cor:verified-selection-lower-bound} gives it for verified local selection.

\begin{lemma}[Composite lower boundedness]\label{comp:lem:lower-bound}
Assume the deterministic composite Lyapunov sequence satisfies Assumption~\ref{ass:lyapunov-lower-bound}. Then there exists $\mathcal L_{r,\star}>-\infty$ such that $\mathcal L_r^t\ge\mathcal L_{r,\star}$.
\end{lemma}

\begin{proof}
This is exactly Assumption~\ref{ass:lyapunov-lower-bound} for the deterministic composite estimate.
\end{proof}

The composite summability estimate differs from the smooth estimate by the accumulated defect term.

\begin{lemma}[Composite summability]\label{comp:lem:summability}
Under \eqref{eq:composite-admissible-region}, there exist $M_r,C_{r,\eta}<\infty$ such that
\begin{equation}\label{comp:eq:summable}
\sum_{t=0}^{T-1}(d_x^t+d_z^t+e^t)\le M_r+C_{r,\eta}\sum_{t=1}^{T}\eta_t^r.
\end{equation}
\end{lemma}

\begin{proof}
Sum \eqref{comp:eq:baseline} and use Lemma~\ref{comp:lem:consensus-by-steps} in the same telescoping argument as Lemma~\ref{det:lem:summability}. Lemma~\ref{comp:lem:lower-bound} bounds the telescoping term from below, and the admissible inequalities make the step coefficients positive.
\end{proof}

The lifted residual can be controlled without aggregating node subgradients at the averaged primal point.

\begin{lemma}[Lifted residual domination]\label{comp:lem:lifted-residual-domination}
There exists $H_{\mathrm{lift}}<\infty$ such that, for every $t\ge1$,
\begin{equation}\label{comp:eq:lifted-residual-domination}
\mathcal R_{\mathrm{lift},r}(x^{t+1},y^{t+1})^2
\le H_{\mathrm{lift}}(d_x^t+e^{t-1}+S^t+S^{t+1}).
\end{equation}
\end{lemma}

\begin{proof}
From \eqref{comp:eq:stacked-opt}, $\xi^{t+1}\in\partial r(x^{t+1})$ and
\[
g^t+\mathcal B y^{t+1}+p^t+\xi^{t+1}=0.
\]
Therefore the lifted stationarity distance at $(x^{t+1},y^{t+1})$ is bounded by
\[
\|\nabla f(x^{t+1})-g^t-p^t\|^2.
\]
The difference between $g^t$ and $\nabla f(x^{t+1})$ consists of the smooth-gradient step error and the Bregman-reference terms $\nabla\phi(x_i^t)-\nabla\phi(z_j^t)$. The local variation bounds, the Bregman lower bound, and the graph Poincare inequality give
\[
\|\nabla f(x^{t+1})-g^t-p^t\|^2\le C(d_x^t+e^{t-1}+S^t).
\]
Adding the feasibility term $S^{t+1}$ proves the lemma.
\end{proof}

The stronger proximal residual uses the aggregation defect, because its stationarity measure is evaluated at a common point.

\begin{lemma}[Composite proximal residual domination]\label{comp:lem:residual-domination}
Assume $c_0+\rho\bar d>L_-$. For the deterministic aggregation defect in Definition~\ref{def:aggregation-defects}(a), there exist $H_r,C_{r,\zeta}<\infty$ such that, for every $t\ge1$,
\begin{equation}\label{comp:eq:residual-domination}
\mathcal G_r(z^t)^2
\le H_r(e^{t-1}+d_x^{t-1}+d_z^{t-1}+S^t)+C_{r,\zeta}\zeta_t^r.
\end{equation}
\end{lemma}

\begin{proof}
Let $u_r^t$ be the selected common component of the consensus proximal point and define
\[
\psi_{r,t}(u)=\sum_i\big(F_i(u,Z_i^t)+r_i(u)\big).
\]
The Bregman terms have total coefficient $N(c_0+\rho\bar d)$, and the regularizers are convex. Thus $\psi_{r,t}$ is strongly convex on $K$ with modulus
\[
m_r=N\mu(c_0+\rho\bar d-L_-)>0.
\]
Set $h_t(u)=\sum_i F_i(u,Z_i^t)$. Its gradient is $m_r$-strongly monotone on $K$, and each $\partial r_i$ is monotone. The optimality condition at $u_r^t$ therefore gives, for $v\in K$,
\[
\|v-u_r^t\|\le m_r^{-1}\operatorname{dist}\!\left(0,\nabla h_t(v)+\sum_i\partial r_i(v)\right).
\]
Apply this inequality at $v=\bar x^t$. For any $\varepsilon>0$, Definition~\ref{def:aggregation-defects}(a) gives $\bar\xi^t\in\sum_i\partial r_i(\bar x^t)$ with $\|\bar\xi^t-\sum_i\xi_i^t\|^2\le M_{C\xi}S^t+\zeta_t^r+\varepsilon$. Summing \eqref{comp:eq:stacked-opt} at time $t-1$ cancels the edge multipliers and gives
\[
\sum_i\nabla_xF_i(x_i^{t-1},Z_i^{t-1})+\sum_i p_i^{t-1}+\sum_i\xi_i^t=0.
\]
The local variation bounds, the Bregman lower bound, and the aggregation estimate imply
\[
\|\bar x^t-u_r^t\|^2\le C(d_x^{t-1}+d_z^{t-1}+S^t+\zeta_t^r+\varepsilon).
\]
Letting $\varepsilon\downarrow0$ and comparing $z_i^t$ with $u_r^t$ as in Lemma~\ref{det:lem:residual-domination} proves the claim.
\end{proof}

\begin{proposition}[Intermediate deterministic composite lifted estimate]\label{prop:deterministic-lifted-estimate}
For any deterministic composite SDBPG trajectory satisfying Assumptions~\ref{ass:standing-local}, \ref{ass:composite-regularizer}, and~\ref{ass:lyapunov-lower-bound}, with subgradient-motion defects defined by Definition~\ref{def:subgradient-motion-defects}(b) and parameters compatible with Lemma~\ref{lem:composite-admissible},
\[
\min_{1\le k\le T}\mathcal R_{\mathrm{lift},r}(x^k,y^k)^2
\le
\frac{\Delta_{\mathrm{lift}}+C_\eta\sum_{t=1}^{T}\eta_t^r}{T},
\qquad T\ge1.
\]
If $\sum_t\eta_t^r<\infty$, the squared lifted residual is $O(T^{-1})$ in the best-iterate sense.
\end{proposition}

\begin{proof}
Lemma~\ref{comp:lem:summability} bounds the accumulated step errors by a constant plus $C\sum_t\eta_t^r$. Lemma~\ref{comp:lem:consensus-by-steps} gives the same type of bound for $\sum_tS^t$. Lemma~\ref{comp:lem:lifted-residual-domination} controls the residuals from index $k=2$ onward; the finite residual at $k=1$ is absorbed into $\Delta_{\mathrm{lift}}$. Summing the domination estimates and dividing by $T$ proves the proposition.
\end{proof}

\begin{corollary}[Asymptotic lifted residual and cluster points]\label{cor:asymptotic-cluster}
Under Proposition~\ref{prop:deterministic-lifted-estimate}, if $\sum_t\eta_t^r<\infty$, then $\mathcal R_{\mathrm{lift},r}(x^t,y^t)\to0$. If, in addition, $(x^{t_j},y^{t_j})\to(x^\star,y^\star)$ along a subsequence and the graph of $\partial r$ is closed along that subsequence, then
\[
0\in\nabla f(x^\star)+\mathcal B y^\star+\partial r(x^\star),
\qquad
\mathcal B^\top x^\star=0.
\]
\end{corollary}

\begin{proof}
The summability argument in Proposition~\ref{prop:deterministic-lifted-estimate} gives $\sum_t\mathcal R_{\mathrm{lift},r}(x^t,y^t)^2<\infty$ when $\sum_t\eta_t^r<\infty$, so the residual converges to zero. Passing to a convergent subsequence and using closedness of the graph of $\partial r$ gives the displayed lifted Karush--Kuhn--Tucker (KKT) system.
\end{proof}

\begin{lemma}[Verified selection controls deterministic defect tails]\label{comp:lem:verified-selection-absorption}
A verified local selection in the sense of Definition~\ref{def:verified-local-selection} admits fixed absorption constants such that
\[
\eta_t^r=\zeta_t^r=0\quad(t\ge T_{\rm id}),
\qquad
\sum_{t=1}^{\infty}\eta_t^r<\infty.
\]
Finite aggregation defects before $T_{\rm id}$ also give Assumption~\ref{ass:finite-defect-accumulation}(a).
\end{lemma}

\begin{proof}
Let $L_\Xi$ be a Lipschitz constant for the tail selections after increasing it to dominate all active maps $\Xi_i$. For $t\ge T_{\rm id}$,
\[
\sum_i\|\xi_i^{t+1}-\xi_i^t\|^2
=\sum_i\|\Xi_i(x_i^{t+1})-\Xi_i(x_i^t)\|^2
\le L_\Xi^2\sum_i\|x_i^{t+1}-x_i^t\|^2
\le \frac{2L_\Xi^2}{\mu}d_x^t.
\]
Choosing the absorption constant $M_{T\xi}$ in Definition~\ref{def:subgradient-motion-defects}(b) at least $2L_\Xi^2/\mu$ makes $\eta_t^r=0$ on the identified tail; the finite prefix $t<T_{\rm id}$ contributes only finitely many (finite) terms to $\sum_{t}\eta_t^r$.

For aggregation, choose $\bar\xi_i^t=\Xi_i(\bar x^t)\in\partial r_i(\bar x^t)$. The graph Poincare inequality and the same Lipschitz bound give, for $t\ge T_{\rm id}$,
\[
\begin{aligned}
\operatorname{dist}^2\!\left(\sum_i\xi_i^t,\sum_i\partial r_i(\bar x^t)\right)
&\le
\left\|\sum_i(\Xi_i(x_i^t)-\Xi_i(\bar x^t))\right\|^2\\
&\le C\sum_i\|x_i^t-\bar x^t\|^2
\le C S^t .
\end{aligned}
\]
Choosing $M_{C\xi}$ to dominate the aggregation constant gives $\zeta_t^r=0$ for $t\ge T_{\rm id}$. When the pre-identification aggregation defects are finite, their finite sum completes Assumption~\ref{ass:finite-defect-accumulation}(a).
\end{proof}

\begin{proof}[Proof of Theorem~\ref{thm:composite-main}]
Use Lemma~\ref{comp:lem:residual-domination} in place of Lemma~\ref{comp:lem:lifted-residual-domination} and combine it with Lemmas~\ref{comp:lem:summability} and~\ref{comp:lem:consensus-by-steps}. Assumption~\ref{ass:lyapunov-lower-bound} supplies the lower bound needed in Lemma~\ref{comp:lem:summability}, and Definition~\ref{def:aggregation-defects}(a) supplies the aggregation defect in Lemma~\ref{comp:lem:residual-domination}. Summing the residual-domination estimate and dividing by $T$ gives
\[
\min_{1\le k\le T}\mathcal G_r(z^k)^2
\le
\frac{\Delta_r+C_r\sum_{t=1}^{T}(\eta_t^r+\zeta_t^r)}{T}.
\]
\end{proof}

\begin{proof}[Proof of Corollary~\ref{cor:finite-transition-composite-rate}]
Assumption~\ref{ass:finite-defect-accumulation}(a) states exactly $\sum_t(\eta_t^r+\zeta_t^r)<\infty$ for the motion and aggregation defects appearing in Theorem~\ref{thm:composite-main}. Substituting this finite sum into the bound of Theorem~\ref{thm:composite-main} absorbs it into $\bar\Delta_r$.
\end{proof}

\begin{proof}[Proof of Corollary~\ref{cor:structured-nonsmooth-rate}]
For finite defect accumulation, apply Corollary~\ref{cor:finite-transition-composite-rate}. For verified local selection, Lemma~\ref{comp:lem:verified-selection-absorption} gives $\sum_t\eta_t^r<\infty$ and $\zeta_t^r=0$ for $t\ge T_{\rm id}$, while Corollary~\ref{cor:verified-selection-lower-bound} gives the Lyapunov lower bound. Lemmas~\ref{comp:lem:summability} and~\ref{comp:lem:consensus-by-steps} then bound the accumulated step errors and consensus error uniformly in the horizon. Sum Lemma~\ref{comp:lem:residual-domination} from $t=T_{\rm id}$ onward. For the preceding indices, $z_i^t,u_r^t\in K$ gives $\mathcal G_r(z^t)\le\operatorname{diam}K$. Thus, for every $T\ge1$,
\[
\begin{aligned}
\sum_{t=1}^{T}\mathcal G_r(z^t)^2
&\le (T_{\rm id}-1)(\operatorname{diam}K)^2
+C\sum_{t=0}^{T-1}(d_x^t+d_z^t+e^t)\\
&\quad+C\sum_{t=1}^{T}S^t\\
&\le\bar\Delta_r<\infty.
\end{aligned}
\]
Here $\bar\Delta_r$ is independent of $T$. Dividing by $T$ proves the squared best-iterate rate.
\end{proof}

\section{Proof of Stochastic Composite SDBPG Convergence}\label{app:stochastic-proof}
The stochastic proof keeps the oracle error $\delta_i^t=G_i^t-\nabla K_i(x_i^t;Y_i^{t+1},Z_i^t)$ in each step of the deterministic composite argument. This adds variance terms and the stochastic subgradient-motion defect.

\begin{lemma}[Stochastic baseline Lyapunov inequality]\label{stoch:lem:baseline}
Under Assumption~\ref{ass:stochastic-oracle},
\begin{equation}\label{stoch:eq:baseline}
\E\big[\mathcal L_r^{t+1}-\mathcal L_r^t\big]
\le
-a_x^s\E d_x^t-a_z\E d_z^t-a_e\E e^t
-(\tau-\alpha)\E S^t+q\E S^{t+1}
+\frac1{2\chi}\sum_i(\sigma_i^t)^2.
\end{equation}
\end{lemma}

\begin{proof}
Repeat the composite one-step descent with $G_i^t$ in the primal optimality condition. The descent inequality acquires the extra term $-\sum_i\langle\delta_i^t,x_i^{t+1}-x_i^t\rangle$. Young's inequality and the Bregman lower bound give
\[
-\langle\delta_i^t,x_i^{t+1}-x_i^t\rangle
\le \frac{\chi}{\mu}D_\phi(x_i^{t+1},x_i^t)+\frac1{2\chi}\|\delta_i^t\|^2.
\]
Conditioning on $\mathcal F_t$ and using $\E[\|\delta_i^t\|^2\mid\mathcal F_t]\le(\sigma_i^t)^2$ gives the conditional version of \eqref{stoch:eq:baseline}. Taking total expectation yields the displayed unconditional inequality.
\end{proof}

\begin{lemma}[Stochastic consensus controlled by step errors]\label{stoch:lem:consensus-by-steps}
There exist finite constants $\widehat\Gamma_{s,x}^+,\widehat\Gamma_{s,x}^-,\widehat\Gamma_{s,z},\widehat\Gamma_{s,\delta},\widehat\Gamma_{s,\eta}$ such that, for $t\ge1$,
\begin{equation}\label{stoch:eq:consensus-by-steps}
\E S^t\le\frac1{\alpha^2}\left(\widehat\Gamma_{s,x}^+\E d_x^t+\widehat\Gamma_{s,x}^-\E d_x^{t-1}+\widehat\Gamma_{s,z}\E d_z^{t-1}+\widehat\Gamma_{s,\delta}\sum_i((\sigma_i^t)^2+(\sigma_i^{t-1})^2)+\widehat\Gamma_{s,\eta}\eta_t^s\right).
\end{equation}
\end{lemma}

\begin{proof}
The stochastic optimality system is
\[
g^t+\delta^t+\mathcal B y^{t+1}+p^t+\xi^{t+1}=0.
\]
Subtract the relation at time $t-1$ and use the graph range inequality. The deterministic $g$ and $p$ terms are controlled as before; the selected-subgradient difference is split into absorbed step, oracle, and defect parts by Definition~\ref{def:subgradient-motion-defects}(c), and the oracle errors are controlled by Assumption~\ref{ass:stochastic-oracle}.
\end{proof}

\begin{lemma}[Stochastic lower boundedness]\label{stoch:lem:lower-bound}
Under Assumption~\ref{ass:lyapunov-lower-bound}, the expected stochastic Lyapunov sequence is bounded below.
\end{lemma}

\begin{proof}
This is Assumption~\ref{ass:lyapunov-lower-bound} for the stochastic theorem.
\end{proof}

\begin{lemma}[Stochastic summability]\label{stoch:lem:summability}
Under \eqref{eq:stochastic-admissible-region}, there exist finite constants $M_s,C_s,C_{s,\eta}$ such that
\begin{equation}\label{stoch:eq:summable}
\sum_{t=0}^{T-1}\E(d_x^t+d_z^t+e^t)
\le M_s+C_s\sum_{t=0}^{T}\sum_i(\sigma_i^t)^2+C_{s,\eta}\sum_{t=1}^{T}\eta_t^s.
\end{equation}
\end{lemma}

\begin{proof}
Sum \eqref{stoch:eq:baseline}, use Lemma~\ref{stoch:lem:consensus-by-steps} to control $\sum_t\E S^{t+1}$, and use the expected lower bound from Lemma~\ref{stoch:lem:lower-bound}. The admissible inequalities make the coefficients of the accumulated step errors positive.
\end{proof}

\begin{lemma}[Stochastic lifted residual domination]\label{stoch:lem:lifted-residual-domination}
There exists $H_s<\infty$ such that
\begin{equation}\label{stoch:eq:lifted-residual-domination}
\E\mathcal R_{\mathrm{lift},r}(x^{t+1},y^{t+1})^2
\le H_s\E(d_x^t+e^{t-1}+S^t+S^{t+1})+H_s\sum_i(\sigma_i^t)^2.
\end{equation}
\end{lemma}

\begin{proof}
The stochastic optimality system gives
\[
g^t+\delta^t+\mathcal B y^{t+1}+p^t+\xi^{t+1}=0.
\]
Thus the lifted stationarity distance is bounded by $\|\nabla f(x^{t+1})-g^t-p^t-\delta^t\|^2$. The deterministic terms are controlled as in Lemma~\ref{comp:lem:lifted-residual-domination}; the final term is controlled by the variance bound.
\end{proof}

\begin{lemma}[Stochastic proximal residual domination]\label{stoch:lem:residual-domination}
Assume $c_0+\rho\bar d>L_-$. For the stochastic aggregation defect in Definition~\ref{def:aggregation-defects}(b), there exist $H_{s,r},C_{s,\zeta}<\infty$ such that
\begin{equation}\label{stoch:eq:residual-domination}
\E\mathcal G_r(z^t)^2
\le H_{s,r}\E(e^{t-1}+d_x^{t-1}+d_z^{t-1}+S^t)+H_{s,r}\sum_i(\sigma_i^{t-1})^2+C_{s,\zeta}\zeta_t^s.
\end{equation}
\end{lemma}

\begin{proof}
Repeat Lemma~\ref{comp:lem:residual-domination}. The summed stochastic optimality relation at time $t-1$ contains the extra term $\sum_i\delta_i^{t-1}$, whose squared norm is bounded in expectation by $N\sum_i(\sigma_i^{t-1})^2$.
\end{proof}

\begin{lemma}[Stochastic verified selection absorbs expected defects]\label{stoch:lem:verified-selection-absorption}
For a stochastic verified local selection in the sense of Definition~\ref{def:stochastic-verified-local-selection}, the stochastic motion defects can be absorbed up to a finite expected prefix, and the aggregation defect of Definition~\ref{def:aggregation-defects}(b) is summable. Consequently, for all $T\ge1$,
\[
\sum_{t=1}^{T}\eta_t^s\le B_{\eta,s},
\qquad
\sum_{t=1}^{T}\zeta_t^s\le B_{\zeta,s},
\]
for finite constants $B_{\eta,s},B_{\zeta,s}$ independent of $T$. In particular $\sum_t\eta_t^s<\infty$ and $\sum_t\zeta_t^s<\infty$, i.e.\ the selection condition in Definition~\ref{def:stochastic-verified-local-selection} implies Assumption~\ref{ass:finite-defect-accumulation}(b).
\end{lemma}

\begin{proof}
Let $L_\Xi$ be the common tail Lipschitz constant from Definition~\ref{def:stochastic-verified-local-selection}. On the event $\{t\ge T_{\rm id}^s\}$,
\[
\sum_i\|\xi_i^{t+1}-\xi_i^t\|^2
\le L_\Xi^2\sum_i\|x_i^{t+1}-x_i^t\|^2
\le \frac{2L_\Xi^2}{\mu}d_x^t.
\]
Choose $M_{T\xi}^s\ge2L_\Xi^2/\mu$ in Definition~\ref{def:subgradient-motion-defects}(c). Since the oracle-error term in that definition is subtracted with a nonnegative coefficient, the remaining positive part can only come from the pre-identification event. The integrable prefix condition in Definition~\ref{def:stochastic-verified-local-selection} gives a constant $B_{\eta,s}<\infty$ with $\sum_{t=1}^T\eta_t^s\le B_{\eta,s}$.

For aggregation, choose $\bar\xi_i^t=\Xi_i(\bar x^t)$ on the identified tail. Then, on $\{t\ge T_{\rm id}^s\}$,
\[
\begin{aligned}
\operatorname{dist}^2\!\left(\sum_i\xi_i^t,\sum_i\partial r_i(\bar x^t)\right)
&\le
\left\|\sum_i(\Xi_i(x_i^t)-\Xi_i(\bar x^t))\right\|^2\\
&\le C\sum_i\|x_i^t-\bar x^t\|^2
\le C S^t .
\end{aligned}
\]
Taking expectations and assigning the pre-identification contribution to $\zeta_t^s$ gives the bound in Definition~\ref{def:aggregation-defects}(b) with $\sum_{t=1}^T\zeta_t^s\le B_{\zeta,s}<\infty$ by the same integrable prefix condition.
\end{proof}

\begin{proof}[Proof of Theorem~\ref{thm:stochastic-main}]
Combine Lemmas~\ref{stoch:lem:summability}, \ref{stoch:lem:consensus-by-steps}, and~\ref{stoch:lem:lifted-residual-domination}. The domination lemma controls the residuals from index $k=2$ onward, and the finite expected residual at $k=1$ is absorbed into $\Delta_s$. Dividing by $T$ and using the bound by the average residual proves the theorem. The stated consequences follow by imposing summable variance and defects, or horizon schedules with average variance and defect of order $O(\epsilon^2)$.
\end{proof}

\begin{proof}[Proof of Corollary~\ref{cor:stochastic-asymptotic}]
Under summable variance and defects, the proof of Theorem~\ref{thm:stochastic-main} yields summability of $\E\mathcal R_{\mathrm{lift},r}(x^t,y^t)^2$. Nonnegativity gives convergence to zero.
\end{proof}

\begin{proof}[Proof of Corollary~\ref{cor:stochastic-proximal}]
Use Lemma~\ref{stoch:lem:residual-domination} instead of Lemma~\ref{stoch:lem:lifted-residual-domination} in the proof of Theorem~\ref{thm:stochastic-main}.
\end{proof}

\begin{proof}[Proof of Corollary~\ref{cor:stochastic-finite-accumulation}]
Assumption~\ref{ass:finite-defect-accumulation}(b) states exactly $\sum_t\eta_t^s<\infty$ and $\sum_t\zeta_t^s<\infty$ for the expected motion and aggregation defects appearing in the bound of Corollary~\ref{cor:stochastic-proximal}. Substituting these finite sums into that bound absorbs the term $C'_{s,r}\sum_{t=1}^T(\eta_t^s+\zeta_t^s)$ into $\bar\Delta_{s,r}$, leaving only the accumulated variance term explicit. Summability of the variance then yields the $O(T^{-1})$ order.
\end{proof}

\begin{proof}[Proof of Corollary~\ref{cor:stochastic-verified-proximal}]
Corollary~\ref{cor:stochastic-verified-selection-lower-bound} supplies the expected Lyapunov lower bound, so Theorem~\ref{thm:stochastic-main} applies. Lemma~\ref{stoch:lem:verified-selection-absorption} shows that the selection condition in Definition~\ref{def:stochastic-verified-local-selection} implies Assumption~\ref{ass:finite-defect-accumulation}(b) and makes the aggregation defect of Definition~\ref{def:aggregation-defects}(b) summable. Applying Corollary~\ref{cor:stochastic-finite-accumulation} gives the displayed bound with only the accumulated variance term left explicit.
\end{proof}

\begin{proof}[Proof of Corollary~\ref{cor:structured-stochastic-rate}]
For stochastic verified selection, apply Corollary~\ref{cor:stochastic-verified-proximal}. In the identified $\ell_1$ or group $\ell_{2,1}$ case, Proposition~\ref{prop:stability-sufficient}(c) gives finite expected defect accumulation. The selected subgradients of these penalties are uniformly bounded, so Proposition~\ref{prop:lyapunov-lower-sufficient} gives the expected Lyapunov lower bound under the uniform variance bound. Corollary~\ref{cor:stochastic-finite-accumulation} then gives the same variance-dependent estimate. If the cumulative variance is finite, the numerator is bounded uniformly in $T$, proving the $O(T^{-1})$ best-iterate order.
\end{proof}

\end{document}